\documentclass[reqno, 11pt]{amsart}
\usepackage[margin=1in]{geometry}
\usepackage{amsmath, amssymb, amsthm,bm}
\usepackage[shortlabels]{enumitem}
\usepackage{url}
\usepackage[hidelinks]{hyperref}
\usepackage[capitalize]{cleveref}
\usepackage{amsaddr}

\theoremstyle{plain}
\newtheorem{thm}{Theorem}
\newtheorem{lemma}[thm]{Lemma}

\theoremstyle{definition}
\newtheorem{defn}[thm]{Definition}

\theoremstyle{remark}
\newtheorem{rem}[thm]{Remark}

\numberwithin{thm}{section}

\numberwithin{equation}{section}

\newcommand{\abs}[1]{\left\lvert#1\right\rvert}
\newcommand{\norm}[1]{\left\lVert#1\right\rVert}

\newcommand{\ang}[1]{\left\langle #1 \right\rangle}
\newcommand{\floor}[1]{\left\lfloor #1 \right\rfloor}
\newcommand{\ceil}[1]{\left\lceil #1 \right\rceil}
\newcommand{\paren}[1]{\left( #1 \right)}
\newcommand{\sqb}[1]{\left[ #1 \right]}
\newcommand{\set}[1]{\left\{ #1 \right\}}

\newcommand{\ol}{\overline}

\newcommand{\cH}{\mathcal H}
\newcommand{\cI}{\mathcal I}
\newcommand{\cL}{\mathcal L}
\newcommand{\cP}{\mathcal P}
\newcommand{\cQ}{\mathcal Q}
\newcommand{\cS}{\mathcal S}

\newcommand{\fB}{\mathfrak B}
\newcommand{\fU}{\mathfrak U}

\newcommand{\bx}{\bm x}

\newcommand{\FF}{\mathbb{F}}

\newcommand{\C}{\mathbb C}
\newcommand{\E}{\mathbb E}
\newcommand{\F}{\mathbb F}

\newcommand{\R}{\mathbb R}
\newcommand{\Z}{\mathbb Z}
\newcommand{\Zp}{\mathbb Z_{>0}}
\newcommand{\Znn}{\mathbb Z_{\geqslant 0}}
\newcommand{\U}{\mathbb U}
\newcommand{\Fpx}{\mathbb F_p^{\times}}
\newcommand{\Fpz}{(\mathbb F_p\setminus\{0\})}

\DeclareMathOperator*{\spn}{span}

\DeclareMathOperator{\rank}{rank}
\DeclareMathOperator{\Id}{Id}
\DeclareMathOperator{\fnz}{fnz}

\title{Testing linear-invariant properties}
\author[{Jonathan Tidor \and Yufei Zhao}]{Jonathan Tidor \and Yufei Zhao}
\address{Massachusetts Institute of Technology, Cambridge, MA 02139, USA}
\email{\{jtidor,yufeiz\}@mit.edu}
\thanks{Tidor was supported by NSF Graduate Research Fellowship Program DGE-1122374. Zhao was supported by NSF Award DMS-1764176, the MIT Solomon Buchsbaum Fund, and a Sloan Research Fellowship.}
\date{}

\begin{document}

\begin{abstract}
Fix a prime $p$ and a positive integer $R$. We study the property testing of functions $\mathbb F_p^n\to[R]$. We say that a property is testable if there exists an oblivious tester for this property with one-sided error and constant query complexity. Furthermore, a property is proximity oblivious-testable (PO-testable) if the test is also independent of the proximity parameter $\epsilon$. It is known that a number of natural properties such as linearity and being a low degree polynomial are PO-testable. 
These properties are examples of linear-invariant properties, meaning that they are preserved under linear automorphisms of the domain. Following work of Kaufman and Sudan, the study of linear-invariant properties has been an important problem in arithmetic property testing.

A central conjecture in this field, proposed by Bhattacharyya, Grigorescu, and Shapira, is that a linear-invariant property is testable if and only if it is semi subspace-hereditary. We prove two results, the first resolves this conjecture and the second classifies PO-testable properties.
\begin{enumerate}
    \item A linear-invariant property is testable if and only if it is semi subspace-hereditary.
    \item A linear-invariant property is PO-testable if and only if it is locally characterized.
\end{enumerate}

Our innovations are two-fold. We give a more powerful version of the compactness argument first introduced by Alon and Shapira. This relies on a new strong arithmetic regularity lemma in which one mixes different levels of Gowers uniformity. This allows us to extend the work of Bhattacharyya, Fischer, Hatami, Hatami, and Lovett by removing the bounded complexity restriction in their work. Our second innovation is a novel recoloring technique called patching. This Ramsey-theoretic technique is critical for working in the linear-invariant setting and allows us to remove the translation-invariant restriction present in previous work.
\end{abstract}

\maketitle

\section{Introduction}
\label{sec-intro}

In property testing, the aim is to find randomized algorithms that distinguish objects that have some given property from those that are far from satisfying the property by querying the given large object at a small number of locations. Property testing emerged from the linearity test of Blum, Luby, and Rubinfeld~\cite{BLR93}, and was formally defined and systematically studied by Rubinfeld and Sudan~\cite{RS07} and Goldreich, Goldwasser, and Ron \cite{GGR98}. There have been important developments especially in the following two settings: graph property testing and arithmetic property testing.

Two representative problems are: (1) given a large graph, test whether the graph is triangle-free or $\epsilon$-far from triangle-free (an $n$-vertex graph is \emph{$\epsilon$-far} from a graph property if one needs to add and/or remove more than $\epsilon n^2$ edges in order to satisfy the property), and (2) given a function $f \colon \FF_p^n \to \FF_p$, test whether $f$ is linear or $\epsilon$-far from linear (for an arithmetic property, being \emph{$\epsilon$-far} means that one needs to change the value of the function on more than an $\epsilon$-fraction of the domain in order to satisfy the property). In both cases, it is known that one can achieve the desired goal by sampling a fixed number of entries repeatedly $C(\epsilon)$ times. For testing whether a graph is triangle-free \cite{RS78}, one samples a uniformly random triple of vertices and checks whether they form a triangle, and for testing linearity~\cite{BLR93}, one samples $x,y \in \FF_p^n$ uniformly and checks if $f(x)+f(y) = f(x+y)$.

In this paper we give a property testing algorithm for a very general class of arithmetic properties. The goal is to determine whether a function $f\colon \FF_p^n \to [R]:=\{1,\ldots,R\}$ (with fixed prime $p$ and positive integer $R$) satisfies some given property or is $\epsilon$-far from satisfying the property.
All the properties we consider are \emph{linear-invariant} in the sense that they are invariant under automorphisms of the vector space $\FF_p^n$. Linear-invariant properties form an important general class of arithmetic properties, e.g., the work of Kaufman and Sudan~\cite{KS08} ``highlights linear-invariance as a central theme in algebraic property testing.''

We say that a property $\cP$ is \emph{testable} if there exists an oblivious tester with one-sided error (and constant query complexity) for the property. A tester for $\cP$ produces a positive integer $d=d(\epsilon)$ and an oracle provides the tester with the restriction $f|_U$ where $U$ is a uniform random $d$-dimensional linear subspace of the domain (if the domain is large enough that such a subspace exists; if the domain has dimension strictly less than $d$, the oracle provides the tester with all of $f$). We require our tester accepts functions $f$ satisfying $\cP$ with probability 1 and reject functions that are $\epsilon$-far from satisfying $\cP$ with probability at least $\delta=\delta(\epsilon)$ for some function $\delta\colon(0,1)\to(0,1)$. Furthermore, we say that $\cP$ is \emph{proximity oblivious-testable} (PO-testable) if $d=d(\epsilon)$ is a constant independent of $\epsilon$. The idea of PO-testability was introduced by Goldreich and Ron \cite{GR11} who, among other results, classified the PO-testable graph properties.

One surprising feature of property testing is that many natural properties, such as linearity, are testable and even PO-testable. A key feature of linearity is that it is subspace-hereditary meaning that if $f\colon\F_p^n\to\F_p$ is linear, then the same is true for $f|_U$ for every linear subspace $U\leq\F_p^n$. To be precise, we say that a linear-invariant property $\cP$ is \emph{subspace-hereditary} if for every $f\colon\F_p^n\to[R]$ satisfying $\cP$ and every linear subspace $U\leq\F_p^n$, the restriction $f|_U$ also satisfies $\cP$.

A central conjecture in this field, by Bhattacharyya, Grigorescu, and Shapira, is that all linear-invariant, subspace-hereditary properties are testable \cite[Conjecture 4]{BGS15}. In fact, they conjecture that the slightly larger class of \emph{semi subspace-hereditary} properties are testable and prove that no other properties can be tested.

\begin{defn}
A linear-invariant property $\cP$ is \emph{semi subspace-hereditary} if there exists a subspace-hereditary property $\cQ$ such that
\begin{enumerate}[(i)]
\item every function satisfying $\cP$ also satisfies $\cQ$;
\item for all $\epsilon>0$, there exists $N(\epsilon)$ such that if $f\colon\F_p^n\to[R]$ satisfies $\cQ$ and is $\epsilon$-far from satisfying $\cP$, then $n< N(\epsilon)$.
\end{enumerate}
\end{defn}

It is known that there are subspace-hereditary properties where the dimension $d$ sampled must grow as the proximity parameter $\epsilon$ approaches 0. To be PO-testable, a property must satisfy the following more restrictive condition.

\begin{defn}
A linear-invariant property $\cP$ is \emph{locally characterized} if there exists some $d$ such that the following holds. For every $f\colon\F_p^n\to[R]$ with $n\geq d$, the function $f$ satisfies $\cP$ if and only if $f|_U$ satisfies $\cP$ for every $U\leq\F_p^n$ of dimension $d$.
\end{defn}

Our first result is a resolution of the conjecture of Bhattacharyya, Grigorescu, and Shapira, classifying the testable linear-invariant properties. Our second result is a classification of the PO-testable linear-invariant properties. 

\begin{thm}
\label{main-thm}
A linear-invariant property is testable if and only if it is semi subspace-hereditary.
\end{thm}

\begin{thm}
\label{main-thm-PO}
A linear-invariant property is PO-testable if and only if it is locally characterized.
\end{thm}

\begin{rem}
Note that under our definition, the tester does not know the dimension of the domain. This rules out some ``unnatural'' properties such as those properties that behave differently depending on whether the dimension of the domain is even or odd.
\end{rem}

Previous work in arithmetic property testing has focused on a number of special cases including monotone properties \cite{KSV12, Sha10}, ``complexity 1'' properties over $\FF_2$ \cite{BGS15}, and bounded complexity translation-invariant properties \cite{BFHHL13}.

We note that very little was previously known about general linear-invariant patterns. One simple way to define a class of linear-invariant patterns can be done by, for example, choosing an arbitrary subset of ``allowable'' maps $\FF_p^2 \to [R]$ and defining a property of functions $\FF_p^n \to [R]$ to consist of those whose restriction to every 2-dimensional linear subspace is allowable. Even this class of 2-dimensionally-defined patterns was not known to be testable in general prior to this work.

Our innovations are two-fold. We prove a strong arithmetic regularity lemma which, unlike previous arithmetic regularity lemmas, mixes different levels of Gowers uniformity. This allows us to give a more powerful version of the \emph{compactness argument} first introduced by Alon and Shapira \cite{AS08}. With this tool we can remove the \emph{bounded complexity} restriction that was present in all previous work.

Our second innovation is a novel recoloring technique we call \emph{patching}. This technique is critical for working in the linear-invariant setting and  allows us to handle an important obstacle encountered by previous works. Roughly speaking, this obstacle is the inability of regularity methods to regularize functions in a neighborhood of the origin. 

In the rest of this section we give a summary of the proof of the main theorem and its relation to previous work. 

\subsection{Graph removal lemmas and property testing}

We begin with an overview of graph removal lemmas and their proof techniques (see also the survey~\cite{CF13}).

The triangle removal lemma of Ruzsa and Szemer\'edi \cite{RS78} states that for all $\epsilon>0$ there exists $\delta>0$ such that any $n$-vertex graph with at most $\delta n^3$ triangles can be made triangle-free by removing $\epsilon n^2$ edges. This was generalized to the graph removal lemma, first stated explicitly by Alon, Duke, Lefmann, R\"odl, and Yuster \cite{ADLRY94} and by F\"uredi \cite{Fur95}.

A key tool for proving the graph removal lemma is a regularity lemma, namely Szemer\'edi's graph regularity lemma. Roughly speaking, the proof proceeds by using this regularity lemma to partition the input graph $G$ into a small number of structured components. Then we ``clean up'' $G$ by removing at most $\epsilon n^2$ edges. This is done in such a way that either the resulting graph is $H$-free or the original graph $G$ contains many copies of $H$.

An important extension of the graph removal lemma is the induced graph removal lemma, proved by Alon, Fischer, Krivelevich, and Szegedy \cite{AFKS00}.
The induced graph removal lemma states that for every graph $H$ (or finite collection $\cH$ of graphs), for all $\epsilon>0$ there exists $\delta>0$ such that every $n$-vertex graph with at most $\delta n^{v(H)}$ induced copies of $H$ can be made induced $H$-free by adding and/or removing at most $\epsilon n^2$ edges (here \emph{induced $H$-free} means not containing any induced subgraph isomorphic to $H$).

The original proof of the induced removal lemma relies on an extension of Szemer\'edi's graph regularity lemma known as the ``strong regularity lemma.'' Using such a regularity lemma combined with a random sampling argument, one can produce a ``regular model'', that is, a large induced subgraph $X:=G[U]$ (on a constant fraction of the vertices of $G$) that is very regular and approximates the original graph well in a certain sense. Then we ``clean up'' $G$ by adding and/or removing at most $\epsilon n^2$ edges in such a way that if the resulting graph is not induced $H$-free then $X$ (in the original graph) must contain many induced copies of $H$.

Alon and Shapira~\cite{AS08} extended the induced graph removal lemma to an infinite collection of graphs. Namely they prove that for a (possibly infinite) set $\cH$ of graphs and for $\epsilon>0$ there exist $\delta>0$ and $k$ such that the following holds: if $G$ is an $n$-vertex graph with at most $\delta n^{v(H)}$ copies of $H$ for all $H\in\cH$ with $k$ or fewer vertices, then $G$ can be made induced $\cH$-free by adding and/or removing at most $\epsilon n^2$ edges (meaning the modified graph has no induced subgraph isomorphic to $H$ for every $H\in\cH$). This theorem immediately implies (and is equivalent to) the fact that every hereditary graph property is testable with constant query-complexity and one-sided error.

This series of works, in addition to being important results in their own right, gives a framework for proving constant query-complexity property testing algorithms in other settings, given an appropriate regularity lemma. In particular, the hypergraph regularity lemma, proved by Gowers \cite{Gow07} and independently by R\"odl et al. \cite{RNSSK05} can be used with the above techniques to prove an infinite induced hypergraph removal lemma \cite{RS07}. Consequently, every hereditary hypergraph property is testable with constant query-complexity and one-sided error.

\subsection{Arithmetic analogs}

The problem of property testing for functions $f \colon \FF_p^n \to [R]$ has been intensely studied, starting with the the classic work of Blum, Luby, and Rubinfeld~\cite{BLR93} on linearity testing. Much of the work focuses on testing whether some function $f \colon \FF_p^n \to \FF_p$ has certain algebraic properties (e.g., a polynomial of some given type) \cite{AKKLR05,KS08}. 
There is also much interest in testing properties that do not arrive from algebraic characterizations.
Below we give an overview of the developments related to property testing in $\FF_p^n$ from a perspective that is parallel to the graph regularity method developments discussed earlier.

The first arithmetic regularity lemma was proved by Green~\cite{Gre05} using Fourier-analytic techniques, and it laid the groundwork for further developments of the regularity method in the arithmetic setting. These regularity lemmas has since found many applications in additive combinatorics and related fields. In particular, combined with the graph removal framework described above, Green's regularity lemma is suitable for proving an arithmetic removal lemma for ``complexity 1'' systems of linear forms (see \cref{sec-prelim} for the definition of complexity); e.g., see~\cite{BCSX11}.

Kr\'al', Sera, and Vena \cite{KSV12} and independently Shapira \cite{Sha10} bypass the need for an arithmetic regularity lemma and prove the full arithmetic removal lemma by a direct reduction from the hypergraph removal lemma. 
Their results imply that all linear-invariant, subspace-hereditary \emph{monotone} properties are testable with constant query-complexity and one-sided error. 
(A property of functions $\FF_p^n\to\{0,1\}$ is \emph{monotone} if changing 1's to 0's preserves the property.)

Note that the above result is an arithmetic removal lemma and not an \emph{induced} arithmetic removal lemma (hence the restriction to monotone properties). 
Due to the nature of the reduction, the techniques do not seem to be capable of deducing the induced arithmetic removal lemma from the induced hypergraph removal lemma. 

An alternative approach is to apply the strong graph regularity approach~\cite{AFKS00} of proving the induced graph removal lemma to Green's arithmetic regularity lemma.
However there is also a major obstacle to the approach, related to the fact that the origin plays a special role in a vector space while there is no corresponding feature of graphs. 
It turns out that it is not always possible to regularize the space in a neighborhood of the origin \cite{GS15}.

Bhattacharyya, Grigorescu, and Shapira~\cite{BGS15} managed to overcome this obstacle in the special case of vector spaces over $\F_2$. They follow the above strategy, implementing the strong regularity idea~\cite{AFKS00} in the style of Green's arithmetic regularity~\cite{Gre05} along with one additional tool, namely a Ramsey-theoretic result, to prove an infinite induced arithmetic removal lemma for ``complexity 1'' patterns over $\F_2$. Unfortunately, it is known \cite{GS15} that this Ramsey-theoretic result fails over all finite fields other than $\F_2$.

Bhattacharyya, Fischer, and Lovett~\cite{BFL12} managed to overcome this obstacle in a different special case, namely for translation-invariant patterns. When all patterns considered are translation-invariant, the origin no longer plays a special role and one can essentially ignore it while carrying out the strong regularity framework. In addition, \cite{BFL12} allows one to handle higher complexity patterns, which requires developing and applying tools from higher-order Fourier analysis.

Higher-order Fourier analysis plays a central role in modern additive combinatorics. These techniques were initiated by Gowers \cite{Gow01} in his celebrated new proof of Szemer\'edi's theorem, and further developed in a sequence of works by Green, Tao, and Ziegler~\cite{GT08, GT10, GTZ12} settling classical conjectures on the asymptotics of prime numbers patterns. A parallel theory of higher-order Fourier analysis was developed in finite field vector spaces by Bergelson, Tao, and Ziegler~\cite{BTZ10, TZ10, TZ12}, leading to an inverse Gowers theorem over finite fields vector spaces.

For applications to property testing, this line of work culminated in the work of Bhattacharyya, Fischer, Hatami, Hatami, and Lovett~\cite{BFHHL13} (extending \cite{BFL12}), who applied the inverse Gowers theorem over finite fields and developed further equidistribution tools to prove an infinite induced arithmetic removal lemma for all linear-invariant, subspace-hereditary properties that are also translation-invariant and bounded-complexity. Their work follows the strong regularity framework of \cite{AFKS00,AS08}. Our results improve upon this work by removing the translation-invariant and bounded-complexity restrictions.

In addition to their property testing algorithm, Bhattacharyya, Fischer, Hatami, Hatami, and Lovett~\cite{BFHHL13} proved that a large class of somewhat algebraically structured properties are indeed affine-invariant, subspace-hereditary, and locally characterized. These are the so-called ``degree-structural properties''. 
A simple extension of their result \cite[Theorem 16.3]{HHL19} implies that the larger class of ``homogeneous degree-structural properties'' are linear-invariant, linear subspace-hereditary (but not affine-invariant and not subspace-hereditary), and locally characterized, and thus these properties are testable by our main theorem. 
As an example, one can test whether a function $\FF_p^n \to \FF_p$ can be written as $A^2 + B^2$ where both $A$ and $B$ are homogeneous polynomials of some given degree $d$.

\subsection{Our contributions}

\subsubsection{Patching}

In this paper, building on the authors' earlier work with Fox~\cite{FTZ19} for complexity 1 patterns, we develop a new technique called ``patching'' that allows us to overcome the obstacle faced by earlier approaches, namely that a neighborhood of the origin cannot be regularized and fails certain Ramsey properties (unless working over $\F_2$). 
In essence, the patching result states that if there exists some map $f\colon \FF_p^n\to[R]$ that has low density of some colored patterns $\cH$ for $n$ large enough, then for all $m$ there must exist some map $g\colon \FF_p^{m}\to[R]$ that has no $\cH$-instances. 

\begin{thm}[Informal patching result]
For every set of colored pattern $\cH$, there exist $\epsilon_0>0$ and $n_0$ such that the following holds. Either:
\begin{itemize}
    \item for every $n$, there exists a function $f\colon\F_p^n\to[R]$ that is $\cH$-free; or
    \item for every function $f\colon\F_p^n\to[R]$ with $n\geq n_0$, the $H$-density in $f$ is at least $\epsilon_0$ for some $H\in\cH$.
\end{itemize}
\end{thm}

Our proof proceeds in two steps. 
First, as in \cite{BFHHL13}, following the strong regularity framework of~\cite{AFKS00} for proving induced graph removal lemmas, we apply a strong arithmetic regularity lemma, which produces a partition $\fB$ of $\FF_p^n$ and a ``regular model'' $X\subseteq\FF_p^n$ made up of a randomly sampled set of atoms from $\fB$. 
Unlike in the graph setting, we cannot ensure that the map $f\colon\FF_p^n\to[R]$ is very regular on every atom of $\fB|_X$. 
In particular, it may be impossible to guarantee that $f$ is regular on the atom containing the origin.
Instead we only ensure that almost every atom of $X$ is very regular. 
Unlike earlier proofs of removal lemmas, our ``recoloring algorithm'' has two components: for the regular atoms we ``clean up'' $f$ as usual, while for the irregular atoms we apply our patching result. Our patching result implies that there is some new global coloring $g\colon\FF_p^n\to[R]$ that avoids some appropriate set of colored patterns. 
To complete the proof we ``patch'' $f$ by replacing it by $g$ on all of the irregular atoms. If $f$ has low density of some set of colored pattern, then our argument shows that these pattern cannot appear in the recoloring, thereby completing the proof of the induced arithmetic removal lemma.

Our proof does not give effective bounds on the rejection probability function $\delta(\epsilon)$ guaranteed by \cref{main-thm}. The ineffectiveness is due to the fact that the current best-known bounds on the inverse theorem for non-classical polynomials are ineffective (the same occurs in \cite{BFHHL13}).

\subsubsection{Unbounded complexity}

The technique used to handle infinite removal lemmas is a compactness argument first introduced by Alon and Shapira~\cite{AS08} in the graph setting. A key ingredient of their proof is a strong regularity lemma.

Bhattacharyya, Fischer, Hatami, Hatami, and Lovett~\cite{BFHHL13} prove that all linear-invariant subspace-hereditary properties that are also translation-invariant and bounded-complexity are testable. Their result follows from an infinite removal lemma for arithmetic patterns of bounded complexity. The proof of this result involves a strong arithmetic regularity lemma and a compactness argument in the spirit of Alon and Shapira.

To remove the bounded complexity assumption from \cite{BFHHL13}, we prove a new strong arithmetic regularity lemma obtained by iterating a weaker arithmetic regularity lemma. The key innovation here is the level of Gowers uniformity used in each iteration is allowed to increase at each step of the process.


\section{Colored patterns and removal lemmas}
\label{sec-removal-lemma}

\cref{main-thm} and \cref{main-thm-PO} both follow from an arithmetic removal lemma for colored linear patterns. In this section we define these objects and state the main removal lemma.

\begin{defn}
A \textbf{linear form} over $\F_p$ in $\ell$ variables is an expression $L$ of the form \[L(x_1,\ldots,x_\ell)=\sum_{i=1}^\ell c_ix_i\] with $c_i\in\F_p$. For any $\F_p$-vector space $V$, the linear form $L$ gives rise to a function $L\colon V^\ell\to V$ that is linear in each variable.
\end{defn}

\begin{defn}
For a prime $p$ and a finite set $\cS$, an \textbf{$\cS$-colored pattern} over $\F_p$ consisting of $m$ linear forms in $\ell$ variables is a pair $(\bm L,\psi)$ given by a system $\bm L=(L_1,\ldots,L_m)$ of $m$ linear forms in $\ell$ variables and a coloring $\psi\colon [m]\to \cS$. Given a finite-dimensional $\F_p$-vector space $V$ and a function $f\colon V\to\cS$, an \textbf{$(\bm L,\psi)$-instance} in $f$ is some $\bx\in V^\ell$ such that $f(L_i(\bx))=\psi(i)$ for all $i\in [m]$. An instance is called \textbf{generic} if $x_1,\ldots,x_\ell$ are linearly independent. We say that $(\bm L,\psi)$ is \textbf{translation-invariant} if the coefficient of $x_1$ is 1 in each of $L_1,\ldots,L_m$. 
\end{defn}

Given a finite-dimensional $\F_p$-vector space $V$ and functions $f_1,\ldots,f_m\colon V\to[-1,1]$, we write \[\Lambda_{\bm L}(f_1,\ldots,f_m):=\E_{\bx\in V^k}[f_1(L_1(\bx))\cdots f_m(L_m(\bx))].\]

\begin{defn}
For an $\cS$-colored pattern over $\F_p$ consisting of $m$ linear forms in $k$ variables $(\bm L,\psi)$, a finite dimensional $\F_p$-vector space $V$, and a function $f\colon V\to\cS$, define the \textbf{$(\bm L,\psi)$-density} in $f$ to be $\Lambda_{\bm L}(f_1,\ldots,f_m)$ where $f_i:=1_{f^{-1}(\psi(i))}$ for each $i \in [m]$.
\end{defn}

Our main removal lemma is the following result.

\begin{thm}[Main removal lemma]
\label{removal-lemma}
Fix a prime $p$ and a finite set $\cS$. Let $\cH$ be a (possibly infinite) set of $\cS$-colored patterns over $\F_p$. For every $\epsilon>0$, there exists a finite set $\cH_{\epsilon}\subseteq\cH$ and $\delta=\delta(\epsilon,\cH)>0$ such that the following holds. Let $V$ be a finite-dimensional $\F_p$-vector space. If $f\colon V\to\cS$ has $H$-density at most $\delta$ for every $H\in \cH_{\epsilon}$, then there exists a recoloring $g\colon V\to\cS$ that agrees with $f$ on all but an at most $\epsilon$-fraction of $V$ such that $g$ has no generic $H$-instances for every $H\in\cH$.
\end{thm}

There are several difficulties in the proof of the main removal lemma. The first is that individual patterns $H\in\cH$ may have ``infinite complexity''. Second, the set of patterns $\cH$ may be infinite. Complicating this, even if all patterns in $\cH$ have finite complexity, these complexities can be unbounded. Finally, there are major difficulties related to the fact that the patterns in $\cH$ are not necessarily translation-invariant.

We use a trick called ``projectivization'' to reduce to the case where all patterns have finite complexity. To do this, we need a slightly modified version of the main removal lemma that we call the projective removal lemma (\cref{removal-lemma-proj}).

A ``compactness argument'' due to Alon and Shapira \cite{AS08} reduces the problem of an infinite collection of patterns to a finite one at the expense of requiring a stronger arithmetic regularity lemma. If the collection of patterns is all complexity at most $d$, we only require a strong $U^{d+1}$-regularity lemma with rapidly decreasing error parameter. In the most general case when the collection of patterns has unbounded complexity we require an even stronger regularity lemma where the error parameter rapidly decreases and the degree of the uniformity norm rapidly increases.

Unless we restrict to the special case where all patterns in $\cH$ are translation-invariant, the origin of the vector space plays a special role. This is unfortunate because it is impossible to regularize a function in the neighborhood of the origin. Since regularity methods are useless here, we turn to a new technique called patching, originally introduced by the authors and Fox \cite{FTZ19}, to deal with the portions of the vector space that cannot be regularized.

\begin{defn}
Let $\cS$ be a finite set equipped with a group action of $\Fpx$ that we denote $c\cdot s$ for $c\in\Fpx$ and $s\in\cS$. Given a finite-dimensional $\F_p$-vector space $V$, a function $f\colon V\to \cS$ is \textbf{projective} if it preserves the action of $\Fpx$, i.e., $f(cx)=c\cdot f(x)$ for all $c\in\Fpx$ and all $x\in V$.
\end{defn}

\begin{defn}
\label{finite-complexity}
A list of linear forms $\bm L=(L_1,\ldots,L_m)$ is \textbf{finite complexity} if no form is identically equal to zero, i.e., $L_i\not\equiv 0$ for all $i\in[m]$, and no two forms are linearly dependent, i.e., $L_i\not\equiv cL_j$ for all $i\neq j$ and $c\in\F_p$.
\end{defn}

\begin{defn}
\label{L}
Fix a prime $p$ and a positive integer $\ell$. We consider two particular systems of linear forms. For $\bm i=(i_1,\ldots,i_{\ell})\in\F_p^{\ell}$, define
\[L^{\ell}_{\bm i}(x_1,\ldots,x_{\ell}):=i_1x_1+\cdots+i_{\ell}x_{\ell}.\]
Then define $\bm L^{\ell}:=(L^{\ell}_{\bm i})_{\bm i\in\F_p^{\ell}}$,
the system of $p^{\ell}$ linear forms in $\ell$ variables that defines an $\ell$-dimensional subspace.

Let $E_\ell\subset\F_p^{\ell}$ be the set of non-zero vectors whose first non-zero coordinate is 1. Then define $\ol{\bm L}^{\ell}:=(L^{\ell}_{\bm i})_{\bm i\in E_\ell}$, a system of $(p^{\ell}-1)/(p-1)$ linear forms in $\ell$ variables.
\end{defn}

Note that unlike $\bm L^{\ell}$, the system $\ol{\bm L}^{\ell}$ has finite complexity. For technical reasons, it will be convenient to reduce the removal lemma for general patterns to the case where all patterns are defined by a system of the form $\bm L^{\ell}$. Then we reduce this to the following projective removal lemma where all patterns are defined by a system of the form $\ol{\bm L}^{\ell}$.

\begin{thm}[Projective removal lemma]
\label{removal-lemma-proj}
Fix a prime $p$ and a finite set $\cS$ equipped with an $\Fpx$-action. Let $\cH$ be a (possibly infinite) set consisting of $\cS$-colored patterns over $\F_p$ of the form $(\ol{\bm L}^{\ell},\psi)$ where $\ell$ is some positive integer and $\psi\colon E_\ell\to\cS$ is some map (see \cref{L} for the definition of $\ol{\bm L}^{\ell}$ and $E_\ell$). For every $\epsilon>0$, there exists a finite subset $\cH_{\epsilon}\subseteq\cH$ and $\delta=\delta(\epsilon,\cH)>0$ such that the following holds. Let $V$ be a finite-dimensional $\F_p$-vector space. If $f\colon V\to\cS$ is a projective function with $H$-density at most $\delta$ for every $H\in \cH_{\epsilon}$, then there exists a projective recoloring $g\colon V\to\cS$ that agrees with $f$ on all but an at most $\epsilon$-fraction of $V$ such that $g$ has no generic $H$-instances for every $H\in\cH$.
\end{thm}


\section{Preliminaries on higher-order Fourier analysis}
\label{sec-prelim}

\subsection{Gowers norms and complexity}

\begin{defn}
Fix a prime $p$, a finite-dimensional $\F_p$-vector space $V$, and an abelian group $G$. Given a function $f\colon V\to G$ and a shift $h\in V$, define the \textbf{additive derivative} $D_hf\colon V\to G$ by \[(D_hf)(x):=f(x+h)-f(x).\]
Given a function $f\colon V\to\C$ and a shift $h\in V$, define the \textbf{multiplicative derivative} $\Delta_hf\colon V\to\C$ by \[(\Delta_hf)(x):=f(x+h)\overline{f(x)}.\]
\end{defn}

\begin{defn}
Fix a prime $p$ and a finite-dimensional $\F_p$-vector space $V$. Given a function $f\colon V\to\C$ and $d\geq 1$, the {\textbf{Gowers uniformity norm}} $\|f\|_{U^d}$ is defined by \[\|f\|_{U^d}:=\abs{\E_{x,h_1,\ldots,h_d\in V}(\Delta_{h_1}\cdots\Delta_{h_d}f)(x)}^{1/2^d}.\]
\end{defn}

\begin{defn}
\label{complexity}
A system $\bm L=(L_1,\ldots,L_m)$ of $m$ linear forms in $\ell$ variables is \textbf{complexity at most $d$} if for all $\epsilon>0$ there exists $\delta>0$ such that for all $f_1,\ldots,f_\ell\colon V\to[-1,1]$ it holds that
\begin{equation}
\label{counting-lemma}
|\Lambda_{\bm L}(f_1,\ldots,f_\ell)|\leq\epsilon\qquad\text{whenever}\qquad\min_{1\leq i\leq \ell}\|f_i\|_{U^{d+1}}\leq\delta.
\end{equation}
The \textbf{complexity} of $\bm L$ is the smallest $d$ such that the above holds, and infinite otherwise.
\end{defn}

\begin{rem}
\label{rem-complexity}
The above definition is sometimes known as true complexity. It is known that a pattern $(L_1,\ldots,L_m)$ is complexity at most $d$ if and only if $L_1^{d+1},\ldots, L_m^{d+1}$ are linearly independent as $(d+1)$th order tensors~\cite{GW11, HHL16}.

Let $(L_1,\ldots,L_m)$ be any pattern such that no form is identically zero and no two forms are linearly dependent. It is known (for example, because Cauchy-Schwarz complexity is an upper bound for true complexity~\cite{GT10}) that $(L_1,\ldots,L_m)$ has complexity at most $m-2$.

It follows from the above discussion that the definition of complexity given in \cref{complexity} agrees with the definition of finite complexity given in \cref{finite-complexity}.
\end{rem}

\subsection{Non-classical polynomials and homogeneity}

For ease of notation we write
\begin{equation}
\label{U}
\U_k:=\tfrac1{p^k}\Z/\Z\subset\R/\Z
\end{equation}
through the paper.

\begin{defn}
Fix a prime $p$, and a non-negative integer $d\geq0$. Let $V$ be a finite-dimensional $\F_p$-vector space. A \textbf{non-classical polynomial} of degree at most $d$ is a map $P\colon V\to\R/\Z$ that satisfies \[(D_{h_1}\cdots D_{h_{d+1}}P)(x)= 0\] for all $h_1,\ldots,h_{d+1},x\in V$. The \textbf{degree} of $P$ is the smallest $d>0$ such that the above holds. The \textbf{depth} of $P$ is the smallest $k\geq 0$ such that $P$ takes values in a coset of $\U_{k+1}$.
\end{defn}

See \cite[Lemma 1.7]{TZ12} for some basic facts about non-classical polynomials. We record one such fact here.

\begin{lemma}[{\cite[Lemma 1.7(iii)]{TZ12}}]
\label{poly-exists}
Fix a prime $p$, and a finite-dimensional $\F_p$-vector space $V\simeq\F_p^n$. Then $P\colon V\to\R/\Z$ is a non-classical polynomial of degree at most $d$ if and only if it can be expressed in the form
\[P(x_1,\ldots,x_n)=\alpha+\sum_{\genfrac{}{}{0pt}{}{0\leq i_1,\ldots,i_n<p,\, j\geq 0:}{0<i_1+\cdots+i_n\leq d-k(p-1)}}\frac{c_{i_1,\ldots,i_n,k}|x_1|^{i_1}\cdots|x_n|^{i_n}}{p^{k+1}}\pmod 1,\]
for some $\alpha\in\R/\Z$ and coefficients $c_{i_1,\ldots,i_n,k}\in\{0,\ldots,p-1\}$ and where $|\cdot|$ is the standard map $\F_p\to\{0,\ldots,p-1\}$. Furthermore, this representation is unique.
\end{lemma}

As a corollary we see that in characteristic $p$, every non-classical polynomial of degree at most $d$ has depth at most $\floor{(d-1)/(p-1)}$.

\begin{defn}
A \textbf{homogeneous non-classical polynomial} is a non-classical polynomial $P\colon V\to\R/\Z$ that also satisfies the following. For all $b\in\F_p$ there exists $\sigma^{(P)}_b\in\Z/p^{k+1}\Z$ such that $P(bx)=\sigma_b^{(P)}P(x)$ for all $x\in V$.
\end{defn}

\begin{lemma}[{\cite[Lemma 3.3]{HHL16}}]
\label{sigma}
Fix a prime $p$ and integers $d>0$ and $k\geq 0$ satisfying $k\leq\floor{(d-1)/(p-1)}$. For each $b\in\F_p$ there exists $\sigma_{b}^{(d,k)}\in\Z/p^{k+1}\Z$ such that $\sigma_{b}^{(P)}=\sigma_b^{(d,k)}$ for all homogeneous non-classical polynomials $P$ of degree $d$ and depth $k$. Furthermore, for $b\neq0$, $\sigma_b^{(d,k)}$ is uniquely determined by the following two properties:
\begin{enumerate}[(i)]
    \item $\sigma_{b}^{(d,k)}\equiv b^d\pmod p$
    \item $\paren{\sigma_b^{(d,k)}}^{p-1}=1$ (in $\Z/p^{k+1}\Z$).
\end{enumerate}
\end{lemma}

\begin{thm}[{\cite[Theorem 3.4]{HHL16}}]
\label{homogeneous}
Let $P$ be a non-classical polynomial of degree $d$ and depth $k$. Then $P$ can be written as the sum of homogeneous non-classical polynomials of degree at most $d$ and depth at most $k$. 
\end{thm}

\subsection{Polynomial factors}

\begin{defn}
Fix a prime $p$. Define \[D_p:=\set{(d,k)\in\Zp\times\Znn : k\leq\floor{(d-1)/(p-1)}},\]and\[\cI_p:=\set{I\in\Znn^{D_p}:\sum_{(d,k)\in D_p}I_{d,k}<\infty}.\] We call $I\in\cI_p$ a \textbf{parameter list}. For $I\in\cI_p$, we write $\|I\|:=p^{\sum_{d,k}(k+1)I_{d,k}}$ and $\deg I$ for the largest $d$ such that $I_{d,k}\neq 0$ for some $k$. We add and subtract parameter list coordinatewise. For $I,I'\in\cI_p$, we write $I\leq I'$ if $I_{d,k}\leq I'_{d,k}$ for all $(d,k)\in D_p$.
\end{defn}

\begin{defn}
For $p$ a prime and $I\in\cI_p$, define the \textbf{atom-indexing set} of $I$ to be
\begin{equation}
\label{A}
A_{I}:=\prod_{(d,k)\in D_p}\paren{\tfrac{1}{p^{k+1}}\Z/\Z}^{I_{d,k}}.
\end{equation}
(Note that $|A_I|=\|I\|$.)
\end{defn}

For $I,I'\in\cI_p$ with $I\leq I'$, write $\pi\colon A_{I'}\to A_I$ for the standard projection map defined by 
\begin{equation}
\label{pi}
\pi\paren{\paren{a^i_{d,k}}_{\genfrac{}{}{0pt}{}{(d,k)\in D_p}{i\in\sqb{I'_{d,k}}}}}\mapsto\paren{a^i_{d,k}}_{\genfrac{}{}{0pt}{}{(d,k)\in D_p}{i\in\sqb{I_{d,k}}}}.
\end{equation}

$A_I$ is equipped with the following $\Fpx$-action:
\begin{equation}
\label{A-action}
c\cdot \paren{a^i_{d,k}}_{\genfrac{}{}{0pt}{}{(d,k)\in D_p}{i\in\sqb{I'_{d,k}}}} := \paren{\sigma_c^{(d,k)} a^i_{d,k}}_{\genfrac{}{}{0pt}{}{(d,k)\in D_p}{i\in\sqb{I'_{d,k}}}},
\end{equation}
where $\sigma_c^{(d,k)}$ is defined in \cref{sigma}.

\begin{defn}
Fix a prime $p$. Let $V$ be a finite-dimensional $\F_p$-vector space and let $I\in\cI_p$ be a parameter list. A \textbf{polynomial factor} on $V$ with parameters $I$, denoted $\fB$, is a collection \[\paren{P^i_{d,k}}_{\genfrac{}{}{0pt}{}{(d,k)\in D_p}{i\in\sqb{I_{d,k}}}}\] where $P^i_{d,k}$ is a homogeneous non-classical polynomial of degree $d$ and depth $k$. We also use $\fB$ to denote the map $\fB\colon V\to A_I$ defined by evaluation of the polynomials. We also associate to $\fB$ the partition of $V$ given by the fibers of this map. The atoms of this partition are called the atoms of $\fB$. We write $\|\fB\|:=\|I\|$ and $\deg\fB:=\deg I$. 
\end{defn}

Note that if $\fB$ is a polynomial factor on $V$ with parameters $I$, then $\fB(cx)=c\cdot\fB(x)$ for all $c\in\Fpx$ and $x\in V$ where the $\Fpx$-action on $A_I$ is defined in \cref{A-action}.

\begin{defn}
Fix a prime $p$. Let $V$ be a finite-dimensional $\F_p$-vector space and let $I,I'\in\cI_p$ be two parameter lists. Let $\fB$ and $\fB'$ be two polynomial factors on $V$ with parameters $I$ and $I'$. We say that $\fB'$ is a \textbf{refinement} of $\fB$ if $I\leq I'$ and the lists of polynomials defining $\fB'$ are extensions of the lists of polynomials defining $\fB$. We say that $\fB'$ is a \textbf{weak refinement} of $\fB$ if the partition of $V$ associated to $\fB'$ is a refinement of the partition associated to $\fB$.
\end{defn}

Note that if $\fB'$ is a refinement of $\fB$, then $\fB=\pi\circ\fB'$ where $\pi\colon A_{I'}\to A_I$ is the projection defined in \cref{pi}.

\begin{defn}
Fix a prime $p$ and integer $d\geq 0$. Let $V$ be a finite-dimensional $\F_p$-vector space. For a non-classical polynomial $P\colon V\to\R/\Z$, define the \textbf{$d$-rank of $P$}, denoted $\rank_d P$, to be the smallest integer $r$ such that there exists non-classical polynomials $Q_1,\ldots,Q_r\colon V\to\R/\Z$ of degree at most $d-1$ and a function $\Gamma\colon(\R/\Z)^r\to\R/\Z$ such that $P(x)=\Gamma(Q_1(x),\ldots,Q_r(x))$ for all $x\in V$.

For a polynomial factor $\fB$ on $V$ with parameters $I\in I_p$, defined by a collection $(P^i_{d,k})_{(d,k)\in D_p, i\in [I_{d,k}]}$ where $P^i_{d,k}$ is a homogeneous non-classical polynomial of degree $d$ and depth $k$, we define the \textbf{rank of $\fB$}, denoted $\rank\fB$, to be\[\min_{\bm\lambda\in\prod_{(d,k)\in D_p}\paren{\Z/p^{k+1}\Z}^{I_{d,k}}}\rank_{d'}\paren{\sum_{(d,k)\in D_p}\sum_{i=1}^{I_{d,k}}\lambda^i_{d,k} P^i_{d,k}}\] where \[d':=\min_{(d,k)\in D_p, i\in[I_{d,k}]}\deg\paren{\lambda^i_{d,k} P^i_{d,k}}.\]
\end{defn}

\begin{lemma}
\label{high-rank-exists}
Fix a prime $p$, a parameter list $I\in\cI_p$, and a positive integer $r$. There exists a constant $n_{high-rank}(p,I,r)$ such that for every $n\geq n_{high-rank}(p,I,r)$, there exists a polynomial factor $\fB$ on $\F_p^n$ with parameters $I$ and satisfying $\rank\fB\geq r$. 
\end{lemma}

\begin{proof}
For $(d,k)\in D_p$, write $d=(k+a)(p-1)+b$ where $a\geq 0$ and $b\in\{1,\ldots,p-1\}$. By \cref{poly-exists} there exists a non-classical polynomial of degree $d$ and depth $k$. For example, consider \[\frac{|x_1|^{p-1}|x_2|^{p-1}\cdots|x_a|^b}{p^{k+1}}\pmod1.\] By \cref{homogeneous}, we can decompose this non-classical polynomial as the sum of homogeneous non-classical polynomials of degree at most $d$ and depth at most $k$. Let $P\colon \F_p^a\to\U_{k+1}$ be the homogeneous part of degree $d$ and depth $k$.

Next define $Q\colon{\F_p^a}^{\oplus N}\to\U_{k+1}$ by \[Q(\bx_1,\ldots,\bx_N):=P(\bx_1)+\cdots+P(\bx_N).\] This is clearly a homogeneous non-classical polynomial of degree $d$ and depth $k$. We claim that for $N$ large enough, we have $\rank_d Q\geq r$. The proof of this fact uses several basic results from \cite{TZ12} that are not used elsewhere in this paper.

First, since $P$ has degree exactly $d$, we have that $(D_{h_1}\cdots D_{h_d}P)(x)$ is a constant, independent of $x$, but is not identically zero. This implies that
\[
c:=\E_{h_1,\ldots,h_d\in \F_p^a}e^{2\pi i (D_{h_1}\cdots D_{h_d}P)}<1.\]
The quantity $-\log_p c$ is known as the \emph{analytic rank} of $P$. Now a simple calculation shows that
\[\E_{\bm h_1,\ldots,\bm h_d\in (\F_p^a)^{\oplus N}}e^{2\pi i (D_{\bm h_1}\cdots D_{\bm h_d}Q)}=c^N.\] To conclude we use \cite[Lemma 1.15(iii)]{TZ12} which implies that $\rank_d Q\geq -c_d \log_p(c^N)$ for some constant $c_d>0$ only depending on $d$. Since $c<1$, taking $N$ large enough gives $\rank_dQ \geq r$, as desired.

Thus there exist homogeneous non-classical polynomials $Q_{d,k}\colon V_{d,k}\to\U_{k+1}$ of degree $d$ and depth $k$ that satisfy $\rank_d Q_{d,k}\geq r$ for each $(d,k)\in D_p$. Define the vector space \[V:=\bigoplus_{(d,k)\in D_p} V_{d,k}^{\oplus I_{d,k}}.\]

Define the homogeneous non-classical polynomials $\ol{Q}^i_{d,k}\colon V\to\R/\Z$ for each $(d,k)\in D_p$ and $i\in[I_{d,k}]$ such that $\ol{Q}^i_{d,k}$ is equal to $Q_{d,k}$ evaluated on the $i$th copy of $V_{d,k}$ and does not depend on the other coordinates. In particular, we define \[\ol{Q}^i_{d,k}\paren{(x^{i'}_{d',k'})_{\genfrac{}{}{0pt}{}{(d',k')\in D_p}{i'\in[I_{d',k'}]}}}:=Q_{d,k}(x^i_{d,k}).\]

These polynomials define a polynomial factor $\fB$ on $V$ with parameters $I$ such that each of the homogeneous non-classical polynomials defining $\fB$ has rank at least $r$. Furthermore, since the polynomials defining $\fB$ depend on disjoint sets of variables, it follows that all non-trivial linear combinations of the polynomials defining $\fB$ also have high rank. Setting $n_{high-rank}(p,I,R):=\dim V$, we have constructed the desired polynomial factor on $\F_p^{n_{high-rank}(p,I,r)}$. To extend this construction to $\F_p^n$ with $n\geq n_{high-rank}(p,I,r)$ one can simply add on extra variables that none of the polynomials depend on.
\end{proof}

\begin{lemma}[{\cite[Lemma 2.13]{BFHHL13}}]
\label{hyperplane-rank}
Fix a prime $p$ and a positive integers $d,r$. Let $V$ be a finite-dimensional $\F_p$-vector space and let $P\colon V\to\R/\Z$ be a non-classical polynomial of degree $d$ such that $\rank_d(P)\geq r+p$. Let $U\leq V$ be a codimension-1 hyperplane. Then $\rank_d(P|_U)\geq r$ unless $d=1$ and $P|_U$ is identically zero.

As a consequence, let $\fB$ be a polynomial factor on $V$ and let $P\colon V\to\R/\Z$ be a linear polynomial. Write $\fB'$ for the common refinement of $\fB$ and $\{P\}$. If $\rank\fB'\geq r+p$, then $\rank \fB|_U\geq r$ where $U$ is the codimension-1 hyperplane where $P$ vanishes.
\end{lemma}

\subsection{Equidistribution and consistency sets}

\begin{defn}
\label{consistency}
Fix a prime $p$, integers $d>0$ and $k\geq 0$ satisfying $k\leq\floor{(d-1)/(p-1)}$, and a system $\bm L=(L_1,\ldots,L_m)$ of $m$ linear forms in $\ell$ variables. Define the \textbf{$(d,k)$-consistency set of $\bm L$}, denoted $\Phi_{d,k}(\bm L)$, to be the subset of $\U_{k+1}^m$ consisting of the tuples $\bm a=(a_1,\ldots,a_m)$ such that there exists a finite-dimensional $\F_p$-vector space $V$, a homogeneous non-classical polynomial $P\colon V\to\U_{k+1}$ of degree $d$ and depth $k$, and a tuple $\bm x\in V^\ell$ such that $a_i=P(L_i(\bm x))$ for all $i\in[m]$.

For a parameter list $I\in\cI_p$, define the \textbf{$I$-consistency set of $\bm L$} to be the set of tuples $\bm a=(a_1,\ldots,a_m)\in A_I^m$ such that for each $(d,k)\in D_p$ and $j\in [I_{d,k}]$ the tuple $\paren{(a_1)_{d,k}^j,\ldots,(a_m)_{d,k}^j}$ lies in $\Phi_{d,k}(\bm L)$.
\end{defn}

\begin{lemma}
\label{consistency-subgroup}
Fix a prime $p$, integers $d>0$ and $k\geq 0$ satisfying $k\leq\floor{(d-1)/(p-1)}$, and a system $\bm L=(L_1,\ldots,L_m)$ of $m$ linear forms. The $(d,k)$-consistency set of $\bm L$ is a subgroup of $\U_{k+1}^m$.
\end{lemma}

\begin{proof}
Suppose $\bm a,\bm b\in \Phi_{d,k}(\bm L)$. We wish to show that $-\bm a$ and $\bm a+\bm b$ both lie in this set. By definition, there exist finite-dimensional $\F_p$-vector spaces $V,W$, homogeneous non-classical polynomials $P\colon V\to\U_{k+1}$ and $Q\colon W\to\U_{k+1}$ of degree $d$ and depth $k$, and tuples $\bm x\in V^k$ and $\bm y\in W^k$ such that $a_i=P(L_i(\bm x))$ and $b_i=Q(L_i(\bm y))$ for all $i\in[m]$.

Note that $-\bm a\in\Phi_{d,k}(\bm L)$ since $-P\colon V\to\U_{k+1}$ is a homogeneous non-classical polynomial of degree $d$ and depth $k$ that satisfies $(-P)(L_i(\bm x))=-a_i$.

Now define $P\oplus Q\colon V\oplus W\to\U_{k+1}$ by $(P\oplus Q)(v\oplus w):=P(v)+Q(w)$. One can easily check that $P\oplus Q$ is a homogeneous non-classical polynomial of degree $d$ and depth $k$. Finally note that $(P\oplus Q)(L_i(\bm x\oplus\bm y))=a_i+b_i$, as desired.\footnote{To be completely correct, we also need to show that $\Phi_{d,k}(\bm L)$ is non-empty, which follows from, for example \cref{poly-exists} and \cref{homogeneous}, which together show the existence of homogeneous non-classical polynomial of degree $d$ and depth $k$ for every $(d,k)\in D_p$.}
\end{proof}

\begin{thm}[{Equidistribution \cite[Theorem 3.10]{HHL16}}]
\label{equidistribution}
Fix a prime $p$, a positive integer $d>0$, and a parameter $\epsilon>0$. There exists $r_{equi}(p,d,\epsilon)$ such that the following holds. Let $V$ be a finite-dimensional $\F_p$-vector space and let $\fB$ be a polynomial factor on $V$ with parameters $I$ such that $\deg\fB\leq d$ and $\rank(\fB)\geq r_{equi}(p,d,\epsilon)$. Then for a system of linear forms $\bm L=(L_1,\ldots,L_m)$ consisting of $m$ forms in $\ell$ variables, and a tuple of atoms $\bm a=(a_1,\ldots,a_m)\in\Phi_I(\bm L)$,
\[\abs{\Pr_{\bx\in V^\ell}\paren{\fB(L_i(\bx))=a_i\text{ for all }i\in[m]}-\frac{1}{|\Phi_I(\bm L)|}}\leq\epsilon.\]
\end{thm}

\begin{rem}
To be completely correct, the statement given above follows by combining \cite[Theorem 3.10]{HHL16} and \cite[Corollary 2.13]{HHL16}.
\end{rem}

Note that the probability above is $0$ if $\bm a\not\in\Phi_I(\bm L)$. We typically apply the above theorem with $\epsilon$ that decreases rapidly with $\|I\|$, for example, taking $\epsilon=2\|I\|^m$ and using the fact that $|\Phi_I(\bm L)|\leq\|\fB\|^m$, we see that in this case the probability above is at least $1/(2|\Phi_I(\bm L)|)$.

Consistency sets are often hard to compute exactly. The next two lemmas give exact relations on the sizes of consistency sets in two special cases that occur in this paper.

\begin{defn}
\label{full-dimensional}
Fix a prime $p$. A system $\bm L$ of $m$ linear forms in $\ell$ variables over $\F_p$ is \textbf{full dimensional} if $|\Phi_{d,k}(\bm L)|=|\Phi_{d,k}(\bm L^{\ell})|$ for all $(d,k)\in D_p$ (recall the system $\bm L^{\ell}$ defined in \cref{L} defines an $\ell$-dimensional subspace).
\end{defn}

\begin{lemma}
\label{full-dimensional-exists}
Fix a prime $p$ and a positive integer $\ell$. Let $J\subseteq \F_p^\ell$ be a set that contains at least one vector in each direction (i.e., for each $\bm i\in\F_p^\ell$, there exists $\bm j\in J$ and $\in\Fpx$ such that $\bm i=b\bm j$). Consider the system $\bm L_J:=(L^\ell_{\bm i})_{\bm i\in J}$ of $|J|$ linear forms in $\ell$ variables (recall the linear form $L^\ell_{\bm i}$ defined in \cref{L})). Then $\bm L_J$ is full dimensional.
\end{lemma}

As a special case of this result we see that the system $\ol{\bm L}^\ell$, defined in \cref{L}, is full rank.

\begin{proof}
Note that the system $\bm L_J$ is a subsystem of $\bm L^{\ell}$. This immediately implies that $|\Phi_{d,k}(\bm L_J)|\leq|\Phi_{d,k}(\bm L^\ell)|$, since if $(a_{\bm i})_{\bm i\in\F_p^\ell}\in\Phi_{d,k}(\bm L^\ell)$, then $(a_{\bm i})_{\bm i\in J}\in\Phi_{d,k}(\bm L_J)$.

To go the other direction, we use the homogeneity of our polynomials. Suppose $(a_{\bm j})_{\bm j\in J}\in\Phi_{d,k}(\bm L_J)$. Thus there exists a finite-dimensional $\F_p$-vector space $V$ and a homogeneous non-classical polynomial $P\colon V\to\U_{k+1}$ of degree $d$ and depth $k$, and a vector $\bm x\in V^\ell$ such that $P(L^\ell_{\bm j}(\bm x))=a_{\bm j}$ for all $\bm j\in J$. Now by assumption, for $\bm i\in\F_p^\ell$, there exists $\bm j\in J$ and $b\in\F_p$ such that $\bm i=b\bm j$. Define $a_{\bm i}:=\sigma_b^{(d,k)}a_{\bm j}$ (recall the definition of $\sigma_b$ from \cref{sigma}). We claim that $(a_{\bm i})_{\bm i\in \F_p^\ell}\in\Phi_{d,k}(\bm L^\ell)$. This is true simply because \[P(L^{\ell}_{\bm i}(x))=P(bL^{\ell}_{\bm j}(\bm x))=\sigma_b^{(d,k)}P(L^{\ell}_{\bm j}(x))=\sigma_b^{(d,k)}a_{\bm j}=a_{\bm i}\] where $\bm j\in J$ and $b\in\F_p$ are defined as above. Thus $|\Phi_{d,k}(\bm L_J)|\geq|\Phi_{d,k}(\bm L^\ell)|$, as desired.
\end{proof}

\begin{lemma}
\label{CS-consistency}
Fix a prime $p$. Let $\bm L$ be a system of $m$ linear forms in $\ell$ variables over $\F_p$. Say $\bm L$ is defined by $M$, an $m\times \ell$ matrix. (By this we mean that $L_i(x_1,\ldots,x_{\ell})=M_{i,1}x_1+\cdots+M_{i,\ell}x_{\ell}$ for all $i\in[m]$.) Let $\bm L'$ be a system of $m(n+1)$ linear forms in $\ell+\ell'$ variables, defined by a matrix of the form \[\paren{\begin{array}{c|c} M & 0 \\\hline \begin{array}{c} c_1 M\\ \vdots \\ c_nM\end{array}& N\end{array}}\] where $c_1,\ldots,c_n\in\F_p$ and $N$ is an $mn\times \ell'$ matrix. Let $\bm L''$ be the system of $m(2n+1)$ linear forms in $\ell+2\ell'$ variables, defined by the matrix \[\paren{\begin{array}{c|c|c} M & 0 & 0 \\\hline \begin{array}{c} c_1 M\\ \vdots \\ c_n M\end{array} & N & 0 \\\hline \begin{array}{c} c_1 M\\ \vdots \\ c_n M \end{array} & 0 & N\end{array}}.\] Then for all $(d,k)\in D_p$, we have \[\abs{\Phi_{d,k}(\bm L)}\cdot\abs{\Phi_{d,k}(\bm L'')}=\abs{\Phi_{d,k}(\bm L')}^2.\]
\end{lemma}

\begin{proof}
We construct injections between $\Phi_{d,k}(\bm L)\times\Phi_{d,k}(\bm L'')$ and $\Phi_{d,k}(\bm L')\times\Phi_{d,k}(\bm L')$ in both directions. Write $\sigma_i:=\sigma_{c_i}^{d,k}$ for $i\in[n]$ for the rest of the proof (see \cref{sigma} for the definition).

Consider $\bm a=(a_1,\ldots,a_m)\in\Phi_{d,k}(\bm L)$ and $\bm b =(b_1,\ldots,b_{m(2n+1)})\in\Phi_{d,k}(\bm L'')$. By definition, there exists a finite-dimensional $\F_p$-vector space $V$, a homogeneous non-classical polynomial $P\colon V\to \U_{k+1}$ of degree $d$ and depth $k$ and a vector $\bm x\in V^{\ell}$ such that $P(L_i(\bm x))=a_i$ for all $i\in[m]$. Also by definition, there exists a finite-dimensional $\F_p$-vector space $W$, a homogeneous non-classical polynomial $Q\colon W\to\U_{k+1}$ of degree $d$ and depth $k$ and a vector $(\bm x',\bm y,\bm y')\in W^{\ell}\times W^{\ell'}\times W^{\ell'}$ such that $Q(L''_i(\bm x',\bm y,\bm y'))=b_i$ for all $i\in[m(2n+1)]$.

Now we map $(\bm a,\bm b)$ to the pair $(\bm a',\bm b')$ where $b'_i=b_i$ for $i\in[m(n+1)]$ and $a'_i=a_i+b_i$ for $i\in[m]$ and $a'_{tm+i}=\sigma_ta_i+b_{m(n+t)+i}$ for $i\in[m]$ and $t\in[n]$. We can easily check that no two pairs $(\bm a,\bm b)$ map to the same pair $(\bm a',\bm b')$. All that remains is to check that $\bm a',\bm b'\in\Phi_{d,k}(\bm L')$.

Define $P\oplus Q\colon V\oplus W\to \U_{k+1}$ by $(P\oplus Q)(x\oplus y):=P(x)+Q(y)$. This is clearly a homogeneous non-classical polynomial of degree $d$ and depth $k$. Note that 
$\bm z:=(\bm x\oplus \bm x',\bm 0\oplus \bm y')\in(V\oplus W)^{\ell+\ell'}$
satisfies $(P\oplus Q)(L'_i(\bm z))=a'_i$ for all $i\in[m(n+1)]$. Similarly note that
$\bm z':=(\bm 0\oplus \bm x',\bm 0\oplus \bm y)\in(V\oplus W)^{\ell+\ell'}$
satisfies $(P\oplus Q)(L'_i(\bm z'))=b'_i$ for all $i\in[m(n+1)]$. This demonstrates the first injection.

Now consider $\bm a=(a_1,\ldots,a_{m(n+1)})\in\Phi_{d,k}(\bm L')$ and $\bm b=(b_1,\ldots,b_{m(n+1)})\in\Phi_{d,k}(\bm L')$. By definition, there exists a finite-dimensional $\F_p$-vector space $V$, a homogeneous non-classical polynomial $P\colon V\to \U_{k+1}$ of degree $d$ and depth $k$ and a vector $(\bm x,\bm y)\in V^{\ell}\times V^{\ell'}$ such that $P(L_i'(\bm x,\bm y))=a_i$ for $i\in[m(n+1)]$. Also by definition, there exists a finite-dimensional $\F_p$-vector space $W$, a homogeneous non-classical polynomial $Q\colon W\to\U_{k+1}$ of degree $d$ and depth $k$ and a vector $(\bm x',\bm y')\in W^{\ell}\times W^{\ell'}$ such that $Q(L_i'(\bm x',\bm y'))=b_i$ for $i\in[m(n+1)]$.

We map $(\bm a,\bm b)$ to the pair $(\bm a',\bm b')$ where $a'_i=a_i$ for $i\in[m]$ and $b'_i=a_i+b_i$ for $i\in[m]$ and $b'_{m+i}=a_{m+i}$ for $i\in[mn]$ and $b'_{m(n+1)+i}=b_{m+i}$ for $i\in[mn]$. We can easily check that no two pairs $(\bm a,\bm b)$ map to the same pair $(\bm a',\bm b')$. All that remains is to check that $\bm a'\in\Phi_{d,k}$ and $\bm b'\in\Phi_{d,k}(\bm L'')$.

As above, define the homogeneous non-classical polynomial $P\oplus Q\colon V\oplus W\to \U_{k+1}$ of degree $d$ and depth $k$ by $(P\oplus Q)(x\oplus y):=P(x)+Q(y)$. Note that $\bm z:=\bm x\oplus \bm 0\in(V\oplus W)^{\ell}$ satisfies $(P\oplus Q)(\bm L_i(\bm z))=a'_i$ for $i\in[m]$ and $\bm z':=(\bm x\oplus \bm x',\bm y\oplus \bm 0,\bm 0\oplus \bm y')\in(V\oplus W)^{\ell}\times (V\oplus W)^{\ell'}\times (V\oplus W)^{\ell'}$ satisfies $(P\oplus Q)(L_i''(\bm z'))=b'_i$ for $i\in[m(2n+1)]$, as desired.
\end{proof}

\subsection{Subatom selection functions}

A situation that often occurs is the following. We have a polynomial factor $\fB$ with parameters $I$ and a refinement $\fB'$ with parameters $I'$. We use the word atom to refer to the atoms of the partition induced by $\fB$; these atoms are indexed by $A_I$. We use the word subatom to refer to the atoms of the partition induced by $\fB'$; these atoms are indexed by $A_{I'}$. The projection map $\pi\colon A_{I'}\to A_I$, defined in \cref{pi}, maps a subatom to the atom that it is contained in.

We wish to designate one subatom inside each atom as special. This choice is given by a map $s\colon A_{I}\to A_{I'}$ that is a right inverse for $\pi$. In this paper we define a certain class of these maps that we call subatom selection functions that have several desirable properties.

First we define certain polynomials $P_{d,k}\colon\F_p\to\U_{k+1}$ for each $(d,k)\in D_p$. For each $i\in\{1,\ldots,p-1\}$, there exists a homogeneous non-classical polynomial $\F_p\to\U_{k+1}$ of degree $k(p-1)+i$ and depth $k$ in one variable. (This follows from \cref{poly-exists} and \cref{homogeneous}.) Let $P_{k(p-1)+i,k}$ be one such a polynomial. Finally, define \[P_{(k+s)(p-1)+i,k+s}:=P_{k(p-1)+i,k}\] for all $i\in\{0,\ldots,p-1\}$ and $s\geq 0$. This defines $P_{d,k}$ for each $(d,k)\in D_p$.

\begin{defn}
\label{sas}
Fix a prime $p$ and parameter lists $I,I'\in\cI_p$ satisfying $I\leq I'$. Let $c_{d,k}^{i,j}\in\Z/p^{k+1}\Z$ be arbitrary elements for $(d,k)\in D_p$ and $i\in[I_{1,0}]$ and $I_{d,k}<j\leq I'_{d,k}$. A \textbf{subatom selection function} is a map of the form $s_{\bm c}\colon A_I\to A_{I'}$, defined by
\[\sqb{s_{\bm c}(a)}_{d,k}^i=
\begin{cases}
a^i_{d,k} &\text{if }i\leq I_{d,k}\\
\sum_{j=1}^{I_{1,0}}c_{d,k}^{j,i}P_{d,k}(|a_{1,0}^j|)\quad&\text{otherwise,}
\end{cases}\] 
where the maps $P_{d,k}\colon\F_p\to\U_{k+1}$ were defined in the preceding paragraph and $|\cdot|$ is the standard map $\U_1\to\F_p$.
\end{defn}

\begin{lemma}
\label{sas-properties}
Fix a prime $p$, parameter lists $I,I'\in\cI_p$ satisfying $I\leq I'$, and a subatom selection function $s_{\bm c}\colon A_I\to A_{I'}$. The following hold:
\begin{enumerate}[(i)]
\item $\pi\circ s=\Id$ (where $\pi\colon A_{I'}\to A_I$ is defined in \cref{pi});
\item for $a\in A_I$ and $b\in\Fpx$, we have $b\cdot s_{\bm c}(a)=s_{\bm c}(b\cdot a)$ (where the action of $\Fpx$ on $A_I$ and $A_{I'}$ is defined in \cref{A-action});
\item for every system $\bm L$ of $m$ linear forms and every consistent tuple of atoms $(a_1,\ldots,a_m)\in\Phi_I(\bm L)$, we have \[\paren{s_{\bm c}(a_1),\ldots,s_{\bm c}(a_m)}\in \Phi_{I'}(\bm L)\](see \cref{consistency} for the definition of the consistency sets $\Phi_I(\bm L)$ and $\Phi_{I'}(\bm L)$).
\end{enumerate}
\end{lemma}

\begin{proof}
Property (i) is immediate.

For the property (ii), by definition, we have \[[b\cdot s_{\bm c}(a)]^i_{d,k}=\sigma_{b}^{(d,k)}[s_{\bm c}(a)]^i_{d,k},\]where $\sigma_b^{(d,k)}$ is defined in \cref{sigma}. Now $s_{\bm c}(b\cdot a)^i_{d,k}=\sigma_b^{(d,k)}a^i_{d,k}$ if $i\leq I_{d,k}$, so we are done in this case.

Assume otherwise. Define $d'$ such that $d'\equiv d\pmod{p-1}$ and $d=k(p-1)+i$ for some $i\in\{1,\ldots,p-1\}$. Remember that $P_{d,k}$ is a homogeneous non-classical polynomial of degree $d'$ and depth $k$. Note that $\sigma_{b}^{(1,0)}=b\in\Z/p\Z$. Then we have\[[s_{\bm c}(b\cdot a)]^i_{d,k}=\sum_{j=1}^{I_{1,0}}c^{j,i}_{d,k}P_{d,k}(b|a^j_{1,0}|)=\sigma_b^{(d',k)}\sum_{j=1}^{I_{1,0}}c^{j,i}_{d,k}P_{d,k}(|a^j_{1,0}|).\]

To complete the proof of (ii) we need to show that $\sigma_b^{d,k}=\sigma_b^{d',k}$ whenever $d\equiv d'\pmod{p-1}$. This follows from \cref{sigma} which implies that $\sigma_b^{d,k}$ is uniquely determined by the facts that $\sigma_b^{d,k}\equiv b^d\pmod{p}$ and $\paren{\sigma_b^{(d,k)}}^{p-1}=1$ in $\Z/p^{k+1}\Z$. The first property does not change when $d$ changes by a multiple of $p-1$ (by Fermat's little theorem) and the second property does not depend on $d$ at all. Thus we conclude the desired result.

Now we prove (iii). We know, by \cref{counting-lemma}, that the consistency set $\Phi_{d,k}(\bm L)$ is a subgroups of $\U_{k+1}^m$. Thus it suffices to prove that for $(a_1,\ldots,a_m)\in\Phi_{1,0}(\bm L)\subseteq\U_1^m$ we have $(P_{d,k}(|a_1|),\ldots,P_{d,k}(|a_m|))\in\Phi_{d,k}(\bm L)$. Given that $(a_1,\ldots,a_m)\in\Phi_{1,0}$ we know that there exists a finite-dimensional $\F_p$-vector space $V$, a linear function $P\colon V\to\U_1$, and vectors $\bx=(x_1,\ldots,x_\ell)\in V^\ell$ such that $P(L_i(\bx))=a_i$ for $i\in[m]$. Since $P$ and $L_i$ are linear, they commute, and thus $L_i(\bm y)=|a_i|$ for all $i\in[m]$ where $\bm y\in\F_p^\ell$ is defined by $y_i=|P(x_i)|$. Finally, since $P_{d,k}$ is a homogeneous non-classical polynomial of degree $d'$ and depth $k$, we have $(P_{d,k}(|a_1|),\ldots,P_{d,k}(|a_m|))=(P_{d,k}(L_1(\bm y)),\ldots,P_{d,k}(L_m(\bm y))\in\Phi_{d',k}(\bm L)$.

To complete the proof, let $Q\colon\F_p^n\to\U_{k+1}$ be a homogeneous non-classical polynomial of degree $d$ and depth $k$. This exists by \cref{poly-exists} and \cref{homogeneous} as long as $n\geq \ceil{(d-1)/(p-1)}-k$. Then consider the map $P_{d,k}\oplus Q\colon \F_p\oplus\F_p^n\to\U_{k+1}$ defined as usual by $(P_{d,k}\oplus Q)(x\oplus y):=P_{d,k}(x)+Q(y)$. This is clearly a non-classical polynomial of degree $d$ and depth $k$. Furthermore, \[(P_{d,k}\oplus Q)(b(x\oplus y))=\sigma_b^{(d',k)}P_{d,k}(x)+\sigma_b^{(d,k)}Q(y)=\sigma_b^{(d,k)}(P_{d,k}\oplus Q)(x\oplus y),\] since $d'\equiv d\pmod{p-1}$. Considering $\bm y\oplus \bm 0\in(\F_p\oplus\F_p^n)^\ell$ shows that $(P_{d,k}(|a_1|),\ldots,P_{d,k}(|a_m|))\in\Phi_{d,k}(\bm L)$, as desired.\footnote{This argument also shows that for any pattern $\bm L$, the consistency sets $\Phi_{d,k}(\bm L)$ are nested as $d$ increases by multiples of $p-1$, though this is the only time we will need that fact in this paper.}
\end{proof}


\section{Arithmetic regularity and subatom selection}
\label{sec-regularity}

This section follows a fairly standard formula in the theory of regularity lemmas. We start with an inverse theorem, due to Tao and Ziegler \cite{TZ12}. Iterating the inverse theorem produces a weak regularity lemma (\cref{reg-1}), iterating the weak regularity lemma produces a regularity lemma (\cref{reg-2}), and iterating the regularity lemma gives a strong regularity lemma (\cref{reg-3}). Finally we use the probabilistic method applied to the output of the strong regularity lemma to produce the desired ``subatom selection'' result (\cref{subatom-selection}).

\cref{reg-1} and \cref{reg-2} are very similar to results in \cite{BFL12, BFHHL13}, differing only in some technical details. The main innovation in this section is that \cref{reg-3} is much stronger than previous results. To accomplish this, we iterate \cref{reg-2} with the complexity parameter (i.e., degree of the non-classical polynomials) increasing at each step of the iteration. To our knowledge, this idea has not appeared previously in the literature.

\textbf{Notation and conventions:} Recall that a polynomial factor $\fB$ on a vector space $V$ with parameters $I$ gives rise to a partition (or $\sigma$-algebra) on $V$ whose atoms are the fibers of the map $\fB\colon V\to A_I$. For a function $f\colon V\to\C$, we write $\E[f|\fB]\colon V\to\C$ for the projection of $f$ onto the $\sigma$-algebra generated by $\fB$. Concretely, $\E[f|\fB](x)$ is defined to be the average of $f$ over the atom of $\fB$ which contains $x$.

In this section we have to deal with many growth functions. Without loss of generality we always assume that these growth functions are monotone in all their parameters.

\begin{thm}[Inverse theorem {\cite[Theorem 1.10]{TZ12}}]
\label{inverse}
Fix a prime $p$, a positive integer $d$, and a parameter $\delta>0$. There exists $\epsilon_{inv}(p,d,\delta)>0$ such that the following holds. Let $V$ be a finite-dimensional $\F_p$-vector space. Given a function $f\colon V\to\C$ satisfying $\|f\|_\infty\leq1$ and $\|f\|_{U^{d+1}}>\delta$, there exists a non-classical polynomial $P\colon V\to\R/\Z$ of degree at most $d$ such that \[\abs{\E_{x\in V} f(x)e^{-2\pi i P(x)}}\geq \epsilon_{inv}(p,d,\delta).\]
\end{thm}

The next lemma is important for making factors high rank and its second claim will be critical in proving the stronger regularity lemma where we need to produce a refinement (instead of a weak refinement).

\begin{lemma}[Making factors high rank {\cite{HHL16}, c.f. \cite[Theorem 2.19]{BFHHL13}}]
\label{high-rank}
Fix a prime $p$, positive integer $d, C_0$, and a non-decreasing function $r\colon\Zp\to\Zp$. There exist constants $C_{rank}(p,d, C_0,r)$ and $r_{rank}(p,d, C_0,r)$ such that the following holds. Let $V$ be a finite-dimensional $\F_p$-vector space. Suppose that $\fB$ and $\fB'$ are polynomial factors on $V$ with degree at most $d$ such that $\fB'$ refines $\fB$ and $\|\fB'\|\leq C_0$ and
\[\rank\fB\geq r_{rank}(p,d,C_0,r).\]
Then there is a polynomial factor $\fB''$ on $V$ that weakly refines $\fB'$, refines $\fB$, and satisfies $\|\fB''\|\leq C_{rank}(p,d,C_0,r)$ and $\deg\fB''\leq d$ and $\rank\fB''\geq r(\|\fB''\|)$.
\end{lemma}

\begin{lemma}[Weak arithmetic regularity]
\label{reg-1}
Fix a prime $p$, positive integers $d, R, C_0$, a parameter $\eta>0$, and a non-decreasing function $r\colon\Zp\to\Zp$. There exist constants $C_{reg'}(p, d, R, C_0, \eta, r)$ and $r_{reg'}(p,d,R,C_0,\eta,r)$ such that the following holds. Let $V$ be a finite-dimensional $\F_p$-vector space and let $\fB_0$ be a polynomial factor on $V$ satisfying $\|\fB_0\|\leq C_0$ and $\deg \fB_0\leq d$ and $\rank\fB_0\geq r_{reg'}(p,d,R,C_0,\eta,r)$. Given functions $f^{(1)},\ldots,f^{(R)}\colon V\to[0,1]$, there exists a polynomial factor $\fB$ on $V$ that refines $\fB_0$ with the following properties. There exists a decomposition \[f^{(\ell)}=f_{str}^{(\ell)}+f_{psr}^{(\ell)}\] for each $\ell\in[R]$ such that:
\begin{enumerate}[(i)]
    \item $f^{(\ell)}_{str}=\E[f^{(\ell)}|\fB]$ for each $\ell\in[R]$;
    \item $\|f^{(\ell)}_{psr}\|_{U^{d+1}}< \eta$ for each $\ell\in[R]$;
    \item $f_{str}^{(\ell)}$ has range $[0,1]$ and $f_{psr}^{(\ell)}$ has range $[-1,1]$ for each $\ell\in[R]$;
    \item $\rank\fB\geq r(\|\fB\|)$;
    \item $\|\fB\|\leq C_{reg'}(p, d, R, C_0, \eta, r)$  and $\deg\fB \leq  d$.
\end{enumerate}
\end{lemma}

\begin{proof}
Set $M:=\ceil{R\epsilon_{inv}(p,d,\eta)^{-2}}$. Define non-decreasing functions $r_i\colon\Zp\to\Zp$ for each $i=0,\ldots,M$ such that $r_0=r$ and $r_{i+1}(N)\geq r_{rank}(p,d,p^{d^3}N,r_i)$ and such that $r_{i+1}(N)\geq r_i(N)$ for all $i=0,\ldots,M-1$ and all $N\in\Zp$. Define $r_{reg'}(p,d,R,C_0,\eta,r):=r_M(C_0)$.

We construct a list of polynomial factors $\fB_0,\fB_1,\ldots,\fB_m$ on $V$ such that
\begin{itemize}
    \item $\fB_i$ refines $\fB_{i-1}$ for $i=1,\ldots,m$;
    \item $\rank \fB_i\geq r_{M-i}(\|\fB_i\|)$ for $i=0,\ldots,m$;
    \item $\|\fB_i\|\leq C_{rank}(p,d,p^{d^3}\|\fB_{i-1}\|,r_{M-i})$ and $\deg\fB_i\leq d$ for $i=1,\ldots, m$.
\end{itemize}
Suppose we have constructed polynomial factors $\fB_0,\ldots,\fB_i$ with the above properties. If $\|f^{(\ell)}-\E[f^{(\ell)}|\fB_i]\|_{U^{d+1}}<\eta$ for each $\ell\in[R]$ we halt the iteration. Otherwise there is some $\ell\in[R]$ such that, writing $g:=f^{(\ell)}-\E[f^{(\ell)}|\fB_i]$, we have
\[\|g\|_{U^{d+1}}\geq\eta.\]
By \cref{inverse}, there exists a non-classical polynomial $P\colon V\to\R/\Z$ of degree at most $d$ such that
\[\abs{\E_{x\in V}g(x)e^{-2\pi i P(x)}}\geq\epsilon_{inv}(p,d,\eta).\]
By \cref{homogeneous}, we can write $P=P_1+\cdots +P_C$ as the sum of homogeneous non-classical polynomials. There are at most $\sum_{i=1}^d1+\floor{(i-1)/(p-1)}\leq d^2$ terms in this sum. Let $\fB'_i$ to be the polynomial factor defined by the polynomials defining $\fB$ as well as the polynomials $P_1,\ldots,P_C$. Note that $\|\fB'_i\|\leq p^{d^3}\|\fB_i\|$. Finally let $\fB_{i+1}$ be the polynomial factor produced by applying \cref{high-rank} to $\fB_i$ and $\fB_i'$ with parameters $p,d,r_{M-i-1}$. In particular $\fB_i'$ refines $\fB_i$ and
\begin{align*}
\rank \fB_i
& \geq r_{M-i}(\|\fB_i\|)\\
&\geq r_{rank}(p,d,p^{d^3}\|\fB_i\|,r_{M-i-1})\\
&\geq r_{rank}(p,d,\|\fB_i'\|,r_{M-i-1}),
\end{align*}
so the hypotheses of \cref{high-rank} are satisfied. Thus we have defined $\fB_{i+1}$ with all the desired properties.

Next we claim that this iteration must stop after at most $M$ steps. We claim that
\[\sum_{\ell=1}^R\norm{\E[f^{(\ell)}|\fB_i]}_2^2\]
increases by at least $\epsilon_{inv}(p,d,\eta)^2$ each time $i$ increases. Since this sum is clearly bounded between 0 and $R$, it suffices to prove this claim.

First note that by the Cauchy-Schwarz inequality, $\|\E[f^{(\ell)}|\fB_{i+1}]\|_2^2\geq\|\E[f^{(\ell)}|\fB_i]\|_2^2$ holds for all $\ell$. Now pick $\ell\in[R]$ as in the $i$th iteration, define $g:=f^{(\ell)}-\E[f^{(\ell)}|\fB_i]$, and let $P$ be the non-classical polynomial defined in the $i$th iteration. Note in particular that $e^{-2\pi i P(x)}$ is in the $\sigma$-algebra defined by $\fB_i'$. Then we compute
\begin{align*}
\|\E[f^{(\ell)}|\fB_{i+1}]\|_2^2-\|\E[f^{(\ell)}|\fB_i]\|_2^2
 & \geq \|\E[f^{(\ell)}|\fB_i']\|_2^2-\|\E[f^{(\ell)}|\fB_i]\|_2^2 \\
 & = \|\E[f^{(\ell)}|\fB_i']-\E[f^{(\ell)}|\fB_i]\|_2^2 \\
 & =\|\E[g|\fB_i']\|_2^2 \\
 & \geq \ang{\E[g|\fB_i'],e^{2\pi i P}}^2 \\
 & = \ang{g,e^{2\pi i P}}^2 \\
 & \geq \epsilon_{inv}(p,d,\eta)^2.
\end{align*}

Thus we have produced $\fB_m$ with $m\leq M$ such that $\fB_m$ refines $\fB_0$ and $\rank \fB_m\geq r_{M-m}(\|\fB_m\|)\geq r(\|\fB_m\|)$ and $\|f^{(\ell)}-\E[f^{(\ell)}|\fB_m]\|_{U^{d+1}}<\eta$ for each $\ell\in[R]$. Defining $f^{(\ell)}_{str}:=\E[f^{(\ell)}|\fB_m]$ and $f^{(\ell)}_{psr}:=f^{(\ell)}-f^{(\ell)}_{str}$, we immediately see that conclusions (i), (ii), (iii), and (iv) hold. Conclusion (v) holds by defining $C_{reg'}(p,d,R,C_0,\eta,r)$ to be the $M$-fold iteration of the function $N\mapsto C_{rank}(p,d,p^{d^3}N,r_M)$ applied to $C_0$.
\end{proof}

\begin{lemma}[Arithmetic regularity]
\label{reg-2}
Fix a prime $p$, positive integers $d, R, C_0$, a parameter $\theta>0$, a non-increasing function $\eta \colon\Zp\to(0,1)$, and a non-decreasing function $r\colon\Zp\to\Zp$. There exist constants $C_{reg''}(p, d, R, C_0, \theta, \eta, r)$ and $r_{reg''}(p,d,R,C_0,\theta,\eta,r)$ such that the following holds. Let $V$ be a finite-dimensional $\F_p$-vector space and let $\fB_0$ be a polynomial factor on $V$ satisfying $\|\fB_0\|\leq C_0$ and $\deg\fB_0\leq d$ and $\rank \fB_0\geq r_{reg''}(p,d,r,I_0,\theta,\eta,r)$. Given functions $f^{(1)},\ldots,f^{(R)}\colon V\to[0,1]$, there exists a polynomial factor $\fB$ on $V$ that refines $\fB_0$ with the following properties. There exists a decomposition \[f^{(\ell)}=f_{str}^{(\ell)}+f_{psr}^{(\ell)}+f_{sml}^{(\ell)}\] for each $\ell\in[R]$ such that:
\begin{enumerate}[(i)]
    \item $f^{(\ell)}_{str}=\E[f^{(\ell)}|\fB]$ for each $\ell\in[R]$;
    \item $\|f^{(\ell)}_{psr}\|_{U^{d+1}}< \eta(\|\fB\|)$ for each $\ell\in[R]$;
    \item $f_{str}^{(\ell)}$ and $f_{str}^{(\ell)}+f_{sml}^{(\ell)}$ have range $[0,1]$ and $f_{psr}^{(\ell)}$ and $f_{sml}^{(\ell)}$ have range $[-1,1]$ for each $\ell\in[R]$;
    \item $\rank\fB\geq r(\|\fB\|)$;
    \item $\|f^{(\ell)}_{sml}\|_2<\theta$ for each $\ell\in[R]$;
    \item $\|\fB\|\leq C_{reg''}(p, d, R, C_0, \theta, \eta, r)$  and $\deg\fB \leq  d$.
\end{enumerate}
\end{lemma}

\begin{proof}
Set $M:=\ceil{R\theta^{-2}}$. Define non-decreasing functions $r_i\colon\Zp\to\Zp$ for each $i=0,\ldots,M$ such that $r_0=r$ and $r_{i+1}(N)\geq r_{reg'}(p,d,R,N,\eta(N),r_i)$ and such that $r_{i+1}(N)\geq r_i(N)$ for all $i=0,\ldots,M-1$ and all $N\in\Zp$. Define $r_{reg''}(p,d,R,C_0,\theta,\eta,r):=r_M(C_0)$.

We construct a list of polynomial factors $\fB_0,\fB_1,\ldots,\fB_m$ on $V$ such that
\begin{itemize}
    \item $\fB_i$ refines $\fB_{i-1}$ for $i=1,\ldots,m$;
    \item $\rank \fB_i\geq r_{M-i}(\|\fB_i\|)$ for $i=0,\ldots,m$;
    \item $\|\fB_i\|\leq C_{reg'}(p,d,R,\|\fB_{i-1}\|,\eta(\|\fB_{i-1}\|),r_{M-i})$ and $\deg\fB_i\leq d$ for $i=1,\ldots, m$.
\end{itemize}
Suppose we have constructed polynomial factors $\fB_0,\ldots,\fB_i$ with the above properties. If $i\geq 1$ and $\|E[f^{(\ell)}|\fB_i]\|_2^2-\|\E[f^{(\ell)}|\fB_{i-1}]\|_2^2<\theta^2$ for each $\ell\in[R]$ we halt the iteration. Otherwise let $\fB_{i+1}$ be the polynomial factor produced by applying \cref{reg-1} to $\fB_i$ with parameters $p,d,R,\|\fB_i\|,\eta(\|\fB_i\|),r_{M-i-1}$. Note that
\begin{align*}
\rank \fB_i 
& \geq r_{M-i}(\|\fB_i\|) \\
& \geq r_{reg'}(p,d,R,\|\fB_i\|,\eta(\|\fB_i\|),r_{M-i-1}),
\end{align*}
so the hypotheses of \cref{reg-1} are satisfied.

Next we claim that this iteration must stop after at most $M$ steps. This is obvious since
\[\sum_{\ell=1}^R\norm{\E[f^{(\ell)}|\fB_i]}_2^2-\sum_{\ell=1}^R\norm{\E[f^{(\ell)}|\fB_{i-1}]}_2^2\geq\theta^2\] for $i=1,\ldots,m$ and the sum is bounded between 0 and $R$.

Thus we have produced $\fB_{m-1}$ with $m\leq M$ such that $\fB_{m-1}$ refines $\fB_0$ and $\rank \fB_{m-1}\geq r_{M-m+1}(\|\fB_{m-1}\|)\geq r(\|\fB_{m-1}\|)$ and $\|f^{(\ell)}-\E[f^{(\ell)}|\fB_m]\|_{U^{d+1}}<\eta(\|\fB_{m-1}\|)$ for each $\ell\in[R]$ and $\norm{\E[f^{(\ell)}|\fB_m]-\E[f^{(\ell)}|\fB_{m-1}]}_2<\theta$ for each $\ell\in[R]$. Defining $f^{(\ell)}_{str}:=\E[f^{(\ell)}|\fB_{m-1}]$ and $f^{(\ell)}_{psr}:=f^{(\ell)}-E[f^{(\ell)}|\fB_m]$ and $f^{(\ell)}_{sml}:=\E[f^{(\ell)}|\fB_m]-\E[f^{(\ell)}|\fB_{m-1}]$, we immediately see that conclusions (i), (ii), (iii), (iv), and (v) hold. Conclusion (vi) holds by defining $C_{reg''}(p,d,R,C_0,\theta,\eta,r)$ to be the $M$-fold iteration of the function $N\mapsto C_{reg'}(p,d,R,N,\eta(N),r_M)$ applied to $C_0$.
\end{proof}

\begin{lemma}[Strong arithmetic regularity]
\label{reg-3}
Fix a prime $p$, positive integers $R,C_0,d_0$, a parameter $\zeta>0$, non-increasing functions $\eta, \theta\colon\Zp\times\Zp\to(0,1)$, and non-decreasing functions $d, r\colon\Zp\times\Zp\to\Zp$. There exist constants $C_{reg'''}(p, R, C_0,d_0, \zeta, \eta, \theta, d, r)$ and $D_{reg'''}(p, R, C_0,d_0, \zeta, \eta, \theta, d, r)$  and $r_{reg'''}(p,R,C_0,d_0,\zeta,\eta,\theta,d,r)$ such that the following holds. Let $V$ be a finite-dimensional $\F_p$-vector space and let $\fB_0$ be a polynomial factor on $V$ satisfying $\|\fB_0\|\leq C_0$ and $\deg \fB_0\leq d_0$ and $\rank \fB_0\geq r_{reg'''}(p,R,C_0,d_0,\zeta,\eta,\theta,d,r)$. Given functions $f^{(1)},\ldots,f^{(R)}\colon V\to[0,1]$, there exist a polynomial factor $\fB$ and a refinement $\fB'$ both on $V$ with parameters $I$ and $I'$ with the following properties. There exists a decomposition \[f^{(\ell)}=f_{str}^{(\ell)}+f_{psr}^{(\ell)}+f_{sml}^{(\ell)}\] for each $\ell\in[R]$ such that:
\begin{enumerate}[(i)]
    \item $f^{(\ell)}_{str}=\E[f^{(\ell)}|\fB']$ for each $\ell\in[R]$;
    \item $\|f^{(\ell)}_{psr}\|_{U^{d(\deg\fB,\|\fB\|)+1}}< \eta(\deg\fB',\|\fB'\|)$ for each $\ell\in[R]$;
    \item $f_{str}^{(\ell)}$ and $f_{str}^{(\ell)}+f_{sml}^{(\ell)}$ have range $[0,1]$ and $f_{psr}^{(\ell)}$ and $f_{sml}^{(\ell)}$ have range $[-1,1]$ for each $\ell\in[R]$;
    \item $\rank\fB \geq r(\deg\fB,\|\fB\|)$ and $\rank \fB'\geq r(\deg\fB',\|\fB'\|)$;
    \item $\|f^{(\ell)}_{sml}\|_2<\theta(\deg\fB, \|\fB\|)$ for each $\ell\in[R]$;
    \item for all but at most a $\zeta$-fraction of $a\in A_{I'}$ it holds that \[\left|\E_{x\in \fB'^{-1}(a)}[f^{(\ell)}(x)]-\E_{x\in\fB^{-1}(\pi(a))}[f^{(\ell)}(x)]\right|<\zeta\]for each $\ell\in[R]$ (recall the definition of the atom indexing sets from \cref{A} and the projection map $\pi\colon A_{I'}\to A_I$ from \cref{pi});
    \item $\|\fB'\|\leq C_{reg'''}(p, R, C_0,d_0,\zeta, \eta, \theta, d, r)$ and $\deg\fB' \leq  D_{reg'''}(p, R,C_0,d_0,\zeta, \eta, \theta, d, r)$.
\end{enumerate}
\end{lemma}

\begin{proof}
Set $M:=\ceil{R\zeta^{-3}}$. Define non-decreasing functions $r_i\colon\Zp\times\Zp \to\Zp$ for each $i=0,\ldots,M$ such that $r_0=r$ and $r_{i+1}(D,N)\geq r_{reg''}(p,d(D,N),R,N,\theta(D,N),\eta(d(D,N),\cdot),r_i(d(D,N),\cdot))$ and such that $r_{i+1}(D,N)\geq r_i(D,N)$ for all $i=0,\ldots,M-1$ and all $D,N\in\Zp$. Define $r_{reg'''}(p,R,C_0,d_0,\zeta,\eta,\theta,d,r):=r_M(d_0,C_0)$.

We construct a list of polynomial factors $\fB_0,\fB_1,\ldots,\fB_m$ on $V$ such that
\begin{itemize}
    \item $\fB_i$ refines $\fB_{i-1}$ for $i=1,\ldots,m$;
    \item $\rank \fB_i\geq r_{M-i}(\deg\fB_i,\|\fB_i\|)$ for $i=0,\ldots,m$;
    \item $\|\fB_i\|\leq C_{reg''}(p,d(\deg\fB_{i-1},\|\fB_{i-1}\|),R,\|\fB_{i-1}\|,\theta,\eta,r_{M-i}(d(\deg\fB_{i-1},\|\fB_{i-1}\|),\cdot))$ and $\deg\fB_i\leq d(\deg\fB_{i-1},\|\fB_{i-1}\|)$ for $i=1,\ldots, m$.
\end{itemize}
Suppose we have constructed polynomial factors $\fB_0,\ldots,\fB_i$ with the above properties. If $i\geq 1$ and $\|E[f^{(\ell)}|\fB_i]\|_2^2-\|\E[f^{(\ell)}|\fB_{i-1}]\|_2^2<\zeta^3$ for each $\ell\in[R]$ we halt the iteration. Otherwise let $\fB_{i+1}$ be the polynomial factor produced by applying \cref{reg-2} to $\fB_i$ with parameters $p,d(\deg\fB_i,\|\fB_i\|),R,\|\fB_i\|,\theta(\deg\fB_i,\|\fB_i\|),\eta(d(\deg\fB_i,\|\fB_i\|),\cdot),r_{M-i-1}(d(\deg\fB_i,\|\fB_i\|),\cdot)$. Note that
\begin{align*}
\rank \fB_i 
& \geq r_{M-i}(\deg\fB_i,\|\fB_i\|) \\
& r_{reg''}(p,d(\deg\fB_i,\|\fB_i\|),R,\|\fB_i\|,\theta(\deg\fB_i,\|\fB_i\|),\eta(d(\deg\fB_i,\|\fB_i\|),\cdot),r_{M-i-1}(d(\deg\fB_i,\|\fB_i\|),\cdot)),
\end{align*}
so the hypotheses of \cref{reg-2} are satisfied.

Next we claim that this iteration must stop after at most $M$ steps. This is obvious since
\[\sum_{\ell=1}^R\norm{\E[f^{(\ell)}|\fB_i]}_2^2-\sum_{\ell=1}^R\norm{\E[f^{(\ell)}|\fB_{i-1}]}_2^2\geq\zeta^3\] for $i=1,\ldots,m$ and the sum is bounded between 0 and $R$.

Thus we have produced $\fB_{m-1}$ and $\fB_m$ with $m\leq M$ such that $\fB_{m-1}$ refines $\fB_0$ and $\fB_m$ refines $\fB_{m-1}$. Furthermore, $\rank \fB_{m-1}\geq r_{M-m+1}(\deg\fB_{m-1},\|\fB_{m-1}\|)\geq r(\deg\fB_{m-1},\|\fB_{m-1}\|)$ and $\rank \fB_m\geq r_{M-m}(\deg\fB_m,\|\fB_m\|)\geq r(\deg\fB_m,\|\fB_m\|)$. Also $\norm{\E[f^{(\ell)}|\fB_m]-\E[f^{(\ell)}|\fB_{m-1}]}_2^2<\zeta^3$ for each $\ell\in[R]$. Let $f^{(\ell)}=f_{str}^{(\ell)}+f_{psr}^{(\ell)}+f_{sml}^{(\ell)}$ be the decomposition produced by the last application of \cref{reg-2}. This decomposition satisfies conclusions (i), (ii), (iii), and (v). Conclusion (iv) we already verified, and conclusion (vi) follows from Markov's inequality applied to the bound $\norm{\E[f^{(\ell)}|\fB_m]-\E[f^{(\ell)}|\fB_{m-1}]}_2^2<\zeta^3$. Finally conclusion (vii) holds where we define the pair $( D_{reg'''}(p, R,C_0,d_0,\zeta, \eta, \theta, d, r), C_{reg'''}(p, R, C_0,d_0,\zeta, \eta, \theta, d, r))$ to be the $M$-fold iteration of the function $(D,N)\mapsto  (C_{reg''}(p,d(D,N),R,N,\theta,\eta,r_M(d(D,N),\cdot)),d(D,N))$ applied to $(d_0,C_0)$.
\end{proof}

Recall that for a parameter list $I\in\cI_p$ and an atom $a\in A_I$, we write $a_{d,k}\in\U_{k+1}^{I_{d,k}}$ to be the degree $d$, depth $k$ part of $a$. In the next theorem we will choose a subatom selection function $s\colon A_I\to A_{I'}$  (recall \cref{sas}) such that $s(a)$ is regular for all $a\in A_I$ except those with $a_{1,0}=0$.

\begin{thm}[Subatom selection]
\label{subatom-selection}
Fix a prime $p$, positive integers $R,c_0$, a parameter $\zeta>0$, non-increasing functions $\eta, \theta\colon\Zp\times\Zp\to(0,1)$, and non-decreasing functions $d, r\colon\Zp\times\Zp\to\Zp$. There exist constants $C_{reg}(p, R, c_0,\zeta, \eta, \theta, d, r)$ and $D_{reg}(p, R, c_0,\zeta, \eta, \theta, d, r)$ and $n_{reg}(p,c_0,\zeta)$ such that the following holds. Let $V$ be a finite-dimensional $\F_p$-vector space satisfying $\dim V\geq n_{reg}(p,c_0,\zeta)$. Given functions $f^{(1)},\ldots,f^{(R)}\colon V\to[0,1]$, there exist a polynomial factor $\fB$ and a refinement $\fB'$ both on $V$ with parameters $I$ and $I'$ with the following properties. There exists a subatom selection function $s\colon A_I\to A_{I'}$  and a decomposition \[f^{(\ell)}=f_{str}^{(\ell)}+f_{psr}^{(\ell)}+f_{sml}^{(\ell)}\] for each $\ell\in[R]$ such that:
\begin{enumerate}[(i)]
    \item $f^{(\ell)}_{str}=\E[f^{(\ell)}|\fB']$ for each $\ell\in[R]$;
    \item $\|f^{(\ell)}_{psr}\|_{U^{d(\deg\fB,\|\fB\|)+1}}< \eta(\deg\fB',\|\fB'\|)$ for each $\ell\in[R]$;
    \item $f_{str}^{(\ell)}$ and $f_{str}^{(\ell)}+f_{sml}^{(\ell)}$ have range $[0,1]$ and $f_{psr}^{(\ell)}$ and $f_{sml}^{(\ell)}$ have range $[-1,1]$ for each $\ell\in[R]$;
    \item $\rank\fB \geq r(\deg\fB,\|\fB\|)$ and $\rank \fB'\geq r(\deg\fB',\|\fB'\|)$;
    \item for each $a\in A_I$ with $a_{1,0}\neq 0$, it holds that \[\|f^{(\ell)}_{sml}1_{\fB'^{-1}(s(a))}\|_2<\theta(\deg\fB, \|\fB\|)\|1_{\fB'^{-1}(s(a))}\|_2\]for each $\ell\in[R]$;
    \item for all but at most a $\zeta$-fraction of $a\in A_I$ it holds that \[\left|\E_{x\in \fB^{-1}(a)}[f^{(\ell)}(x)]-\E_{x\in\fB'^{-1}(s(a))}[f^{(\ell)}(x)]\right|<\zeta\]for each $\ell\in[R]$;
    \item $I_{1,0}\geq c_0$;
    \item $\|\fB'\|\leq C_{reg}(p, R, c_0, \zeta, \eta, \theta, d, r)$ and $\deg\fB' \leq  D_{reg}(p, R, c_0, \zeta, \eta, \theta, d, r)$.
\end{enumerate}
\end{thm}

\begin{proof}
Define $n_{reg}(p,c_0,\zeta):=\max\{c_0,\lceil{\log_p(2/\zeta)}\rceil\}$. Let $\fB_0$ be a polynomial factor on $V$ defined by $n_{reg}(p,c_0,\zeta)$ linearly independent linear functions. Define $\theta'\colon\Zp\times\Zp\to\Zp$ by $\theta'(D,N):=\theta(D,N)/(2\sqrt{R}N)$.
We apply \cref{reg-3} to $\fB_0$ with parameters $p,R,\|\fB_0\|,1,\zeta/4,\eta,\theta',d,r$. Let $\fB$ and $\fB'$ be the polynomial factors produced. We immediately have conclusions (i), (ii), (iii), and (iv). Conclusion (vii) follows since $\fB$ refines $\fB_0$ which is defined by at least $c_0$ linear forms. Conclusion (viii) holds by defining $C_{reg}(p,R,c_0,\zeta,\eta,\theta,d,r)=C_{reg'''}(p,R,n_{reg}(p,c_0,\zeta),1,\zeta/4,\eta,\theta',d,r)$ and  $D_{reg}(p,R,c_0,\zeta,\eta,\theta,d,r)=D_{reg'''}(p,R,p^{n_{reg}(p,c_0,\zeta)},1,\zeta/4,\eta,\theta',d,r)$.

Let $c_{d,k}^{i,j}\in\Z/p^{k+1}\Z$ be elements chosen independently and uniformly at random for each $(d,k)\in D_p$ and $i\in [I_{1,0}]$ and $I_{d,k}<j\leq I'_{d,k}$. Consider the subatom selection function $s_{\bm c}\colon A_I\to A_{I'}$ defined in \cref{sas}. We claim that with positive probability, this $s_{\bm c}$ satisfies conclusions (v) and (vi). 

Fix $a\in A_I$ with $a_{1,0}\neq 0$. We first claim that for this fixed $a$, as $\bm c$ varies, the subatom $s_{\mathbb c}(a)$ is uniformly distributed over $\pi^{-1}(a)\subset A_{I'}$. To see this, first note that the univariate homogeneous non-classical polynomials $P_{d,k}\colon\F_p\to\U_{k+1}$ used in the definition of $s_{\bm c}$ satisfy $P_{d,k}(x)\not\in\U_k$ for all $x\neq 0$. This follows by homogeneity: if $P_{d,k}(x)\in U_k$ for some $x\neq 0$, then $P_{d,k}$ always takes values in $\U_k$, contradicting the assumption that $P_{d,k}$ has depth exactly $k$. Thus for $a\in A_I$ with $a_{1,0}\neq 0$, we find that the vector $(P_{d,k}(|a_{1,0}^1|),\ldots,P_{d,k}(|a_{1,0}^{I_{1,0}}|))\in\U_{k+1}^{I_{1,0}}$ does not lie in $\U_k^{I_{1,0}}$.

Considering the definition of $s_{\bm c}$, we see that the vector $([s_{\bm c}]^i_{d,k})_{I_{d,k}<i\leq I'_{d,k}}$ is produced by the following matrix multiplication.
\[
\begin{pmatrix}
[s_{\bm c}]^{I_{d,k}+1}_{d,k} \\
[s_{\bm c}]^{I_{d,k}+2}_{d,k} \\
\vdots \\
[s_{\bm c}]^{I'_{d,k}}_{d,k}
\end{pmatrix}
=
\begin{pmatrix}
c^{I_{d,k}+1,1}_{d,k} & c^{I_{d,k}+1,2}_{d,k} & \cdots & c^{I_{d,k}+1,I_{d,k}}_{d,k} \\
c^{I_{d,k}+2,1}_{d,k} & c^{I_{d,k}+2,2}_{d,k} & \cdots & c^{I_{d,k}+2,I_{d,k}}_{d,k} \\
\vdots & \vdots & \ddots & \vdots \\
c^{I'_{d,k},1}_{d,k} & c^{I'_{d,k},2}_{d,k} & \cdots & c^{I'_{d,k},I_{d,k}}_{d,k} 
\end{pmatrix}
\begin{pmatrix}
P_{d,k}(|a^1_{1,0}|) \\
P_{d,k}(|a^2_{1,0}|) \\
\vdots \\
P_{d,k}(|a^{I_{d,k}}_{1,0}|)
\end{pmatrix}
\]
As stated above, the vector on the right lies in $\U_{k+1}^{I_{1,0}}$ but not in $\U_k^{I_{1,0}}$ while the matrix is uniform random with entries in $\Z/p^{k+1}\Z$. This implies that the vector on the left is uniformly distributed over $\U_{k+1}^{I'_{d,k}-I_{d,k}}$, as desired. This holds for each $(d,k)\in D_p$ and furthermore, since for each $(d,k)$ the $c$ matrices are chosen independently, one can see that the resulting vectors are also independent, proving the desired result.

Now note that by \cref{reg-3}(vi), for each $\ell\in[R]$,
\[\sum_{a\in A_{I'}}\|f^{(\ell)}_{sml}1_{\fB'^{-1}(a)}\|_2^2\leq\frac{\theta(\deg\fB, \|\fB\|)^2}{4R\|\fB\|^2}\sum_{a\in A_{I'}}\|1_{\fB'^{-1}(a)}\|_2^2.\]
Thus by Markov's inequality, for $a'\in A_{I'}$ chosen uniformly at random, with probability at least $1-1/(4R\|\fB\|^2)$ the following holds
\begin{equation}
\label{good-atom}
\|f^{(\ell)}_{sml}1_{\fB'^{-1}(a')}\|_2^2\leq\theta(\|\fB\|)^2\|1_{\fB'^{-1}(a')}\|_2^2.
\end{equation}
Thus for a fixed $a\in A_I$ with $a_{1,0}\neq 0$, we have that for $\bm c$ chosen at random, we have $s_{\bm c}(a)$ satisfies \cref{good-atom} with probability at least $1-1/(4R\|\fB\|)$. Taking a union bound over all choices of $a\in A_I$ and $\ell\in[R]$, we see that with probability at least $3/4$, conclusion (v) holds.

Now we deduce conclusion (vi). First note that the fraction of $a\in A_I$ which satisfy $a_{1,0}=0$ is $p^{-I_{1,0}}\leq\zeta/2$. For the other $a\in A_I$, the expected fraction of $a\in A_I$ that fail to satisfy the desired inequality is at most $\zeta/4$. Thus by Markov's inequality, with probability at least $1/2$, at most $\zeta/2$ fraction of $a\in A_I$ satisfy $a_{1,0}\neq 0$ and fail to satisfy the desired inequality. Thus with probability at least $1/2$, conclusion (vi) holds.
\end{proof}


\section{Patching}
\label{sec-patching}

To motivate the kind of results proved in this section consider the following result, which follows from an application of Ramsey's theorem.

\begin{quote}
Let $\cH$ be a finite set of red/blue edge-colored graphs. There exists an integer $n_0=n_0(\cH)$ such that the following holds. Either:
\begin{enumerate}[(a)]
\item either the all-red coloring of $K_n$ or the all-blue coloring of $K_n$ contains no subgraph from $\cH$ for every $n$; or
\item every 2-edge-coloring of $K_n$ with $n\geq n_0$ contains a subgraph from $\cH$.
\end{enumerate}
\end{quote}

We call such a statement a dichotomy result. The first main result of this section, \cref{dichotomy}, is a dichotomy result for our setting. In our setting we consider colored labeled patterns instead of edge-colored subgraphs and instead of monochromatic colorings we have to consider so-called canonical colorings, defined below.

The second main result of this section, \cref{patching}, is our patching result, which is a supersaturation version of the dichotomy result.

\begin{defn}
For a prime $p$, a finite set $\cS$, and a parameter list $I\in\cI_p$, an \textbf{$\cS$-colored $I$-labeled pattern} over $\F_p$ consisting of $m$ linear forms in $\ell$ variables is a triple $(\bm L,\psi,\phi)$ given by:
\begin{itemize}
    \item a system $\bm L=(L_1,\ldots,L_m)$ of $m$ linear forms in $\ell$ variables,
    \item a coloring $\psi\colon [m]\to\cS$, and
    \item a labeling $\phi\colon[m]\to A_I$ (recall the definition of the atom-indexing set $A_I$ from \cref{A}).
\end{itemize}
Given a finite-dimensional $\F_p$-vector space $V$, a function $f\colon V\to\cS$, and a polynomial factor $\fB$ on $V$ with parameters $I$, an \textbf{$(\bm L,\psi,\phi)$-instance} in $(f,\fB)$ is some $\bx\in V^\ell$ such that $f(L_i(\bx))=\psi(i)$ for all $i\in [m]$ and $\fB(L_i(\bx))=\phi(i)$ for all $i\in[m]$. An instance is called \textbf{generic} if $x_1,\ldots,x_\ell$ are linearly independent.
\end{defn}

\begin{defn}
\label{colored-labeled-pattern}
For an $\cS$-colored $I$-labeled pattern $(\bm L,\psi,\phi)$ consisting of $m$ linear forms, a finite dimensional $\F_p$-vector space $V$, a function $f\colon V\to\cS$, and a polynomial factor $\fB$ on $V$ with parameters $I$, define the \textbf{$(\bm L,\psi,\phi)$-density} in $(f,\fB)$ to be
\[\Lambda_{\bm L}(1_{f^{-1}(\psi(1))\cap\fB^{-1}(\phi(1))},\ldots,1_{f^{-1}(\psi(m))\cap\fB^{-1}(\phi(m))}).\]
Given a set $X\subseteq V$, define the \textbf{relative density of $(\bm L,\psi,\phi)$ in $X$} to be
\[\frac{\Lambda_{\bm L}\paren{f_1,\ldots,f_m}}{\Lambda_{\bm L}(1_X,\ldots,1_X)}\] where $f_i:=1_{X\cap f^{-1}(\psi(i))\cap \fB^{-1}(\phi(i))}$.
\end{defn}

\begin{defn}
Define the \textbf{first non-zero coordinate} function $\fnz\colon \F_p^n\to\F_p$ by $\fnz(0,\ldots,0):=0$ and $\fnz(x_1,\ldots,x_n):=x_k$ where $x_1=\cdots=x_{k-1}=0$ and $x_k\neq 0$. Given a finite-dimensional $\F_p$-vector space $V$ equipped with an isomorphism $\iota\colon V\xrightarrow{\sim}\F_p^n$, define the function $\fnz_{\iota}\colon V\to\F_p$ by $\fnz_{\iota}(x):=\fnz(\iota(x))$.
\end{defn}

\begin{defn}
Fix a prime $p$, a finite set $\cS$, a parameter list $I\in\cI_p$, and a function $\xi\colon\F_p\times A_I\to\cS$. For a finite-dimensional $\F_p$-vector space $V$ equipped with an isomorphism $\iota\colon V\xrightarrow{\sim}\F_p^n$ and a polynomial factor $\fB$ on $V$ with parameters $I$, define the \textbf{$\xi$-canonical coloring} $\Xi_{\xi,\iota,\fB}\colon V\to\cS$ by $\Xi_{\xi,\iota,\fB}(x):=\xi(\fnz_{\iota}(x),\fB(x))$. Furthermore, if $\cS$ is equipped with an $\Fpx$-action, say that $\xi$ is \textbf{projective} if the same is true for every function $\Xi_{\xi,\iota,\fB}$. (Note that this property is equivalent to the condition that $\xi$ preserves the action of $\Fpx$, i.e., $\xi(cx,c\cdot a)=c\cdot\xi(x,a)$ for all $c\in\Fpx$, all $x\in\F_p$, and all $a\in A_I$. Recall the action of $\Fpx$ on $A_I$ defined in \cref{A-action}.)
\end{defn}

\begin{defn}
Given a prime $p$, a finite set $\cS$, a parameter list $I\in\cI_p$, a function $\xi\colon\F_p\times A_I\to\cS$, and a $\cS$-colored $I$-labeled pattern $H=(\bm L,\psi,\phi)$, say that \textbf{$\xi$ canonically induces $H$} if the following holds. There exists some $n\geq 0$ and a polynomial factor $\fB$ on $\F_p^n$ with parameters $I$ such that there exists a generic $H$-instance in $(\Xi_{\xi,\Id,\fB},\fB)$. For a finite set of $\cS$-colored $I$-labeled patterns $\cH$, say that \textbf{$\xi$ canonically induces $\cH$} if $\xi$ canonically induces some $H\in\cH$.
\end{defn}

It is not hard to show that if $\xi$ canonically induces $H$, then there exists a generic $H$-instance in $(\Xi_{\xi,\iota,\fB},\fB)$ for every $V,\iota,\fB$ as long as $\dim V$ and $\rank\fB$ are large enough. Our first result is a strengthening of this: if every $\xi$ canonically induces $H$, then there exists a generic $H$-instance in $(f,\fB)$ for every $f\colon V\to\cS$ and every $\fB$ as long as $\dim V$ and $\rank\fB$ are large enough.

\begin{thm}[Dichotomy]
\label{dichotomy}
Fix a prime $p$, a finite set $\cS$ with an $\Fpx$-action, a parameter list $I\in\cI_p$, and a positive integer $\ell_0$. There exist constants $n_{dich}=n_{dich}(p,|\cS|,I,\ell_0)$ and $r_{dich}=r_{dich}(p,|\cS|,I,\ell_0)$ such that the following holds. Let $\cH$ be a finite set of $\cS$-colored, $I$-labeled patterns each defined by a system of linear forms in at most $\ell_0$ variables. Either:
\begin{enumerate}[(a)]
    \item there exists a projective $\xi\colon\F_p\times A_I\to\cS$ that does not canonically induce $\cH$; or
    \item for every finite-dimensional $\F_p$-vector space $V$ satisfying $\dim V\geq n_{dich}$, every projective function $f\colon V\to\cS$, and every polynomial factor $\fB$ on $V$ with parameters $I$ which has rank at least $r_{\text{dich}}$, there is a generic $H$-instance in $(f,\fB)$ for some $H\in\cH$.
\end{enumerate}
\end{thm}

\begin{proof}
Define $m$ to be the smallest positive integer such that the following holds. Let $H$ be a $\cS$-colored, $I$-labeled patterns defined by a system of linear forms in at most $\ell_0$ variables and let $\xi\colon\F_p\times A_I\to\cS$ be a projective function. If $\xi$ canonically induces $H$, then there exists some $n_H\leq m$ and a polynomial factor $\fB_H$ on $\F_p^{n_H}$ with parameters $I$ such that there exists a generic $H$-instance in $(\Xi_{\xi,\Id,\fB_H},\fB_H)$. This is well defined since there are only a finite number of $\cS$-colored, $I$-labeled patterns defined by a system of linear forms in at most $\ell_0$ variables.

\begin{lemma}
\label{ramsey}
Fix a prime $p$, a finite set $\cS$, a parameter list $I\in\cI_p$, and positive integers $m,r_0,n_0$. There exist constants $n_{ramsey}=n_{ramsey}(p,|\cS|,I,m,r_0,n_0)$ and $r_{ramsey}=r_{ramsey}(p,|\cS|,I,m,r_0,n_0)$ such that the following holds. Let $V$ be a finite dimensional $\F_p$-vector space satisfying $\dim V\geq n_{ramsey}$, let $\fB$ be a polynomial factor on $\F_p^n$ with parameters $I$ such $\rank\fB\geq r_{ramsey}$, and let $f\colon V\to\cS$ be a function. Then there exists a subspace $U\leq V$, and linear functions $P_1,\ldots,P_m\colon V\to\F_p$, and a function $\xi\colon\F_p\times A_I\to\cS$ such that the following holds:
\begin{enumerate}[(i)]
    \item $\xi(\fnz(P_1(x),\ldots,P_m(x)),\fB(x))=f(x)$ for all $x\in U$ that also satisfy $(P_1(x),\ldots,P_m(x))\neq(0,\ldots,0)$;
    \item $\xi(0,0)=f(0)$
    \item the polynomial factor $\fB'$ on $U$ defined by the homogeneous non-classical polynomials that define $\fB|_U$ in addition to the polynomials $P_1,\ldots,P_m$ satisfies $\rank\fB'\geq r_0$;
    \item $\dim U\geq n_0$.
\end{enumerate}
\end{lemma}

Let us show how this lemma completes the proof. Define
\[r_0:=r_{equi}\paren{p,\deg I,\frac1{2(p^m\|I\|)^{p^{m}}}}\qquad\text{and}\qquad n_0:=2p^m(m+\ceil{\log_p\|I\|}).\]
Then define
\[n_{dich}(p,|\cS|,I,\ell_0):=n_{ramsey}(p,|\cS|,I,m,r_0,n_0),\]
and
\[r_{dich}(p,|\cS|,I,\ell_0):=r_{ramsey}(p,|\cS|,I,m,r_0,n_0).\]

Let $\cH$ be a finite set of $\cS$-colored, $I$-labeled patterns each defined by a system of linear forms in at most $\ell_0$ variables. Suppose (a) does not hold. Thus for every projective $\xi\colon\F_p\times A_I\to\cS$, there exists an $H\in\cH$ such that $\xi$ canonically induces $H$.

Now we apply \cref{ramsey} to $f\colon V\to\cS$. This produces a subspace $U\leq V$, linear functions $P_1,\ldots,P_m$, and a function $\xi\colon\F_p\times A_I\to\cS$ with several desirable properties.

First note that since $f\colon V\to\cS$ is projective, the same is true of $\xi\colon\F_p\times A_I\to\cS$. Thus by assumption there exists an $H\in\cH$ such that $\xi$ canonically induces $H$. By the choices in the first paragraph, there exists a $n_H\leq m$ and a polynomial factor $\fB_H$ on $\F_p^{n_H}$ such that $(\Xi_{\xi,\Id,\fB_H},\fB_H)$ contains a generic $H$-instance. 

To complete the proof, all we need to show is that there exists an injective linear map $\kappa\colon \F_p^{n_H}\to U$ such that $\fB_H(x)=\fB(\kappa(x))$ for all $x\in \F_p^{n_H}$ and $\fnz(x)=\fnz(P_1(\kappa(x)),\ldots,P_m(\kappa(x)))$ for all $x\in \F_p^{n_H}$. This follows by an application of equidistribution (\cref{equidistribution}).

Recall the definition of $\bm L^{n_H}$ (\cref{L}), the system of $p^{n_H}$ linear forms in $n_H$ variables that define an $n_H$-dimensional subspace. 

Say that $\fB'$ has parameters $I'$. Thus the atom-indexing set of $\fB'$ can be written as  $A_{I'}\simeq \F_p^m\times A_I$. We define the following tuple of atoms $\bm a=(a_{\bm i})_{i\in \F_p^{n_H}}$ by $a_{\bm i}=((i_1,\ldots,i_{n_H},0,\ldots,0),\fB_H(\bm i))$ where there are $m-n_H$ 0's. We claim that $\bm a\in \Phi_{I'}(\bm L^{n_H})$. We can check this separately for the first and second coordinate; each is trivial.

Thus by \cref{equidistribution} and the rank bound on $\fB'$, we have
\[\Pr_{x_1,\ldots,x_{n_H}\in U}\paren{\fB'(L^{n_H}_{\bm i}(x_1,\ldots,x_{n_H}))=a_{\bm i}\text{ for all }\bm i\in\F_p^{n_H}}\geq \frac{1}{|\Phi_{I'}(\bm L^{n_H})|}-\frac1{2(p^m\|I\|)^{p^m}}\geq\frac1{2(p^m\|I\|)^{p^m}}.\]
We wish to find a single tuple $(x_1,\ldots,x_{n_H})\in V^{n_H}$ that satisfies the above condition and also has $x_1,\ldots,x_{n_H}$ linearly independent. The number of linearly dependent tuples is small, so we calculate that the number of good tuples is at least\[\frac{|U|^{n_H}}{2(p^m\|I\|)^{p^m}}-|U|^{n_H-1}p^{n_H}.\]
This is positive by our assumption that $\dim U\geq n_0$. Thus there exists some good tuple $(x_1,\ldots,x_{n_H})\in V^{n_H}$. Defining $\kappa\colon \F_p^{n_H}\to U$ by $\kappa(\bm i):=L^{n_H}_{\bm i}(x_1,\ldots,x_{n_H})$ has all the desired properties. Thus we have shown (b) assuming that (a) does not hold.
\end{proof}

\begin{proof}[Proof of \cref{ramsey}]
Define $M:=m|\cS|^{p\|I\|}$. Our strategy is to find a large subspace $U_M$ and linear functions $P_1,\ldots,P_M\colon V\to \F_p$ such that for $x\in U_M$, the value of $f(x)$ only depends on $\fB(x)$, $\fnz(P_1(x),\ldots,P_M(x))$, and the index $k$ such that $P_1(x)=\cdots=P_{k-1}(x)=0$ and $P_k(x)\neq 0$. Once we have found such a configuration, we can complete the proof by a simple pigeonhole argument. 

Define \[\cL:=\{p^sP_{(p-1)s+1,s}^i:s\geq0,i\in[I_{(p-1)s+1,s}]\},\]where $P_{d,k}^i$ is the $i$th homogeneous non-classical polynomial of degree $d$ and depth $k$ defining $\fB$. Note that $\cL$ is a finite set of linear functions (in particular, it is the set of all $p^sP$ that are classical linear polynomials where $s\geq 0$ is a non-negative integer and $P$ is one of the homogeneous non-classical polynomials that define $\fB$.) It is immediate from the definition of rank that if $P_1,\ldots,P_m$ are linear functions such that $\{P_1,\ldots,P_m\}\cup\cL$ are linearly independent, then $\rank \fB'=\rank\fB$. We will use this fact to guarantee conclusion (iii).

Our main tool is the following lemma which is a Van der Waerden-type result.

\begin{lemma}
\label{ramsey-2}
Fix a prime $p$, a finite set $\cS$, a parameter list $I\in\cI_p$, and positive integers $n_0,r_0$. There exist constants $n_{ramsey'}=n_{ramsey'}(p,|\cS|,I,n_0,r_0)$ and $r_{ramsey'}=r_{ramsey'}(p,|\cS|,I,n_0,r_0)$ such that the following holds. Let $V$ be a finite dimensional $\F_p$-vector space satisfying $\dim V\geq n_{ramsey'}$ and let $P\colon V\to\F_p$ be a non-trivial linear function. Let $\fB$ be a polynomial factor on $\F_p^n$ with parameters $I$ and let $\fB'$ be the common refinement of $\fB$ and $\{P\}$. Suppose that $\rank\fB'\geq r_{ramsey'}$. Let $f\colon V\to\cS$ be a function. Then there exists a subspace $U\leq V$ contained in the zero set of $P$, a vector $z\in V$ such that $P(z)=1$, and a function $\chi\colon A_I\to\cS$ such that the following holds:
\begin{enumerate}[(i)]
    \item $\chi(\fB(x))=f(x)$ for all $x\in z+U$;
    \item $\rank \fB|_U\geq r_0$;
    \item $\dim U\geq n_0$.
\end{enumerate}
\end{lemma}

Define $r_1,\ldots, r_M$ and $n_1,\ldots,n_M$ by \[n_i:=\max\{n_{ramsey'}(p,|\cS|^{p-1},I,n_{i-1},r_{i-1}),|\cL|+1\}\]and\[r_i:=r_{ramsey'}(p,|\cS|^{p-1},I,n_{i-1},r_{i-1}).\]
Then define \[n_{ramsey}(p,|\cS|,I,m,r_0,n_0):=n_M\qquad\text{and}\qquad r_{ramsey}(p,|\cS|,I,m,r_0,n_0):=r_M.\]

We will find nested subspaces $V=U_0\geq U_1\geq \cdots\geq U_M$, linear functions $P_i\colon V\to\F_p$, and functions $\xi_i\colon\Fpz\times A_I\to \cS$ such that the following holds for each $i\in[M]$:
\begin{itemize}
    \item $\xi_i(P_i(x),\fB(x))=f(x)$ for all $x\in U_i$ that satisfy $P_1(x)=\cdots=P_{i-1}(x)=0$ and $P_i(x)\neq 0$;
    \item $\{P_1,\ldots,P_i\}\cup \cL$ are linearly independent;
    \item $\rank \fB|_{W_i}\geq r_{M-i}$ where $W_i:=\{x\in U_i:P_1(x)=\cdots=P_i(x)\}$;
    \item $\dim W_i\geq n_{M-i}$ where $W_i:=\{x\in U_i:P_1(x)=\cdots=P_i(x)\}$.
\end{itemize}

Suppose we have defined $V=U_0\geq U_1\geq\cdots \geq U_i$, linear functions $P_1,\ldots,P_i\colon V\to\F_p$, and functions $\xi_1,\ldots,\xi_i\colon\Fpz\times A_I\to\cS$ with the above properties.

Define $W:=\{x\in U_i:P_1(x)=\cdots=P_i(x)=0\}$. We have $\dim W\geq m_i>|\cL|$. Pick an arbitrary $y\in W$ such $y\neq 0$ but all of the linear functions $\cL$ vanish on $y$. Let $P_{i+1}\colon V\to\F_p$ be an arbitrary linear function such that $P_{i+1}(y)=1$. Note that automatically we have $\{P_1,\ldots,P_i,P_{i+1}\}\cup\cL$ are linearly independent.

Define $W':=\{x\in W:P_{i+1}(x)=0\}$. Note that the subspace $W$ is partitioned into hyperplanes as $W=W'\sqcup(y+W')\sqcup(2y+W')\sqcup\cdots$. Write $\ol\cS:=\cS^{\Fpx}$. Then define $\ol f\colon W\to\ol\cS$ by \[\ol f(x+ty):=(f(bx+by))_{b\in\Fpx}\qquad\text{for }x\in W'\text{ and }t\in\F_p.\]

We apply \cref{ramsey-2} to $\ol f,\fB|_W,P_{i+1}$ with parameters $n_{M-i-1},r_{M-i-1}$ to produce a subspace $U_i'\leq W'$, a vector $z\in y+U_i'$ and a function $\chi\colon A_I\to\ol S$ with several desirable properties.

We have $\chi(\fB(x+z))=\ol f(x+z)$ for all $x\in U_i'$. Looking at the $b$th coordinate of this equation for some $b\in\Fpx$ gives $\chi(\fB(x+z))_a=f(bx+bz)$. Finally using the homogeneity of $\fB$ (recall the action of $\Fpx$ on $A_I$ defined in \cref{A-action}) gives $\chi(b^{-1}\cdot \fB(bx+bz))_a=f(bx+bz)$ for all $x\in U_i'$ and all $b\in\Fpx$.

Define $U_{i+1}\leq U_i$ to be a $(\dim W+i+1)$-dimensional subspace of $U_i$ that contains $z$ and $W$ and such that none of $P_1,\ldots, P_i$ are identically 0 on $U_{i+1}$. Then define $\xi_{i+1}\colon\Fpz\times A_I\to\cS$ by \[\xi_{i+1}(b,a):=\chi(b^{-1}\cdot a)_b.\]
Note $\dim U_{i+1}\geq \dim W'\geq m_{i+1}$ and $\rank\fB|_{U_{i+1}}\geq\rank\fB|_{W'}\geq r_{i+1}$. Furthermore, for $x\in U_{i+1}$ such that $P_1(x)=\cdots=P_{i-1}(x)=0$ and $P_i(x)=b\neq 0$, we can write $x=bx'+bz$ for some $x'\in W'$. Then $f(x)=\ol f(x'+z)_b=\chi(\fB(x'+z))_b=\chi(b^{-1}\cdot\fB(bx'+bz))_b=\xi_{i+1}(b,\fB(x))$, as desired.

Thus we have defined a sequence of nested subspaces $V=U_0\leq\cdots\leq U_M$, linear functions $P_1,\ldots,P_m\colon V\to\F_p$ and functions $\xi_1,\ldots,\xi_M\colon\Fpz\times A_I\to\cS$ with the above properties.

Finally note that the number of possible functions $\Fpz\times A_I\to\cS$ is at most $|\cS|^{p\|I\|}$. Thus by the pigeonhole principle, there exists $1\leq i_1<\cdots< i_m\leq M$ such that $\xi_1=\cdots=\xi_m$. Define $\xi\colon\F_p\times A_I\to \cS$ by $\xi(0,a)=f(0)$ for all $a\in A$ and $\xi(b,a)=\xi_1(b,a)$ for all $b\in\Fpz$ and all $a\in A$. Define $W_M:=\{x\in U_M:P_1(x)=\cdots=P_M(x)=0\}$ and let $U$ be a $(\dim W_M+m)$-dimensional subspace of $V$ that contains $W_M$ and such that none of $P_{i_1},\ldots,P_{i_m}$ are identically 0 on $U$. Then $U$, $P_{i_1},\ldots,P_{i_m}$, and $\xi$ have all the desired properties.
\end{proof}

\begin{proof}[Proof of \cref{ramsey-2}]
Define constants \[n_1:=\max\{n_0,n_{high-rank}(p,I,r_0+p)\},\]and\[\theta:=\frac{1}{8(2|\cS|)^{2p^{n_1}}}\qquad\text{and}\qquad\theta':=\frac{\theta}{\|I\|\sqrt{2p|\cS|}}\]
Define the non-increasing functions $\alpha\colon\Zp\to(0,1)$ by
\[\alpha(N):=\frac1{2N^{2p^{n_1}}},\]
and $\eta\colon\Zp\to(0,1)$ by
\[\eta(N):=\frac{1}{8(3|\cS|N)^{p^{n_1}}},\]
and define the non-decreasing function $r\colon\Zp\to\Zp$ by \[r(N):=r_{equi}(p,p^{n_1},\alpha(N)).\]
Define
\[n_{ramsey'}(p,|\cS|,I,n_0,r_0):=2p^{n_1}\ceil{\log_p(16|\cS|C_{reg''}(p,p^{n_1},|\cS|,p\|I\|,\theta',\eta,r))}\]
and
\[r_{ramsey'}(p,|\cS|,I,n_0,r_0):=r_{reg''}(p,p^{n_1},|\cS|,p\|I\|,\theta',\eta,r)\]

For $c\in\cS$, define $f^{(c)}\colon V\to[0,1]$ by $f^{(c)}:=1_{f^{-1}(c)}$. We are now ready to proceed with the proof. We apply the arithmetic regularity lemma, \cref{reg-2}, to the polynomial factor $\fB'$ on $V$ and the functions $(f^{(c)})_{c\in\cS}$ with parameters $p,p^{n_1},|\cS|,p\|I\|,\theta',\eta,r$. This produces a polynomial factor $\fB''$ refining $\fB'$ and decompositions $f^{(c)}=f_{str}^{(c)}+f_{psr}^{(c)}+f_{sml}^{(c)}$ with several desirable properties.

Say that $\fB'$ has parameters $I'$ and $\fB''$ has parameters $I''$ (note that $\|I'\|=p\|I\|$). We consider the atom-indexing set of $\fB''$ (see \cref{A} for the definition) as $A_{I''}\simeq \F_p\times A_I\times A_{I''-I'}$.

Say that an atom $a\in A_{I''}$ is \emph{regular} if \begin{equation*}
\label{regular}
\|f^{(c)}_{sml}1_{\fB''^{-1}(a)}\|_2\leq\theta\|1_{\fB''^{-1}(a)}\|_2\qquad\text{for all }c\in\cS.
\end{equation*}
Our first goal is to find $s\in A_{I''-I'}$ such that all atoms of the form $(1,a,s)\in A_{I''}\simeq\F_p\times A_I\times A_{I''-I'}$ are regular.

By \cref{reg-2}(v), for each $c\in\cS$,
\[\sum_{a\in A_{I''}}\|f^{(c)}_{sml}1_{\fB''^{-1}(a)}\|_2^2\leq\frac{\theta^2}{2p|\cS|\|I\|^2}\sum_{a\in A_{I''}}\|1_{\fB''^{-1}(a)}\|_2^2.\] Thus at least a $(1-1/(2p\|I\|^2))$-fraction of atoms are regular. For each $a\in A_I$, at least a $(1-1/(2\|I\|))$-fraction of the atoms of the form $(1,a,s)$ are regular, for $s\in A_{I''-I'}$. Thus by a union bound there exists some $s\in A_{I''-I'}$ such that $(1,a,s)$ is regular for all $a\in A_I$. Fix this value of $s$ for the rest of the proof.

Define $\chi\colon A_I\to \cS$ such that $\chi(a)$ is a color that appears in the atom $\fB^{-1}(1,a,s)$ with density at least $1/|\cS|$.

By the definition of $n_1$ and \cref{high-rank-exists}, there exists a polynomial factor $\fB_1$ on $\F_p^{n_1}$ with parameters $I$ and satisfying $\rank\fB_1\geq r_0+p$. Our goal is to find vectors $x_0,x_1,\ldots,x_{n_1}\in V$ such that
\begin{itemize}
    \item $\fB''(x_0+i_1x_1+\cdots+i_{n_1}x_{n_1})=(1,\fB_1(i_1,\ldots,i_{n_1}),s)$ for all $(i_1,\ldots,i_{n_1})\in\F_p^{n_1}$;
    \item $f(x_0+i_1x_1+\cdots+i_{n_1}x_{n_1})=\chi(\fB_1(i_1,\ldots,i_{n_1}))$ for all $(i_1,\ldots,i_{n_1})\in\F_p^{n_1}$;
    \item $x_1,\ldots,x_{n_1}$ are linearly independent.
\end{itemize}

We choose $x_0,x_1,\ldots,x_{n_1}\in V$ independently and uniformly at random. Let $p_1$ be the probability that this choice of $\bm x$ satisfies all three conditions above. First note that the probability that $x_0,\ldots,x_{n_1}$ are linearly dependent is at most $p^{n_1+1}/|V|$. Let $p_2$ be the probability that this choice of $\bm x$ satisfies the first two conditions above. We have shown that\[p_1\geq p_2-p^{n_1+1}/|V|.\]

Let $\bm L=(L_{\bm i})_{\bm i\in\F_p^{n_1}}$ to be the system of $p^{n_1}$ linear forms in $n_1+1$ variables defined by \[L_{\bm i}(x_0,x_1,\ldots,x_{n_1}):=x_0+i_1x_1+\cdots+i_{n_1}x_{n_1}.\] This system defines an $n_1$-dimensional affine subspace. One can easily see that $\bm L$ is finite complexity and in fact its complexity is at most $p^{n_1}$ (see \cref{rem-complexity}).

Define $g^{(\bm i)}:=f^{(\chi(\fB_1(\bm i)))}$ and define $g^{(\bm i)}_{str},g^{(\bm i)}_{sml},g^{(\bm i)}_{psr}$ similarly. Define $h^{(\bm i)}\colon V\to[0,1]$ by $h^{(\bm i)}:=1_{\fB''^{-1}(1,\fB_1(\bm i),s)}$.

We compute $p_2$ as 
\begin{align*}
p_2 & =\E_{\bx}\left[\prod_{\bm i\in\F_p^{n_1}} g^{(\bm i)}(L_{\bm i}(\bx))h^{(\bm i)}(L_{\bm i}(\bx))\right] \\
& =\E_{\bx}\left[\prod_{\bm i}\left( g^{(\bm i)}_{str}(L_{\bm i}(\bx))+g^{(\bm i)}_{sml}(L_{\bm i}(\bx))+g^{(\bm i)}_{psr}(L_{\bm i}(\bx))\right)h^{(\bm i)}(L_{\bm i}(\bx))\right]\\
&\geq \E_{\bx}\left[\prod_{\bm i}\left(g^{(\bm i)}_{str}(L_{\bm i}(\bx))+g^{(\bm i)}_{sml}(L_{\bm i}(\bx))\right)h^{(\bm i)}(L_{\bm i}(\bx))\right]-3^{p^{n_1}}\eta(\|\fB''\|).
\end{align*}
The inequality follows from \cref{reg-2}(ii) and the counting lemma, \cref{counting-lemma}.

Write $p_3$ for the expectation in the last line above. We have $p_2\geq p_3-3^{p^{n_1}}\eta(\|\fB''\|)$. Expanding the product, there are at most $2^{p^{n_1}}$ terms involving $g^{(\bm j)}_{sml}$ for some $\bm j\in\F_p^{n_1}$. Each of these is bounded in magnitude by \[\E_{\bx}\left[\abs{g^{(\bm j)}_{sml}(L_{\bm j}(\bx))}\prod_{\bm i}h^{(\bm i)}(L_{\bm i}(\bx))\right].\] By applying a change of coordinates, we can transform to the case that $\bm j=0$ (and $L_{\bm 0}(\bm x)=x_0$). Then by the Cauchy-Schwarz inequality, the square of the above expression is bounded by\[\E_{x_0}\left[\abs{g^{(0)}_{sml}(x_0)}^2 h^{(0)}(x_0)\right]\E_{x_0}\left[h^{(0)}(x_0)\E_{x_1,\ldots,x_{n_1}}\left[\prod_{\bm i\neq0}h^{(\bm i)}(L_{\bm i}(\bx))\right]^2\right].\]
The first term is at most $\theta(\|\fB\|)^2\|h^{(0)}\|_2^2$ by the fact that $(1,\fB_1(\bm i),s)$ is a regular atom for all $\bm i$. The second term can be counted by equidistribution applied to the system $\bm L'$ of $2p^{n_1}-1$ linear forms in $2n_1-1$ variables defined as follows. Set \[L'_{0}(x_0,x_1,\ldots,x_{n_1},x'_1,\ldots,x'_{n_1}):=x_0,\]
and for $\bm i\in\F_p^{n_1}\setminus\{0\}$, define
\[L'_{\bm i,1}(x_0,x_1,\ldots,x_{n_1},x'_1,\ldots,x'_{n_1}):=x_0+i_1x_1+\cdots+i_{n_1}x_{n_1},\]
\[L'_{\bm i,2}(x_0,x_1,\ldots,x_{n_1},x'_1,\ldots,x'_{n_1}):=x_0+i_1x'_1+\cdots+i_{n_1}x'_{n_1}.\]
By \cref{CS-consistency}, we know that $\|\fB''\|\cdot|\Phi_{I''}(\bm L')|=|\Phi_{I''}(\bm L)|^2$ (see also \cite[Lemma 5.13]{BFHHL13}.)

Thus by equidistribution, \cref{equidistribution}, and the rank bound on $\fB''$, we have the second term is at most
\[\frac{1}{|\Phi_{I''}(\bm L')|}+\alpha(\|\fB''\|)=\frac{\|\fB''\|}{|\Phi_{I''}(\bm L)|^2}+\alpha(\|\fB''\|)\leq\frac{2\|\fB''\|}{|\Phi_{I''}(\bm L)|^2}.\]
Applying equidistribution again we have that the first term is at most
\[\theta^2\paren{\frac{1}{\|\fB''\|}+\alpha(\|\fB''\|)}\leq\frac{2\theta^2}{\|\fB''\|}.\]
Combining these bounds and summing over all terms that contain some $g^{(\bm j)}_{sml}$, we see that
\[p_3\geq \E_{\bx}\left[\prod_{\bm i}g_{str}^{(\bm i)}(L_{\bm i}(\bx))h^{(\bm i)}(L_{\bm i}(\bm x))\right]-2^{p^{n_1}+1}\frac{\theta}{|\Phi_{I''}(\bm L)|}.\]

Write $p_4$ for the expectation in the last line. The quantity $g^{(\bm i)}_{str}(L_{\bm i}(\bm x))$ is the density of $\chi(\fB_1(\bm i))$ in the atom of $\fB''$ that $L_{\bm i}(\bm x)$ lies in. When $\fB(L_{\bm i}(\bm x))=(1,\fB_1(\bm i),s)$, the choice of $\chi$ implies that this density is at least $1/|\cS|$. Thus
\[p_4\geq \frac{1}{|\cS|^{p^{n_1}}}\E_{\bx}\left[\prod_{\bm i}h^{(\bm i)}(L_{\bm i}(\bx))\right].\] 

Write $p_5$ for the expectation in the last line. Unwrapping the definition of $h^{(\bm i)}$, this can be written as \[p_5=\Pr_{\bm x}\paren{\fB''(L_{\bm i}(\bm x))=(1,\fB_1(\bm i),s)\text{ for all }\bm i\in\F_p^{n_1}}.\]
We claim that $((1,\fB_1(\bm i),s))_{\bm i\in\F_p^{n_1}}$ is an $\bm L$-consistent tuple of atoms. We check this coordinate by coordinate. Obviously $(\fB_1(\bm i))_{\bm i\in\F_p^{n_1}}$ is $\bm L$-consistent. Furthermore any constant tuple is also obviously $\bm L$-consistent (this follows from the fact that $\bm L$ is a translation-invariant pattern). Thus by another application of equidistribution and the rank bound on $\fB''$, we have
\[p_5\geq \frac{1}{|\Phi_{I''}(\bm L')|}-\alpha(\|\fB'\|)\geq\frac{1}{2|\Phi_{I''}(\bm L')|}.\]

Combining all these inequalities, we see that \[p_1\geq \frac{1}{|\cS|^{p^{n_1}}}\frac{1}{2|\Phi_{I''}(\bm L')|} - 2^{p^{n_1}+1}\frac{\theta}{|\Phi_{I''}(\bm L)|}-3^{p^{n_1}}\eta(\|\fB''\|)-\frac{p^{n_1+1}}{|V|}.\]
This expression is positive by the definition of $\theta,\eta$ and the assumption that $\dim V\geq n_{ramsey'}$.

Thus we have define a function $\chi\colon A_I\to\cS$ and found linearly independent $x_0,x_1,\ldots,x_{n_1}\in V$ with several desirable properties. Define $z:=x_0$ and $U:=\spn\{x_1,\ldots,x_{n_1}\}$. Since $P\colon V\to\F_p$ is a linear function and $P(x_0+i_1x_1+\cdots+i_{n_1}x_{n_1})=1$ for all $\bm i\in\F_p^{n_1}$, we conclude that $P(z)=P(x_0)=1$ and $U$ is contained in the zero set of $P$. Furthermore, $\fB(x_0+i_1x_1+\cdots+i_{n_1}x_{n_1})=\fB_1(i_1,\ldots,i_{n_1})$ for all $\bm i\in\F_p^{n_1}$. Since $\fB_1$ was chosen such that $\rank\fB_1\geq r_0+p$ and $n_1\geq n_0$, we see that $\dim U\geq n_0$ and $\rank \fB|_{\spn\{x_0,U\}}\geq r_0+p$. By \cref{hyperplane-rank}, we conclude that $\rank\fB|_U\geq r_0$. Finally, $f(x_0+i_1x_1+\cdots+i_{n_1}x_{n_1})=\chi(\fB_1(i_1,\ldots,i_{n_1}))=\chi(\fB(x_0+i_1x_1+\cdots+i_{n_1}x_{n_1}))$ for all $\bm i\in\F_p^{n_1}$, which proves the desired result.
\end{proof}

We boost this Ramsey dichotomy result to a density result using a Cauchy-Schwarz supersaturation argument. For technical reasons we need this result to hold inside ``subvarieties'' of a vector space, i.e., the zero sets of a sufficiently high-rank collection of non-classical polynomials. Again for technical reasons, this supersaturation argument only works for full dimensional patterns (recall \cref{full-dimensional}).

\begin{thm}[Patching]
\label{patching}
Fix a prime $p$, a finite set $\cS$ with an $\Fpx$-action, parameter lists $I,I'\in\cI_p$ satisfying $I\leq I'$, and a positive integer $\ell_0$. There exist constants $n_{patch}=n_{patch}(p,|\cS|, I',\ell_0)$ and $\beta_{patch}=\beta_{patch}(p,|\cS|,I,\ell_0)>0$ and a non-decreasing function $r_{patch} = r_{patch}(p, |\cS|, I, \ell_0) \colon \Zp\times\Zp \to \Zp$ such that the following holds.  Let $\cH$ be a finite set of $\cS$-colored, $I$-labeled patterns such that each pattern is defined by a full dimension system of linear forms in at most $\ell_0$ variables (recall \cref{full-dimensional}). Either:
\begin{enumerate}[(a)]
    \item there exists a projective $\xi\colon\F_p\times A_I\to\cS$ that does not canonically induce $\cH$; or
    \item for every finite-dimensional $\F_p$-vector space $V$ satisfying $\dim V\geq n_{patch}$, every projective function $f\colon V\to\cS$, every polynomial factor $\fB$ on $V$ with parameters $I$ that satisfies $\rank\fB\geq r_{\text{patch}}(\deg\fB,\|\fB\|)$, and every polynomial factor $\fB'$ on $V$ with parameters $I'$ that refines $\fB$ and satisfies $\rank\fB'\geq r_{patch}(\deg\fB',\|\fB'\|)$, there is a pattern $H\in\cH$ such that in $(f,\fB)$, the relative density of $H$ in $\fB'^{-1}(A_I\times\{0\})$ is at least $\beta_{patch}$.
\end{enumerate}
\end{thm}

Note that $n_{patch}$ may depend on $I'$, but critically $\beta_{patch}$ does not depend on $I'$.

\begin{proof}
First we define several parameters.

Write $r:=r_{dich}(p,|\cS|,I,\ell_0)$ for brevity. Define \[n_1:=\max\{n_{dich}(p,|\cS|,I,\ell_0), n_{high-rank}(p,I,r)\}\] where $n_0$ is defined in \cref{high-rank-exists}.

Define the constants\[n_{patch}(p,|\cS|, I',\ell_0):=2p^{n_1}\ceil{\log_p(\|I'\|)},\] and
\[\beta_{patch}(p,|\cS|,I,\ell_0):=\frac1{6400\ell_0^2p^{2n_1\cdot \ell}\|I\|^{2p^{n_1}}},\]
and define the non-increasing function $\alpha\colon\Zp\to(0,1)$ by \[\alpha(N):=\frac{1}{2N^{2p^{n_1}}},\]
and the non-decreasing function $r_{patch}(p,|\cS|,I,\ell_0)\colon\Zp\times\Zp\to\Zp$ by \[r_{patch}(p,|\cS|,I,\ell_0)(D,N):=r_{equi}(p,D,\alpha(N)).\]

We now proceed to the proof.  Let $\cH$ be a finite set of $\cS$-colored, $I$-labeled patterns each defined by a full dimension system of linear forms in at most $\ell_0$ variables. We apply \cref{dichotomy} to $\cH$. If \cref{dichotomy}(a) holds, then clearly conclusion (a) holds. Now assume that \cref{dichotomy}(b) holds. We wish to show conclusion (b).

Let $V$ be a finite-dimensional $\F_p$-vector space satisfying $\dim V\geq n_{patch}$, let $f\colon V\to\cS$ be a projective function, let $\fB$ be a polynomial factor on $V$ with parameters $I$ that satisfies $\rank\fB\geq r_{patch}(\|\fB\|)$, and let $\fB'$ be a polynomial factor on $V$ with parameters $I'$ that refines $\fB$ and satisfies $\rank \fB'\geq r_{patch}(\|\fB'\|)$.

Write $X:= \fB'^{-1}(A_I\times\{0\})$ and $n:=\dim V$. By assumption, $n\geq n_{patch}$. We wish to count $\cH$ instances in $(f,\fB)$ that are contained in $X$. Write $\cH=\bigsqcup_{\ell=1}^{\ell_0}\cH_\ell$ where $\cH_\ell$ is defined to be the subset of $\cH$ consisting of colored labeled patterns defined by a system of linear forms in exactly $\ell$ variables. We define sets $\fU_1,\ldots,\fU_{\ell_0}$ as follows. $\fU_\ell$ is the set of $\ell$-dimensional subspaces $U$ of $V$ which satisfy the following:
\begin{itemize}
    \item $U\subseteq X$;
    \item there exists a colored labeled pattern $H\in\cH_\ell$ in $\ell$ variables and a generic $H$-instance $x_1,\ldots,x_\ell\in U$.
\end{itemize}

Note that the requirement that $x_1,\ldots,x_{\ell}\in U$ are generic implies that $U=\spn\{x_1,\ldots,x_\ell\}$. Thus $\sum_{\ell=1}^{\ell_0}|\fU_\ell|$ is a lower bound on the number of $\cH$-instances in $X$.

Define the following counting function $c\colon\fU_\ell\to\Znn$ for each $\ell\in[\ell_0]$ as follows. For $U\in\fU_\ell$, let $c(U)$ be the number of $n_1$-dimensional subspaces that contain $U$ and are contained in $X$. An application of the Cauchy-Schwarz inequality implies that
\begin{equation}
\label{density-bound}
|\fU_\ell|\geq\frac{\paren{\sum_{U\in\fU_\ell} c(U)}^2}{\sum_{U\in\fU_\ell} c(U)^2}.
\end{equation}

Define $S_1$ to be the number of $n_1$-dimensional subspaces $W$ of $V$ such that $W\subseteq X$ and $\rank(\fB|_W)\geq r$. By \cref{dichotomy}(b), every such $W$ contains a generic $\cH$-instance. Thus \[\sum_{\ell=1}^{\ell_0}\sum_{U\in\fU_\ell}c(U)\geq S_1.\] By the pigeonhole principle, there exists some $\ell\in[\ell_0]$ such that
\begin{equation}
\label{lower-bound-1}
\sum_{U\in\fU_\ell}c(U)\geq \frac{S_1}{\ell_0}.
\end{equation}
We fix such a value of $\ell\in[\ell_0]$ for the rest of the proof.

Define $S_2$ to be the number of ordered $n_1$-tuples $(x_1,\ldots,x_{n_1})\in V^{n_1}$ such that $x_1,\ldots,x_{n_1}$ are linearly independent, $\spn\{x_1,\ldots,x_{n_1}\}\subseteq X$, and $\rank(\fB|_{\spn\{x_1,\ldots,x_{n_1}\}})\geq r$. We can compute
\begin{equation}
\label{lower-bound-2}
S_2=S_1\prod_{i=0}^{n_1-1}(p^{n_1}-p^i)\leq p^{n_1^2}S_1.
\end{equation}

Define $S_3$ to be the number of ordered $n_1$-tuples $(x_1,\ldots,x_{n_1})\in V^{n_1}$ such that $\spn\{x_1,\ldots,x_{n_1}\}\subseteq X$ and $\rank(\fB|_{\spn\{x_1,\ldots,x_{n_1}\}})\geq r$. We can easily bound
\begin{equation}
\label{lower-bound-3}
S_2\geq S_3-p^{n\cdot n_1}p^{n_1-n}.
\end{equation}

By the definition of $n_1$ and \cref{high-rank-exists}, there exists a polynomial factor $\fB_1$ on $\F_p^{n_1}$ with parameters $I$ and rank at least $r$. Define $S_4$ to be the number of ordered $n_1$-tuples $(x_1,\ldots,x_{n_1})\in V^{n_1}$ such that $\fB_1(i_1,\ldots,i_{n_1})=\fB(i_1x_1+\cdots+i_{n_1}x_{n_1})$ and $\spn\{x_1,\ldots,x_n\}\subseteq X$. Notice that $S_3\geq S_4$.

Write $\bm L':=\bm L^{n_1}$, the system of $p^{n_1}$ linear forms in $n_1$ variables that define an $n_1$-dimensional subspace (see \cref{L}).
By definition, $(P^i_{d,k}(i_1,\ldots,i_{n_1}))_{\bm i\in\F_p^{n_1}}\in\Phi_{d,k}(\bm L')$ for every $(d,k)\in D_p$ and $i\in[I_{d,k}]$ where $P^i_{d,k}$ is the $i$th non-classical polynomial of degree $d$ and depth $k$ defining $\fB_1$. Also, $(0,\ldots,0)\in\Phi_{d,k}(\bm L')$ for all $(d,k)\in D_p$. Define $\bm a\in A_{I'}^{\F_p^{n_1}}$ by
\[(a_{\bm i})_{d,k}^i =
\begin{cases}
P^i_{d,k}(i_1,\ldots,i_{n_1}) \quad & \text{if }i\leq I_{d,k}, \\
0 &\text{if }I_{d,k}<i\leq I'_{d,k}.
\end{cases}
\]
By the above discussion, $\bm a$ is $\bm L'$-consistent, so by \cref{equidistribution} and the rank assumption of $\fB'$, we find
\begin{equation}
\label{lower-bound-4}
\begin{split}
S_3\geq S_4
& =\abs{\{ \bm x\in V^{n_1}:\fB'(L'_{\bm i}(\bm x))=a_{\bm i}\text{ for all }\bm i\in\F_p^{n_1}\}} \\
& \geq\paren{\frac{1}{|\Phi_{I'}(\bm L')|}-\alpha(\|I'\|)}p^{n\cdot n_1}\geq\frac{p^{n\cdot n_1}}{2|\Phi_{I'}(\bm L')|}.
\end{split}
\end{equation}

Combining \cref{lower-bound-1}, \cref{lower-bound-2}, \cref{lower-bound-3}, and \cref{lower-bound-4}, we conclude
\begin{equation}
\label{lower-bound}
\sum_{U\in\fU_\ell}c(U)\geq \frac1{\ell_0 p^{n_1^2}}p^{n\cdot n_1}\paren{\frac1{2|\Phi_{I'}(\bm L')|}-p^{n_1-n}}\geq\frac{1}{4\ell_0p^{n_1^2}|\Phi_{I'}(\bm L')|}p^{n\cdot n_1}.
\end{equation}

Next we find an upper bound on $\sum_{U\in\fU_\ell} c(U)^2$. Define $T_1$ to be the number of triples $(U,W,W')$ where $U$ is a $\ell$-dimensional subspace of $V$ and $W,W'$ are both $n_1$-dimensional subspaces of $V$ that contain $U$ and are contained in $X$. First note that
\begin{equation}
\label{upper-bound-1}
\sum_{U\in\fU_\ell}c(U)^2\leq T_1.
\end{equation}

Now define $T_2$ to be the number of ordered $(2n_1-\ell)$-tuples $(x_1,\ldots,x_\ell,y_1,\ldots,y_{n_1-\ell},z_1,\ldots,z_{n_1-\ell})\in V^{2n_1-\ell}$ such that $x_1,\ldots,x_\ell,y_1,\ldots,y_{n_1-\ell}$ are linearly independent, $x_1,\ldots,x_\ell,z_1,\ldots,z_{n_1-\ell}$ are linearly independent, $\spn\{x_1,\ldots,x_\ell,y_1,\ldots,y_{n_1-\ell}\}\subseteq X$, and $\spn\{x_1,\ldots,x_\ell,z_1,\ldots,z_{n_1-\ell}\}\subseteq X$. We compute
\begin{equation}
\label{upper-bound-2}
T_2=\paren{\prod_{i=0}^{\ell-1}(p^\ell-p^i)}\paren{\prod_{i=\ell}^{n_1-1}(p^{n_1}-p^i)}^2 T_1\geq \frac{p^{2n_1(n_1-\ell)}}{100} T_1.
\end{equation}

Next define $T_3$ to be the number of ordered $(2n_1-\ell)$-tuples $(x_1,\ldots,x_\ell,y_1,\ldots,y_{n_1-\ell},z_1,\ldots,z_{n_1-\ell})\in V^{2n_1-\ell}$ such that $\spn\{x_1,\ldots,x_\ell,y_1,\ldots,y_{n_1-\ell}\}\subseteq X$ and $\spn\{x_1,\ldots,x_\ell,z_1,\ldots,z_{n_1-\ell}\}\subseteq X$. Clearly $T_2\leq T_3$.

Define $\bm L''$ to be the following system of $2p^{n_1}-p^\ell$ linearly forms in $2n_1-n_\ell$ variables. For $\bm i\in\F_p^\ell$, define
\[L''_{\bm i}(x_1,\ldots,x_\ell,y_1,\ldots,y_{n_1-\ell},z_1,\ldots,z_{n_1-\ell}):=i_1x_1+\cdots+i_\ell x_\ell.\]
For $\bm i\in\F_p^{n_1}\setminus\paren{\F_p^{\ell}\times\{0\}^{n_1-\ell}}$, define
\begin{align*}
L''_{\bm i,1}(x_1,\ldots,x_\ell,y_1,\ldots,y_{n_1-\ell},z_1,\ldots,z_{n_1-\ell})
& :=i_1x_1+\cdots+i_\ell x_\ell+i_{\ell+1}y_1+\cdots+i_{n_1}y_{n_1-\ell}, \\
L''_{\bm i,2}(x_1,\ldots,x_\ell,y_1,\ldots,y_{n_1-\ell},z_1,\ldots,z_{n_1-\ell})
& :=i_1x_1+\cdots+i_\ell x_k+i_{\ell+1}z_1+\cdots+i_{n_1}z_{n_1-\ell}.
\end{align*}

Now let $\fB''$ be the polynomial factor on $V$ with parameters $I'-I$ defined by homogeneous non-classical polynomials
\[\paren{P^i_{d,k}}_{\genfrac{}{}{0pt}{}{(d,k)\in D_p}{I_{d,k}<i\leq I_{d,k}'}}\] where $P^i_{d,k}$ is the $i$th non-classical polynomial of degree $d$ and depth $k$ that defines $\fB'$. The important property of this polynomial factor is that $X=\fB''^{-1}(\bm 0)$. Also note that $\rank\fB''\geq\rank\fB'\geq r(\|\fB'\|)$. By \cref{equidistribution} and the bounds on $\rank \fB''$, we find
\begin{equation}
\label{upper-bound-3}
\begin{split}
T_2\leq T_3
& =\abs{\{\bm x\in V^{2n_1-\ell}:\fB''(\bm L''(\bm x))=\bm 0\}} \\ & \leq\paren{\frac{1}{|\Phi_{I'-I}(\bm L'')|}+\alpha(\|I'\|)}p^{n(2n_1-\ell)}\leq\frac{2p^{n(2n_1-\ell)}}{|\Phi_{I'-I}(\bm L'')|}.
\end{split}
\end{equation}

Combining \cref{upper-bound-1}, \cref{upper-bound-2},  and \cref{upper-bound-3}, we conclude
\begin{equation}
\label{upper-bound}
\sum_{U\in\fU_\ell}c(U)^2\leq \frac{200p^{2n_1\cdot \ell}}{p^{2n_1^2}|\Phi_{I'-I}(\bm L'')|}p^{n(2n_1-\ell)}.
\end{equation}

Finally, we combine \cref{density-bound}, \cref{lower-bound}, and \cref{upper-bound} to find
\begin{align*}
|\fU_\ell|
& \geq \frac{|\Phi_{I'-I}(\bm L'')|}{3200\ell_0^2p^{2n_1\cdot \ell}|\Phi_{I'}(\bm L')|^2}p^{n\cdot \ell}.
\end{align*}

Consider $\bm L^\ell$, the system of $p^\ell$ linear forms in $\ell$ variables that define an $\ell$-dimension subspace (see \cref{L}).
By \cref{CS-consistency}, we have $|\Phi_{I'-I}(\bm L^\ell)|\cdot|\Phi_{I'-I}(\bm L'')|=|\Phi_{I'-I}(\bm L')|^2$. Thus the above expression simplifies to
\[|\fU_\ell|\geq\frac{1}{3200\ell_0^2p^{2n_1\cdot \ell}|\Phi_I(\bm L')|^2}\cdot\frac{p^{n\cdot \ell}}{|\Phi_{I'-I}(\bm L^\ell)|}\geq\frac{1}{3200\ell_0^2p^{2n_1\cdot \ell}\|I\|^{2p^{n_1}}}\cdot\frac{p^{n\cdot \ell}}{|\Phi_{I'-I}(\bm L^\ell)|}.\]

Therefore there exists some colored labeled pattern $H=(\bm L,\psi,\phi)\in\cH_\ell$ where $\bm L$ is a full dimension system of linear forms in $\ell$ variables and such that the number of generic $H$-instances in $(f,\fB)$ which are contained in $X$ is at least $1/|\cH_\ell|$ times the right-hand side of the above equation. Note that by equidistribution, \cref{equidistribution}, and the rank bound on $\fB''$,
\begin{equation}
\label{density-upper}
\Lambda_{\bm L}(1_X,\ldots,1_X) \leq \frac{1}{|\Phi_{I'-I}(\bm L)|}+\alpha(\|I'\|)\leq\frac{2}{|\Phi_{I'-I}(\bm L)|}.
\end{equation}. Noting that since $\bm L$ is full dimensional, we have $|\Phi_{I'-I}(\bm L)|=|\Phi_{I'-I}(\bm L^\ell)|$. Thus dividing the two above quantities, we find that the relative density of the above $H$ in $X$ is at least
\[\frac1{6400\ell_0^2p^{2n_1\cdot \ell}\|I\|^{2p^{n_1}}}=\beta_{patch}(p,|\cS|,I,\ell_0).\qedhere\]
\end{proof}


\section{Proof of removal lemmas}
\label{sec-proof-of-removal-lemmas}

As usual, for an atom $a\in A_I$ (defined in \cref{A}), we use $a_{d,k}\in\U_{k+1}^{I_{d,k}}$ to denote the degree $d$, depth $k$ part of $a$. We use the notation $\tilde A_I\subset A_I$ to denote the set $\tilde A_I:=\{a\in A_I:a_{1,0}=0\}$. This is the set of atoms that are regularized by \cref{subatom-selection}. Also define $\tilde I\in\cI_p$ by $\tilde I_{1,0}=0$ and $\tilde I_{d,k}=I_{d,k}$ otherwise.

\begin{defn}
Fix a prime $p$, a finite set $\cS$ equipped with an $\Fpx$-action, and a parameter list $I\in\cI_p$. A \textbf{summary function} with parameters $I$ is a pair $(F,\xi)$ consisting of a function $F\colon (A_I\setminus\tilde A_I)\to 2^{\cS}\setminus\{\emptyset\}$ and a projective function $\xi\colon \F_p\times A_{\tilde I}\to\cS$.
\end{defn}

\begin{defn}
For an $\cS$-colored pattern $H=(\bm L,\psi)$ consisting of $m$ linear forms and a summary function $(F,\xi)$ with parameters $I$, say that \textbf{$(F,\xi)$ partially induces $H$} if there exists a tuple of atoms $\bm a\in A_I^m$ such that the following holds:
\begin{enumerate}[(i)]
    \item $\bm a$ is $\bm L$-consistent, i.e., $\bm a\in\Phi_{I}(\bm L)$;
    \item for each $i\in[m]$ such that $a_i\not\in \tilde A_I$, we have $\psi(i)\in F(a_i)$;
    \item defining $J:=\{i\in[m]:a_i\in\tilde A_i\}$ and $H_J:=((L_i)_{i\in J},\psi|_J,(a_i)_{i\in J}\}$, an $\cS$-colored $I$-labeled pattern, we have $\xi$ canonically induces $H_J$.
\end{enumerate}
\end{defn}

\begin{proof}[Proof of \cref{removal-lemma-proj}]
We are given a parameter $\epsilon>0$ and a possibly infinite set $\cH$ of $\cS$-colored patterns over $\F_p$ of the form $(\ol{\bm L}^\ell,\psi)$ where $\ell$ is some positive integer and $\psi\colon E_\ell\to\cS$ is some map. (See \cref{L} for the definition of $\ol{\bm L}^\ell$ and $E_\ell$.)

We begin with a ``compactness argument'' based on ideas of Alon and Shapira that allows us to reduce to the case when $\cH$ is finite size.

For each parameter list $I\in\cI_p$, we define a finite subset $\cH_I\subseteq\cH$ as follows. Consider the set of all summary functions $(F,\xi)$ with parameters $I$. If there exists any $H\in\cH$ such $(F,\xi)$ partially induces $H$, include one such $H$ in $\cH_I$. Note that $|\cH_I|$ is at most the number of summary functions with parameters $I$, which is finite.

Define the compactness functions $\Psi_{\cH}\colon\Zp\times\Zp\to\Zp$ as follows. Let $\Psi_{\cH}(D,N)$ be the largest positive integer $\ell$ such that there exists a parameter list $I\in\cI_p$ satisfying $\deg I\leq D$ and $\|I\|\leq N$ such that a pattern of the form $(\ol{\bm L}^\ell,\psi)$ exists in $\cH_{I}$.

Now we set several parameters.
Define non-increasing functions $\eta\colon\Zp\times\Zp\to\Zp$ by
\[\eta(D,N):=\frac{1}{40}\paren{\frac{\epsilon}{12N|\cS|}}^{p^{\Psi_\cH(D,N)}},\]
$\beta\colon\Zp\times\Zp\to\Zp$ by
\[\beta(D,N):=\min_{I\in\cI_p:\deg I\leq D,\|I\|\leq N}\beta_{patch}(p,|\cS|,I,\Psi_\cH(D,N)),\] $\theta\colon\Zp\times\Zp\to\Zp$ by
\[\theta(D,N):=\frac{\beta(D,N)}{40}\paren{\frac{\epsilon}{8|\cS|}}^{p^{\Psi_\cH(D,N)}},\]
and $\alpha\colon\Zp\times\Zp\to(0,1)$ by
\[\alpha(D,N):=\frac{\beta(D,N)}{10N^{2p^{\Psi_\cH(D,N)}}}\]
Then define non-decreasing functions $r\colon \Zp\times\Zp\to\Zp$ by
\[r(D,N):=\max\{r_{equi}(p,D,\alpha(D,N)),\max_{I\in\cI_p:\deg I\leq D,\|I\|\leq N}r_{patch}(p,|\cS|,I,\Psi_{\cH}(D,N))(D,N)+p\ceil{\log_pN}\}\]
and $d\colon\Zp\times\Zp\to\Zp$ by
\[d(D,N):=p^{\Psi_\cH(D,N)}.\]
Define parameters
\[\zeta:=\frac{\epsilon}{16|\cS|}\qquad\text{and}\qquad c_0:=\ceil{\log_p(2/\epsilon)}.\]
Then define
\[C_{max}:=C_{reg}(p,|\cS|,c_0,\zeta,\eta,\theta,d,r),\] \[D_{max}:=D_{reg}(p,|\cS|,c_0,\zeta,\eta,\theta,d,r),\]
\[n_{min}:=\max\left\{n_{reg}(p,c_0,\zeta),\max_{I\in\cI_p\colon\deg I\leq D_{max},\|I\|\leq C_{max}}n_{patch}(p,|\cS|,I,\Psi_\cH(D_{max},C_{max}))+\ceil{\log_pC_{max}}\right\}.\]
Finally, define
\[\delta(\epsilon,\cH):=\min\left\{\frac{\beta(D_{max},C_{max})}{40}\paren{\frac{\epsilon}{4C_{max}|\cS|}}^{p^{\Psi_\cH(D_{max},C_{max})}},p^{-n_{min}\cdot \Psi_\cH(D_{max},C_{max})}\right\}\]
and
\[\cH_{\epsilon}:=\bigcup_{I\in\cI_p: \|I\|\leq C_{max},\,\deg I\leq D_{max}}\cH_I.\]
Since the union is over a finite set of $I$, we have $\cH_{\epsilon}$ is finite. We will show that this choice of $\delta$ and $\cH_{\epsilon}$ satisfies the desired conclusion.

Let $V$ be a finite-dimensional $\F_p$-vector space and $f\colon V\to\cS$ be a projective function with $H$-density at most $\delta(\epsilon,\cH)$ for every $H\in\cH_\epsilon$. Our goal is to produce a projective recoloring $g\colon V\to\cS$ that agrees with $f$ on all but an at most $\epsilon$-fraction of $V$ that has no generic $H$-instances for every $H\in\cH$.

First note that if $\dim V<n_{min}$, the theorem is easily follows. This is because for the pattern $H=(\ol{\bm L}^\ell,\psi)$, if there exists an $H$-instance in $f$, then the $H$-density in $f$ is at least $1/|V|^{\ell}$. Thus taking $g=f$ and noticing that we chose $\delta(\epsilon,\cH)\leq p^{-n_{min}\cdot \Psi_{\cH}(C_{max},D_{max})}$, the theorem holds in this case.

Now assume that $\dim V\geq n_{min}$. We apply \cref{subatom-selection} to the functions $\{1_{f^{-1}(c)}\}_{c\in\cS}$ with parameters $p,|\cS|,c_0,\zeta,\eta,\theta,d,r$. This produces a polynomial factor $\fB$ and a refinement $\fB'$ both on $V$ with parameters $I$ and $I'$ and a subatom selection function $s\colon A_I\to A_{I'}$ satisfying several other desirable properties.

As above, define $\tilde A_I\subset A_I$ to be the set $\tilde A_I:=\{a\in A_I:a_{1,0}= 0\}$. We call the atoms $a\in\tilde A_I$ \emph{irregular} and the remaining atoms $a\in A_I\setminus \tilde A_I$ \emph{regular}. Define $\tilde V$ to be the codimension-$I_{1,0}$ subspace of $V$ that is the common zero set of all $I_{1,0}$ linear polynomials defining $\fB$. The irregular atoms of $\fB$ exactly consist of $\tilde V$, i.e., $\fB^{-1}(\tilde A_I)=\tilde V$.

Define $\tilde\fB$ to be the polynomial factor on $\tilde V$ defined by the restrictions of the homogeneous non-classical polynomials that define $\fB$ to $\tilde V$, except for the linear polynomials (which restrict to the zero function). Let $\tilde I\in\cI_p$ be the parameter list of $\tilde\fB$ ($\tilde I_{1,0}=0$ and $\tilde I_{d,k}=I_{d,k}$ otherwise). Also define $\tilde\fB'$ to be the polynomial factor on $\tilde V$ defined by the restrictions of the homogeneous non-classical polynomial that define $\fB'$ to $\tilde V$, except for the linear polynomials that also define $\fB$. Let $\tilde I'\in\cI_p$ be the parameter list of $\tilde\fB'$ ($\tilde I'_{1,0}=I'_{1,0}-I_{1,0}$ and $\tilde I'_{d,k}=I'_{d,k}$ otherwise). Note that by \cref{hyperplane-rank} and our definition of $r$, we have
\[\rank\tilde\fB\geq r_{patch}(p,|\cS|,\tilde I,\Psi_\cH(\deg\fB,\|\fB\|))(\deg\fB,\|\fB\|),\]
\[\rank\tilde\fB'\geq r_{patch}(p,|\cS|,\tilde I,\Psi_\cH(\deg\fB,\|\fB\|))(\deg\fB',\|\fB'\|).\]

We will ``clean up'' the regular atoms by removing low-density colors in a projective manner. We will ``patch'' the irregular atoms by replacing the coloring by a new coloring $\Xi_{\xi,\iota,\tilde\fB}$ for some projective $\xi\colon \F_p\times A_{\tilde I}\to\cS$ and $\iota\colon \tilde V\xrightarrow{\sim}\F_p^{\dim \tilde V}$.

Note that to check that the recoloring $g\colon V\to\cS$ is projective, it suffices to check this fact separately on $\tilde V$ and on $V\setminus\tilde V$.

\textbf{Clean up regular atoms:}
For each $a\in (A_I\setminus\tilde A_I)$, say that a color $c\in\cS$  is \emph{high-density in $a$} if it appears in $\fB'^{-1}(s(a))$ with density at least $\epsilon/(4|\cS|)$. Say that a color is \emph{low-density in $a$} otherwise.

First note that a basic property of subatom selection functions, \cref{sas-properties}(ii), is the following. For $a\in A_I$ and $b\in\Fpx$, we have $b\cdot s(a)=s(b\cdot a)$. Combined with the projectiveness of $f$, this implies that for a color $c\in\cS$ and $b\in\Fpx$, the $c$-density in $\fB'^{-1}(s(a))$ is the same as the $(b\cdot c)$-density in $\fB^{-1}(s(b\cdot a))$. Thus a color $c$ is high-density in $a$ if and only if $b\cdot c$ is high density in $b\cdot a$.

We pick a single high-density color $c_a\in\cS$ for each regular atom $a\in(A_I\setminus\tilde A_I)$ in a projective way, namely such that $b\cdot c_a= c_{b\cdot a}$ for all $a\in(A_I\setminus\tilde A_I)$ and $b\in\Fpx$. By the argument in the above paragraph, this is possible.

Now we define our recoloring of the regular atoms, $g\colon(V\setminus\tilde V)\to\cS$ as follows. For each $a\in(A_I\setminus\tilde A_I)$ and $x\in\fB^{-1}(a)$, we define $g(x):=f(x)$ unless $f(x)$ is low-density in $a$, in which case we define $g(x):=c_a$. Note that $g$ is a projective function. Furthermore we claim that $g$ differs from $f$ on at most an $(\epsilon/2)$-fraction of $V$.

By \cref{subatom-selection}(vi), for all but at most a $\zeta$-fraction of $a\in A_I$, the $c$-density in $\fB^{-1}(a)$ and the $c$-density in $\fB'^{-1}(s(a))$ differ by at most $\zeta$ for all $s\in\cS$. Thus for most atoms $a$, each low-density color appears in $\fB^{-1}(a)$ with density at most $\epsilon/(4|\cS|)+\zeta$, so $f$ and $g$ differ on at most an $(\epsilon/4+\zeta|\cS|)$-fraction of these atoms. The functions $f$ and $g$ may differ completely on the other atoms, but there are at most $\zeta\|\fB\|$ of these. It follows by equidistribution, \cref{equidistribution}, and the rank bound on $\fB$ that each atom of $\fB$ is at most an $(\|\fB\|^{-1}+\alpha(\deg\fB,\|\fB\|))$-fraction of $V$. Putting this all together, we see that $g$ differs from $f$ on at most the following fraction of $V$ \[\zeta\|\fB\|\paren{\frac{1}{\|\fB\|}+\alpha(\deg\fB,\|\fB\|)}+\left(\frac{\epsilon}{4|\cS|}+\zeta\right)|\cS| <\frac{\epsilon}{2}.\]

\textbf{Patch irregular atoms:}
We define $\tilde\cH$ to be the set of all $\cS$-colored $I$-labeled patterns that are defined by a full dimension system of linear forms in at most $\Psi_\cH(\deg I,\|I\|)$ variables and whose relative density in $\tilde\fB'^{-1}(A_{\tilde I}\times\{0\})$ is less than $\beta_{patch}(p,|\cS|,\tilde I,\Psi_{\cH}(\deg I,\|I\|))$.

We apply our patching result, \cref{patching}, to the set $\tilde\cH$. Our definitions are exactly such that $f|_{\tilde V}$ with $\tilde\fB,\tilde\fB'$ demonstrate that \cref{patching}(b) does not hold. In particular, we checked the rank assumptions on $\tilde\fB$ and $\tilde\fB'$ above when they were defined. Furthermore, we assumed that $\dim V\geq n_{min}$, which implies that $\dim\tilde V\geq n_{patch}(p,|\cS|,\tilde I',\Psi_{\cH}(\deg I,\|I\|)$. Finally we defined $\tilde\cH$ to be the set of patterns which appear with very low density in $\tilde\fB'^{-1}(A_{\tilde I}\times\{0\})$. Thus we conclude that \cref{patching}(a) holds. This means that there exists a projective $\xi\colon\F_p\times A_{\tilde I}\to\cS$ that does not canonically induce $\tilde\cH$. In particular, this means that for any fixed isomorphism $\iota\colon \tilde V\xrightarrow{\sim}\F_p^{\dim \tilde V}$, there are no generic $H$-instances in $(\Xi_{\xi,\iota,\tilde\fB})$ for any $H\in\tilde\cH$.

We complete our definition of $g\colon V\to\cS$ by defining $g(x):=\Xi_{\xi,\iota,\tilde\fB}(x)$ for all $x\in\tilde V$. To conclude this portion of the proof, we make sure that $|\tilde V|\leq(\epsilon/2)|V|$. By assumption, \[|\tilde V|/|V|=p^{-I_{1,0}}\leq p^{-c_0}=\epsilon/2,\]as desired.

\textbf{Proof of correctness:} We claim that $g$ has no generic $H$-instances for each $H\in\cH$. Define $F\colon \tilde (A_I\setminus\tilde A_I)\to 2^{\cS}\setminus\{0\}$ to map $a$ to the set of high-density colors in $a$ and recall the projective function $\xi\colon\F_p\times A_{\tilde I}\to\cS$ defined in the ``patch irregular atoms'' section. Now suppose that the desired conclusion does not hold, i.e., that there is a generic $H'$-instance in $g$ for some $H'\in\cH$. By the construction of $g$, this means that $(F,\xi)$ partially induces $H'$. By the definition of the $\cH_I$, this means that there is some $H\in\cH_I\subseteq\cH_{\epsilon}$ so that $(F,\xi)$ also partially induces $H$. We will reach a contradiction by showing that this implies that the $H$-density in $f$ is larger than $\delta(\epsilon,\cH)$.

Say that $H=(\ol{\bm L}^\ell,\psi)$ (note that $\ell\leq\Psi_\cH(\deg I,\|I\|)$). Since $(F,\xi)$ partially induces $H$, this implies that there exists a tuple of atoms $\bm a\in \Phi_I(\ol{\bm L}^\ell)\subseteq A_I^{E_\ell}$ with several desirable properties. Define $J:=\{i\in{E_\ell}:a_i\in\tilde A_I\}$. Recalling that $\tilde A_I$ is just the set of atoms whose linear part is zero, we can conclude that  $J\subseteq E_{\ell}\subset\F_p^\ell$ is the intersection of $E_{\ell}$ with some linear subspace $U\leq \F_p^\ell$ of dimension $\ell'\leq\ell$. This means that the system $(L^{\ell}_i)_{i\in J}$ is equivalent to the system $(L^{\ell'}_i)_{i\in J}$ where now we view $J\subset U\simeq\F_p^{\ell'}$. Define, $H_J:=((L^{\ell'}_i)_{i\in {E_\ell}},\psi|_J,(a_i)_{i\in J})$, an $\cS$-colored $I$-labeled pattern. By \cref{full-dimensional-exists}, we see that $H_j$ is a full dimension pattern.

The first property, that $\bm a$ is $\ol{\bm L}^{\ell}$-consistent, implies that $s(\bm a)$ is also $\ol{\bm L}^{\ell}$-consistent, by \cref{sas-properties}(iii). The second property implies that for each $i\in(E_\ell\setminus J)$, the color $\psi(i)$ is high-density in $a_i$. The third property, together with our definition of $\xi$ implies that in $(f,\fB)$, the relative density of $H_J$ in $\tilde\fB'^{-1}(A_{\tilde I}\times\{0\})$ is at least $\beta_{patch}(p,|\cS|,\tilde I,\Psi_{\cH}(\deg I,\|I\|)\geq \beta(\deg I,\|I\|))$.

Now we put everything together as follows. Write $f^{(i)}$ for $1_{f^{-1}(\psi(i))}$. There is a decomposition $f^{(i)}=f^{(i)}_{str}+f^{(i)}_{sml}+f^{(i)}_{psr}$ given by \cref{subatom-selection} for each $i\in{E_\ell}$. Let $p_1$ be the $H$-density in $f$. We lower bound $p_1$ as follows.
\begin{align*}
p_1 
& =\E_{\bx}\left[\prod_{i\in{E_\ell}} f^{(i)}(L^\ell_i(\bx))\right]\\
& =\E_{\bx}\left[\prod_{i\in J} f^{(i)}(L^\ell_i(\bx))\prod_{i\in{E_\ell}\setminus J}\left( f^{(i)}_{str}(L^\ell_i(\bx))+f^{(i)}_{sml}(L^\ell_i(\bx))+f^{(i)}_{psr}(L^\ell_i(\bx))\right)\right]\\
&\geq \E_{\bx}\left[\prod_{i\in J} f^{(i)}(L^\ell_i(\bx))\prod_{i\in{E_\ell}\setminus J}\left( f^{(i)}_{str}(L^\ell_i(\bx))+f^{(i)}_{sml}(L^\ell_i(\bx))\right)\right]-3^{|E_\ell|}\eta(\deg\fB',\|\fB'\|).
\end{align*}
The inequality follows from \cref{subatom-selection}(iii), the counting lemma (\cref{counting-lemma}), and the fact that the complexity of $H$ is at most $d(\deg I,\|I\|)=p^{\Psi_{\cH}(\deg I,\|I\|)}$.

Write $p_2$ for the expectation in the last line above. By \cref{subatom-selection}(iv), the expression inside the expectation is non-negative so we can restrict the expectation to $\bx$ such that $\ol{\bm L}^{\ell}(\bx)=s(\bm a)$. Thus
\[p_2\geq \E_{\bx}\left[\prod_{i\in J} f^{(i)}(L^\ell_i(\bx))1_{\fB'^{-1}(s(a_i))}(L^\ell_i(\bx))\prod_{i\in{E_\ell}\setminus J}\left( f^{(i)}_{str}(L^\ell_i(\bx))+f^{(i)}_{sml}(L^\ell_i(\bx))\right)1_{\fB'^{-1}(s(a_i))}(L^\ell_i(\bx))\right].\]

Write $p_3$ for the expectation in the last line above. Expanding the product, there are at most $2^{|E_\ell|}$ terms involving $f^{(j)}_{sml}$ for some $j\in {E_\ell}\setminus J$. Each of these is bounded in magnitude by \[\E_{\bx}\left[\abs{f^{(j)}_{sml}(L_j(\bx))}\prod_{i\in E_\ell}1_{\fB'^{-1}(s(a_i))}(L^\ell_i(\bx))\right].\] By applying a change of coordinates, we can assume that $L^{\ell}_j(\bx)=x_1$. Then by the Cauchy-Schwarz inequality, the square of the above expression is bounded by\[\E_{x_1}\left[\abs{f^{(j)}_{sml}(x_1)}^2 1_{\fB'^{-1}(s(a_j))}(x_1)\right]\E_{x_1}\left[1_{\fB'^{-1}(s(a_j))}(x_1)\E_{x_2,\ldots,x_\ell}\left[\prod_{i\neq j}1_{\fB'^{-1}(s(a_i))}(L^\ell_i(\bx))\right]^2\right].\]
The first term is at most $\theta(\deg\fB, \|\fB\|)^2\|1_{\fB'^{-1}(s(a_j))}\|_2^2$ by \cref{subatom-selection}(vi) and the fact that $a_j$ is a regular atom for $j\in{E_\ell}\setminus J$. The second term can be counted by equidistribution applied to the system $\bm L'$ of $2|E_k|-1$ linear forms in $2\ell-1$ variables defined as follows. Set \[L'_{j}(x_1,x_2,\ldots,x_\ell,x'_2,\ldots,x'_\ell):=x_1,\]
and for $i\in{E_\ell}\setminus\{j\}$, define
\[L'_{i,1}(x_1,x_2,\ldots,x_\ell,x'_2,\ldots,x'_\ell):=L^\ell_i(x_1,x_2\ldots,x_\ell),\]
\[L'_{i,2}(x_1,x_2,\ldots,x_\ell,x'_2,\ldots,x'_\ell):=L^\ell_i(x_1,x_2',\ldots,x_\ell').\]
By \cref{CS-consistency}, we know that $\|\fB'\|\cdot|\Phi_{I'}(\bm L')|=|\Phi_{I'}(\ol{\bm L}^\ell)|^2$ (see also \cite[Lemma 5.13]{BFHHL13}.)

Thus by equidistribution (\cref{equidistribution}) and the rank bound on $\fB'$, we have the second term is at most
\[\frac{1}{|\Phi_{I'}(\bm L')|}+\alpha(\deg\fB',\|\fB'\|)=\frac{\|\fB'\|}{|\Phi_{I'}(\ol{\bm L}^\ell)|^2}+\alpha(\deg\fB',\|\fB'\|)\leq\frac{2\|\fB'\|}{|\Phi_{I'}(\ol{\bm L}^\ell)|^2}.\]
Applying equidistribution again we have that the first term is at most
\[\theta(\deg\fB,\|\fB\|)^2\paren{\frac{1}{\|\fB'\|}+\alpha(\deg\fB',\|\fB'\|)}\leq\frac{2\theta(\deg\fB,\|\fB\|)^2}{\|\fB'\|}.\]
Combining these bounds and summing over all terms that contain some $f^{j}_{sml}$, we see that
\[p_3\geq \E_{\bx}\left[\prod_{i\in J}f^{(i)}(L^\ell_i(\bx))1_{\fB'^{-1}(s(a_i))}(L^\ell_i(\bx))\prod_{i\in{E_\ell}\setminus J}f^{(i)}_{str}(L^\ell_i(\bx))1_{\fB'^{-1}(s(a_i))}(L^\ell_i(\bx))\right]-2^{|E_\ell|+1}\frac{\theta(\deg\fB,\|\fB\|)}{|\Phi_{I'}(\ol{\bm L}^\ell)|}.\]

Write $p_4$ for the expectation above. The quantity $f^{(i)}_{str}(L^\ell_i(\bm x))$ is the density of $\psi(i)$ in the atom of $\fB'$ that $L^\ell_i(\bm x)$ lies in. When $\fB(L^\ell_i(\bm x))=s(a_i)$, the fact that $\psi(i)$ is high density in $a_i$ for all $i\in{E_\ell}\setminus J$ implies that\[p_4\geq \paren{\frac{\epsilon}{4|\cS|}}^{|E_\ell|}\E_{\bx}\left[\prod_{i\in J}f^{(i)}(L^\ell_i(\bx))1_{\fB'^{-1}(s(a_i))}(L^\ell_i(\bx))\prod_{i\in{E_\ell}\setminus J}1_{\fB'^{-1}(s(a_i))}(L^\ell_i(\bx))\right].\] 

Write $p_5$ for the expectation above. We write $\bm L_J:=(L^\ell_i)_{i\in J}$. By assumption, we know that 
\begin{equation}
\label{patch-1}
\E_{\bx}\left[\prod_{i\in J}f^{(i)}(L^\ell_i(\bx))1_{\fB'^{-1}(s(a_i))}(L^\ell_i(\bx))\right]\geq\frac{\beta(\deg I,\|I\|)}{|\Phi_{I'}(\bm L_J)|}.
\end{equation}
We want to use this inequality to show that $p_5$ is at least on the order of $\beta(\deg I,\|I\|)/|\Phi_{I'}(\ol{\bm L}^\ell)|$. For simplicity, write $\beta:=\beta(\deg I,\|I\|)$ in the rest of this argument.

By applying a change of coordinates, we can assume that $\bm L_J$ only depends on $x_1,\ldots,x_{\ell'}$ and is independent of $x_{\ell'+1},\ldots,x_{\ell}$. To lower bound $p_5$, we want to show that each tuple $(x_1,\ldots,x_{\ell'})$ that lies in certain atoms extends to a tuple $(x_1,\ldots,x_\ell)$ that still lies in certain atoms in approximately the same number of ways. We do this by a Cauchy-Schwarz argument. Define $\bm L''$ to be the following system of $2|E_{\ell}|-|J|$ linear forms in $2\ell-\ell'$ variables. For $i\in J$, set
\[L''_i(x_1,\ldots,x_{\ell},x'_{\ell'+1},\ldots,x'_{\ell}):=L^{\ell'}_i(x_1,\ldots,x_{\ell'}),\] and for $i\in E_{\ell}\setminus J$, define
\[L''_{i,1}(x_1,\ldots,x_{\ell},x'_{\ell'+1},\ldots,x'_{\ell}):=L^{\ell}_i(x_1,\ldots,x_{\ell}),\]
\[L''_{i,1}(x_1,\ldots,x_{\ell},x'_{\ell'+1},\ldots,x'_{\ell}):=L^{\ell}_i(x_1,\ldots,x_{\ell'},x_{\ell'+1},\ldots,x_{\ell}).\]
By \cref{CS-consistency}, we know that $|\Phi_{I'}(\bm L_J)|\cdot|\Phi_{I'}(\bm L'')|=|\Phi_{I'}(\ol{\bm L}^{\ell})|^2$.

Define $S\subseteq V^{\ell'}$ to be the set of tuples $\bx=(x_1,\ldots,x_{\ell'})$ such that $\fB'(L^\ell_i(\bx))=s(a_i)$ for each $i\in J$. For $\bx\in S$, let $c_{\bx}$ be the number of tuples $\bx'=(x_1,\ldots,x_{\ell})$ such that $\fB'(L^\ell_i(\bx))=s(a_i)$ for each $i\in E_{\ell}$. By applying equidistribution (\cref{equidistribution}) to $\bm L_J$ and $\ol{\bm L}^\ell$ and $\bm L''$, we find that
\begin{equation}
\label{patch-2}
|S|=\sum_{\bx\in S}1\leq(1+\beta/10)\frac{|V|^{\ell'}}{|\Phi_{I'}(\bm L_J)|},
\end{equation}

\begin{equation}
\label{patch-3}
\sum_{\bx\in S}c_{\bx}\geq(1-\beta/10)\frac{|V|^{\ell}}{|\Phi_{I'}(\ol{\bm L}^\ell)|},
\end{equation}

\begin{equation}
\label{patch-4}
\sum_{\bx\in S}c_{\bx}^2\leq(1+\beta/10)\frac{|V|^{2\ell-\ell'}}{|\Phi_{I'}(\bm L'')|}=(1+\beta/10)\frac{|V|^{2\ell-\ell'}|\Phi_{I'}(\bm L_J)|}{|\Phi_{I'}(\ol{\bm L}^\ell)|^2}.
\end{equation}

Define $T\subseteq S\subseteq V^{\ell'}$ to be the set of tuples $\bx=(x_1,\ldots,x_{\ell'})$ such that $\fB'(L^\ell_i(\bx))=s(a_i)$  and $f^{(i)}(L^\ell_i(\bx))=1$ for each $i\in J$. \cref{patch-1} implies that 
\begin{equation}
\label{patch-5}
|T|\geq\beta\frac{|V|^{\ell'}}{|\Phi_{I'}(\bm L_J)|}.
\end{equation}
We express $p_5$ as follows.
\[
p_5
=\frac1{|V|^\ell}\sum_{\bx\in T}c_{\bx}
=\frac1{|V|^\ell}\paren{\sum_{\bx\in S}c_{\bx}-\sum_{\bx\in S\setminus T}c_{\bx}}
\geq(1-\beta/10)\frac1{|\Phi_{I'}(\bm L_J)|}-\frac1{|V|^\ell}\sum_{\bx\in S\setminus t}c_{\bx}.
\]
Then combining \cref{patch-2}, \cref{patch-4}, \cref{patch-5} with the Cauchy-Schwarz inequality gives
\[
\paren{\sum_{\bx\in S\setminus T}c_{\bx}}^2
\leq |S\setminus T|\cdot\sum_{\bx\in S\setminus T}c_{\bx}^2
\leq (1-4\beta/5)\frac{|V|^{2\ell}}{|\Phi_{I'}(\ol{\bm L}^\ell)|^2}.
\]
Taking the square root and combining the above two inequalities gives \[p_5\geq\frac{\beta}{10}\frac{1}{|\Phi_{I'}(\ol{\bm L}^{\ell})|}.\]

Combining all the above inequalities we see that $p_1$, the $H$-density in $f$, is bounded by\[p_1\geq\paren{\paren{\frac{\epsilon}{4|\cS|}}^{|E_\ell|}\frac{\beta(\deg\fB,\|\fB\|)}{10}-2^{|E_{\ell}|+1}\theta(\deg\fB,\|\fB\|)}\frac{1}{{|\Phi_{I'}(\ol{\bm L}^\ell)|}}-3^{|E_\ell|}\eta(\deg\fB',\|\fB'\|)>\delta(\epsilon,\cH).\]
This provides the desired contradiction. Therefore we conclude that the recoloring $g\colon V\to\cS$ has no generic $H$-instances for every $H\in\cH$.
\end{proof}

\begin{proof}[Proof of \cref{removal-lemma}]
Define $\ol\cS:=\cS^{\Fpx}$ with $\Fpx$-action defined by \[b'\cdot(c_b)_{b\in\Fpx}:=(c_{b'b})_{b\in\Fpx}.\]

First we partition \[\cH=\bigsqcup_{c\in\cS\sqcup\{0\}}\cH_c\]as follows. If $L_i\equiv 0$ for some $i\in[m]$, place $H$ in the set $\cH_{\psi(i)}$. Otherwise, place $H$ in $\cH_0$. (Without loss of generality, we can assume that no pattern in $\cH$ has multiple linear forms that are identically equal to 0.)

Now we define sets $\ol\cH_c$ of $\ol\cS$-colored patterns for each $c\in\cS$. Let $H=(\bm L,\psi)\in\cH_0\cup\cH_c$ be an $\cS$-colored pattern over $\F_p$ consisting of $m$ linear forms in $\ell$ variables. We can write $\bm L=(L^\ell_i)_{i\in J}$ for some set $J\subseteq\F_p^\ell$ of size $m$. We convert $H$ to a set of patterns defined by the system $\ol{\bm L}^\ell=(L^\ell_i)_{i\in E_\ell}$ as follows. (See \cref{L} for the definitions of $L^{\ell}_i$ and $E_\ell$.) For each function $\ol\psi\colon E_\ell\to\ol\cS$ that satisfies $\ol\psi(i)_b=\psi(bi)$ whenever $i\in E_\ell\subseteq\F_p^\ell$ and $b\in\Fpx$ are such that $bi\in J\subseteq\F_p^\ell$, we add $(\ol{\bm L}^\ell,\ol\psi)$ to $\ol\cH_c$.

For each $c\in\cS$, we apply \cref{removal-lemma-proj} to $\ol\cH_c$ with parameter $\epsilon$. This produces a finite subset $\ol\cH_{c,\epsilon}\subseteq \ol\cH_c$ and $\delta_c=\delta(\epsilon,\ol\cH_c)>0$ with several desirable properties.

Let $\cH_\epsilon\subseteq\cH$ be the finite subset consisting of all patterns $H$ such that some pattern $\ol H$ corresponding to $H$ lies in $\ol\cH_{c,\epsilon}$ for some $c\in\cS$. Let $\delta=\min_{c\in\cS} \delta_c>0$. We claim that $\cH_\epsilon, \delta$ satisfy the desired conclusion.

Let $V$ be a finite-dimensional $\F_p$-vector space. Let $f\colon V\to\cS$ be a function. Suppose that the $H$-density in $f$ is at most $\delta$ for every $H\in\cH_{\epsilon}$. Define $\ol f\colon V\to\ol\cS$ by \[\ol f(x):=(f(bx))_{b\in\Fpx}.\] Note that $\ol f$ is a projective function. Furthermore, we claim that the $\ol H$-density in $\ol f$ is at most $\delta$ for all $\ol H\in\ol\cH_{f(0)}$. This is true simply because if $\bx=(x_1,\ldots,x_{\ell})\in V^\ell$ is an $\ol H$-instance in $\ol f$, then $\bx$ is also an $H$-instance in $f$ where $\ol H\in\ol\cH_{f(0)}$ is any patterns corresponding to $H\in\cH_0\sqcup\cH_{f(0)}$.

Thus by assumption, there exists a projective recoloring $\ol g\colon V\to\ol\cS$ such that $\ol g$ agrees with $\ol f$ on all but an at most $\epsilon$-fraction of $V$ and $\ol g$ has no generic $\ol H$-instances for every $\ol H\in\ol\cH_{f(0)}$.

Define $g\colon V\to\cS$ by $g(x):=\ol g(x)_1$ for $x\neq 0$ and $g(0)=f(0)$. Note that $g$ agrees with $f$ on all but an at most $\epsilon$-fraction of $V$. Furthermore, note that $g$ has no $H$-instances for $H\in\cH_c$ with $c\neq f(0)$ since $f(0)=g(0)$. Finally, $g$ has no generic $H$-instances for $H\in\cH_0\sqcup\cH_{f(0)}$ since any such generic $H$-instance in $g$ is a generic $\ol H$-instance in $\ol g$ for some $\ol H$ corresponding $H$ in $\ol\cH_{f(0)}$, which we assumed was not the case.
\end{proof}

\section{Proof of property testing results}
\label{sec-proof-of-main-thms}

\begin{proof}[Proof of \cref{main-thm}]
It follows from \cite[Theorem 10]{BGS15} that a linear-invariant property is testable only if it is semi subspace-hereditary.

Now suppose that $\cP$ is a linear-invariant semi subspace-hereditary property. By definition, there exists a subspace-hereditary property $\cQ$ such that
\begin{enumerate}[(i)]
\item every function satisfying $\cP$ also satisfies $\cQ$;
\item for all $\epsilon>0$, there exists $N(\epsilon)$ such that if $f\colon V\to\cS$ satisfies $\cQ$ and is $\epsilon$-far from satisfying $\cP$, then $\dim V<N(\epsilon)$.
\end{enumerate}

We define $\cH$ a (possibly infinite) set of $\cS$-colored patterns. For each $f\colon \F_p^\ell\to\cS$ that does not satisfy $\cQ$, include $H=({\bm L}^\ell, f)$ in $\cH$ ($\bm L^\ell$ is the system of linear forms that defines an $\ell$-dimensional subspace, defined in \cref{L}). Since $\cQ$ is subspace-hereditary, it immediately follows that $\cQ$ consists exactly of the functions with no generic $H$-instances for any $H\in\cH$.

By \cref{removal-lemma}, there exist a finite subset $\cH_{\epsilon}\subseteq\cH$ and some $\delta(\epsilon,\cH)>0$ such that the following holds. If $f\colon V\to\cS$ has $H$-density at most $\delta(\epsilon,\cH)$ for every $H\in\cH_{\epsilon}$, then $f$ is $\epsilon$-close to $\cQ$. Define $\ell(\epsilon)$ to be the largest $\ell$ such that some pattern defined by the system ${\bm L}^\ell$ is present in $\cH_{\epsilon}$.

Now we define the oblivious tester for $\cP$. Given $\epsilon>0$, the tester produces \[d(\epsilon):=\max\left\{N(\epsilon/2),\lceil\log_p(2/\delta(\epsilon/2,\cH))\rceil+\ell(\epsilon/2)\right\}.\] Given a function $f\colon V\to\cS$ our tester receives oracle access to $f|_U$ where 
\begin{enumerate}[(i)]
\item if $\dim V\geq d(\epsilon)$, then $U$ is a uniform random affine subspace of dimension $d(\epsilon)$;
\item else, $U=V$.
\end{enumerate}
Our tester works as follows. If $\dim U<d(\epsilon)$ the tester accepts if $f|_U\in\cP$. If $\dim U\geq d(\epsilon)$ the tester accepts in $f|_U\in\cQ$.

Suppose $f\in\cP$. If $\dim U< d(\epsilon)$, then $U=V$, so $f|_U=f\in\cP$. Thus the tester always accepts in this case. In the other case, note that since $f\in\cP$, it follows that $f\in\cQ$, and since $\cQ$ is subspace-hereditary, $f|_U\in\cQ$. Thus the tester also always accepts in this case.

Now suppose that $f$ is $\epsilon$-far from $\cP$. By the definition of $\cQ$ we know that either $\dim V< N(\epsilon/2)\leq d(\epsilon)$ or $f$ is $\epsilon/2$-far from $\cQ$. Consider the action of the tester. If $\dim U<d(\epsilon)$, then $U=V$ so $f|_U=f\not\in\cP$. Thus the tester always rejects in this case. In the other case, note that since $f$ is $\epsilon/2$-far from $\cQ$, by assumption there is some $H\in\cH_{\epsilon/2}$ such that $f$ has $H$-density more than $\delta(\epsilon/2,\cH)$. Let $H=({\bm L}^\ell,\psi)$ for some $\ell\leq \ell(\epsilon/2)$. We claim the fact that the $H$-density in $f$ is large implies that there is at least a $\delta(\epsilon/2,\cH)/2$-fraction of $\ell$-dimensional subspaces that $f$ colors by $\psi$. (Note that this does not immediately follow since the $H$-density includes the contribution of $H$-instances that are not generic.) We can compute that the probability a uniform random ${\bm L}^\ell$-instance in $V$ is not generic is at most $p^{\ell-\dim V}$. It follows that the fraction of $\ell$-dimensional subspaces that $f$ colors by $\psi$ is at least \[\delta(\epsilon/2,\cH)-p^{\ell-\dim V}\geq\delta(\epsilon/2,\cH)-p^{-\lceil\log_p(2/\delta(\epsilon/2,\cH))\rceil}\geq\delta(\epsilon/2,\cH)/2.\] We conclude that in this case the tester rejects with probability at least $\delta(\epsilon/2,\cH)/2$, as desired.
\end{proof}

\begin{proof}[Proof of \cref{main-thm-PO}]
Suppose $\cP$ is a linear-invariant property that is PO-testable. By definition, there exists some $d$, independent of $\epsilon$ and $\dim V$, such that to test $f\colon V\to\cS$, the tester receives $f|_U$ where $U$ is
\begin{enumerate}[(i)]
\item if $\dim V\geq d$, then $U$ is a uniform random linear subspace of dimension $d$;
\item else, $U=V$.
\end{enumerate}

We define $\cH$ to be the set of patterns of the form $(\bm L^d,\psi)$ where $\psi\colon\F_p^d\to\cS$ is a restriction that the tester rejects on  ($\bm L^d$ is the pattern that defines a $d$-dimensional subspace, defined in \cref{L}). We claim that for every $f\colon V\to\cS$ with $\dim V\geq d$, it holds that $f\in\cP$ if and only if $f$ has no generic $\cH$-instances. This claim suffices to prove that $\cH$ is subspace-hereditary and locally characterized. 

Suppose $f\colon V\to\cS$ satisfies $\cP$ and $\dim V\geq d$. By the definition of PO-testable, the tester must accept $f$ with probability 1. Thus the tester must accept $f|_U$ for every $U\leq V$ of dimension $d$. This implies that $f$ has no generic $\cH$-instances. Now suppose that $f\colon V\to\cS$ does not satisfy $\cP$ and $\dim V\geq d$ holds. By definition, the tester must accept $f$ with positive probability. Thus there must be some $U\leq V$ of dimension $d$ such that $f|_U$ rejects. This is equivalent to the fact that $f$ contains a generic $H$-instance for some $H\in\cH$, proving the desired result.

Now we show that every linear-invariant subspace-hereditary locally characterized property is testable. Suppose $\cP$ is such a property. It follows that there is some $d$ and (finite) $\cH$ consisting of patterns of the form $(\bm L^d,\psi)$ such that for $f\colon V\to\cS$ with $\dim V\geq d$, we have $f$ satisfies $\cP$ if and only if $f$ has no generic $\cH$-instances.

The PO-tester for $\cP$ proceeds in the obvious way. The tester is given oracle access to $f|_U$ where $U$ is
\begin{enumerate}[(i)]
\item if $\dim V\geq d$, then $U$ is a uniform random linear subspace of dimension $d$;
\item else, $U=V$.
\end{enumerate}
The tester accepts if and only if $f|_U\in\cP$.

Suppose $f\colon V\to\cS$ satisfies $\cP$. If $\dim V<d$, then $f|_U=f\in\cP$, so the tester accepts $f$. If $\dim V\geq d$, then by the fact that $\cP$ is subspace-hereditary and locally characterized, it follows that $f|_U\in\cP$ for all $d$-dimensional $U\leq V$. Thus the tester accepts $f$ in this case as well.

Now suppose that $f\colon V\to\cS$ is $\epsilon$-far from $\cP$. If $\dim V<d$, then $f|_U=f\not\in\cP$, so the tester rejects $f$. If $\dim V\geq d$, by \cref{removal-lemma}, there must be some $H=(\bm L^d,\psi)\in\cH$ such that the $H$-density in $f$ is more than $\delta(\epsilon,\cH)$. We claim that this implies that there is a large fraction of $d$-dimensional subspaces that $f$ colors by some $H\in\cH$. (Note that this does not immediately follow since the $H$-density includes the contribution of $H$-instances that are not generic.) We can compute that at most a $p^{d-\dim V}$-fraction of $H$-instances are non-generic. Thus at least the fraction of $d$-dimensional subspaces that $f$ colors by $\psi$ is at least \[\delta(\epsilon,\cH)-p^{d-\dim V}.\] This parameter is negative for small values of $\dim V$, so we can also use the fact that since $f$ does not satisfy $\cP$, there is at least one $d$-dimensional subspaces that is colored by $\cH$. Thus the rejection probability of this tester is at least\[\max\left\{\delta(\epsilon)-p^{d-\dim V}, p^{-d\cdot\dim V}\right\}.\] Note that for $\dim V\geq d$, this parameter is uniformly bounded away from 0, as desired.
\end{proof}


\begin{thebibliography}{99}

\bibitem{ADLRY94}
Noga Alon, Richard~A. Duke, Hanno Lefmann, Vojt\v{e}ch R\"{o}dl, and Raphael
  Yuster, \emph{The algorithmic aspects of the regularity lemma}, J. Algorithms
  \textbf{16} (1994), 80--109.

\bibitem{AFKS00}
Noga Alon, Eldar Fischer, Michael Krivelevich, and Mario Szegedy,
  \emph{Efficient testing of large graphs}, Combinatorica \textbf{20} (2000),
  451--476.

\bibitem{AKKLR05}
Noga Alon, Tali Kaufman, Michael Krivelevich, Simon Litsyn, and Dana Ron,
  \emph{Testing {R}eed-{M}uller codes}, IEEE Trans. Inform. Theory \textbf{51}
  (2005), 4032--4039.

\bibitem{AS08}
Noga Alon and Asaf Shapira, \emph{A characterization of the (natural) graph
  properties testable with one-sided error}, SIAM J. Comput. \textbf{37}
  (2008), 1703--1727.

\bibitem{BTZ10}
Vitaly Bergelson, Terence Tao, and Tamar Ziegler, \emph{An inverse theorem for
  the uniformity seminorms associated with the action of {$\mathbb
  F^\infty_p$}}, Geom. Funct. Anal. \textbf{19} (2010), 1539--1596.

\bibitem{BCSX11}
Arnab Bhattacharyya, Victor Chen, Madhu Sudan, and Ning Xie, \emph{Testing
  linear-invariant non-linear properties}, Theory Comput. \textbf{7} (2011),
  75--99.

\bibitem{BFHHL13}
Arnab Bhattacharyya, Eldar Fischer, Hamed Hatami, Pooya Hatami, and Shachar
  Lovett, \emph{Every locally characterized affine-invariant property is
  testable}, S{TOC}'13---{P}roceedings of the 2013 {ACM} {S}ymposium on
  {T}heory of {C}omputing, ACM, New York, 2013, pp.~429--435.

\bibitem{BFL12}
Arnab Bhattacharyya, Eldar Fischer, and Shachar Lovett, \emph{Testing low
  complexity affine-invariant properties}, Proceedings of the {T}wenty-{F}ourth
  {A}nnual {ACM}-{SIAM} {S}ymposium on {D}iscrete {A}lgorithms, SIAM,
  Philadelphia, PA, 2012, pp.~1337--1355.

\bibitem{BGS15}
Arnab Bhattacharyya, Elena Grigorescu, and Asaf Shapira, \emph{A unified
  framework for testing linear-invariant properties}, Random Structures
  Algorithms \textbf{46} (2015), 232--260.

\bibitem{BLR93}
Manuel Blum, Michael Luby, and Ronitt Rubinfeld, \emph{Self-testing/correcting
  with applications to numerical problems}, Proceedings of the 22nd {A}nnual
  {ACM} {S}ymposium on {T}heory of {C}omputing ({B}altimore, {MD}, 1990),
  vol.~47, 1993, pp.~549--595.

\bibitem{CF13}
David Conlon and Jacob Fox, \emph{Graph removal lemmas}, Surveys in
  combinatorics 2013, London Math. Soc. Lecture Note Ser., vol. 409, Cambridge
  Univ. Press, Cambridge, 2013, pp.~1--49.

\bibitem{FTZ19}
Jacob Fox, Jonathan Tidor, and Yufei Zhao, \emph{Induced arithmetic removal:
  complexity 1 patterns over finite fields}, arXiv:1911.03427.

\bibitem{Fur95}
Zolt\'{a}n F\"{u}redi, \emph{Extremal hypergraphs and combinatorial geometry},
  Proceedings of the {I}nternational {C}ongress of {M}athematicians, {V}ol. 1,
  2 ({Z}\"{u}rich, 1994), Birkh\"{a}user, Basel, 1995, pp.~1343--1352.

\bibitem{GGR98}
Oded Goldreich, Shafi Goldwasser, and Dana Ron, \emph{Property testing and its
  connection to learning and approximation}, J. ACM \textbf{45} (1998),
  653--750.
  
\bibitem{GR11}
Oded Goldreich and Dana Ron, \emph{On proximity-oblivious testing}, SIAM J. Comput. \textbf{40} (2011), 534--566.

\bibitem{Gow01}
W.~Timothy Gowers, \emph{A new proof of {S}zemer\'{e}di's theorem}, Geom.
  Funct. Anal. \textbf{11} (2001), 465--588.

\bibitem{Gow07}
W.~Timothy Gowers, \emph{Hypergraph regularity and the multidimensional
  {S}zemer\'{e}di theorem}, Ann. of Math. (2) \textbf{166} (2007), 897--946.

\bibitem{GW11}
W.~Timothy Gowers and Julia Wolf, \emph{Linear forms and higher-degree
  uniformity for functions on {$\mathbb F^n_p$}}, Geom. Funct. Anal.
  \textbf{21} (2011), 36--69.

\bibitem{Gre05}
Ben Green, \emph{A {S}zemer\'{e}di-type regularity lemma in abelian groups,
  with applications}, Geom. Funct. Anal. \textbf{15} (2005), 340--376.

\bibitem{GS15}
Ben Green and Tom Sanders, \emph{Fourier uniformity on subspaces},
  arXiv:1510.08739.

\bibitem{GT08}
Ben Green and Terence Tao, \emph{The primes contain arbitrarily long arithmetic
  progressions}, Ann. of Math. (2) \textbf{167} (2008), 481--547.

\bibitem{GT10}
Ben Green and Terence Tao, \emph{Linear equations in primes}, Ann. of Math. (2)
  \textbf{171} (2010), 1753--1850.

\bibitem{GTZ12}
Ben Green, Terence Tao, and Tamar Ziegler, \emph{An inverse theorem for the
  {G}owers {$U^{s+1}[N]$}-norm}, Ann. of Math. (2) \textbf{176} (2012),
  1231--1372.

\bibitem{HHL16}
Hamed Hatami, Pooya Hatami, and Shachar Lovett, \emph{General systems of linear
  forms: equidistribution and true complexity}, Adv. Math. \textbf{292} (2016),
  446--477.

\bibitem{HHL19}
Hamed Hatami, Pooya Hatami, and Shachar Lovett, \emph{Higher-order fourier
  analysis and applications}, Foundations and Trends{\textregistered} in
  Theoretical Computer Science \textbf{13} (2019), 247--448.

\bibitem{KS08}
Tali Kaufman and Madhu Sudan, \emph{Algebraic property testing: the role of
  invariance}, S{TOC}'08, ACM, New York, 2008, pp.~403--412.

\bibitem{KSV12}
Daniel Kr\'{a}\v{l}, Oriol Serra, and Llu\'{\i}s Vena, \emph{A removal lemma
  for systems of linear equations over finite fields}, Israel J. Math.
  \textbf{187} (2012), 193--207.

\bibitem{RNSSK05}
Vojt\v{e}ch R\"{o}dl, Brendan Nagle, Jozef Skokan, Mathias Schacht, and
  Yoshiharu Kohayakawa, \emph{The hypergraph regularity method and its
  applications}, Proc. Natl. Acad. Sci. USA \textbf{102} (2005), 8109--8113.

\bibitem{RS07}
Vojt\v{e}ch R\"{o}dl and Mathias Schacht, \emph{Property testing in hypergraphs
  and the removal lemma [extended abstract]}, S{TOC}'07---{P}roceedings of the
  39th {A}nnual {ACM} {S}ymposium on {T}heory of {C}omputing, ACM, New York,
  2007, pp.~488--495.

\bibitem{RS78}
Imre~Z. Ruzsa and Endre Szemer\'{e}di, \emph{Triple systems with no six points
  carrying three triangles}, Combinatorics ({P}roc. {F}ifth {H}ungarian
  {C}olloq., {K}eszthely, 1976), {V}ol. {II}, Colloq. Math. Soc. J\'{a}nos
  Bolyai, vol.~18, North-Holland, Amsterdam-New York, 1978, pp.~939--945.

\bibitem{Sha10}
Asaf Shapira, \emph{A proof of {G}reen's conjecture regarding the removal
  properties of sets of linear equations}, J. Lond. Math. Soc. (2) \textbf{81}
  (2010), 355--373.

\bibitem{TZ10}
Terence Tao and Tamar Ziegler, \emph{The inverse conjecture for the {G}owers
  norm over finite fields via the correspondence principle}, Anal. PDE
  \textbf{3} (2010), 1--20.

\bibitem{TZ12}
Terence Tao and Tamar Ziegler, \emph{The inverse conjecture for the {G}owers
  norm over finite fields in low characteristic}, Ann. Comb. \textbf{16}
  (2012), 121--188.
  
\end{thebibliography}
\end{document}